\documentclass[11pt,letterpaper]{amsart}
\usepackage{verbatim}
\usepackage{xspace}
\usepackage{dsfont}

\makeindex

\usepackage{syntonly}
\usepackage{amsmath,amsthm,amssymb,amscd}
\usepackage[all]{xy}
\usepackage{hyperref}

\newsavebox{\cm}
\sbox{\cm}{ \begin{picture}(12,8)
              \put(6,4){\oval(8,8)[b]}
              \put(6,4){\oval(8,8)[r]}
              \put(6,8){\vector(-1,0){2}}
              \end{picture}  }

\newcommand{\inc}{\ensuremath{\hookrightarrow}}
\newcommand{\f}{\to}
\newcommand{\ff}{\ensuremath{\longmapsto}}
\newcommand{\ov}{\overline}

\newcommand{\s}{\mathcal{S}}

\newcommand{\g}{\mathcal{G}}

\newcommand{\Z}{\ensuremath{\mathbb{Z}}}

\newcommand{\R}{\ensuremath{\mathbb{R}}}
\newcommand{\C}{\ensuremath{\mathbb{C}}}

\renewcommand{\r}[1]{\ensuremath{\big|_{#1}}}

\newcommand{\bc}{\begin{center}}
\newcommand{\ec}{\end{center}}
\newcommand{\be}{\begin{enumerate}}
\newcommand{\ee}{\end{enumerate}}
\newcommand{\bi}{\begin{itemize}}
\newcommand{\ei}{\end{itemize}}
\newcommand{\bd}{\begin{description}}
\newcommand{\ed}{\end{description}}
\newcommand{\beq}{\begin{equation}}
\newcommand{\eeq}{\end{equation}}
\newcommand{\beqa}{\begin{eqnarray}}
\newcommand{\eeqa}{\end{eqnarray}}
\newcommand{\bfr}{\begin{flushright}}
\newcommand{\efr}{\end{flushright}}
\newcommand{\bfl}{\begin{flushleft}}
\newcommand{\efl}{\end{flushleft}}
\newcommand{\bp}{\begin{picture}}
\newcommand{\ep}{\end{picture}}

\theoremstyle{plain}
\newtheorem{thm}{Theorem}[section]
\newtheorem{prop}[thm]{Proposition}
\newtheorem{lem}[thm]{Lemma}
\newtheorem{cor}[thm]{Corollary}          

\theoremstyle{definition}
\newtheorem{ex}[thm]{Example}
\newtheorem{rk}[thm]{Remark}
\newtheorem{df}[thm]{Definition}

\newcommand{\al}{\ensuremath{\alpha}} 
 
\newcommand{\ga}{\ensuremath{\gamma}}
\newcommand{\la}{\ensuremath{\lambda}}
\newcommand{\sig}{\ensuremath{\sigma}}

\newcommand{\env}[1]{ \ensuremath{ {#1}^{\mathsf{e}} } }
\newcommand{\hal}{\ensuremath{\env{\al}}}

\newcommand{\hex}{\ensuremath{\env{X}}}

\newcommand{\ve}{\vspace*{.01cm}}

\newcommand{\css}{$C^*$-algebras\xspace}

\newcommand{\ea}{enveloping action\xspace}
\newcommand{\eas}{enveloping actions\xspace}

\newcommand{\pa}{partial action\xspace}
\newcommand{\pas}{partial actions\xspace}

\DeclareMathOperator{\gr}{Gr}

\AtEndDocument{\multipasswarning}
\newcommand{\multipasswarning}{%
 \clearpage
 \typeout{%
 ***************************************************************************}
 \typeout{%
 Note: This document needs to run through LaTeX three times, instead of }
 \typeout{%
 the usual two, to resolve indirect cross-references}
 \typeout{%
 ***************************************************************************}
 }

\begin{document}

\title[Amenable Partial Actions]
      {Amenable  partial actions}
\author[Massoud Amini]{Massoud Amini}

\address{Department of Mathematics\\ Faculty of Mathematical Sciences\\ Tarbiat Modares University\\ Tehran 14115-134, Iran}
\email{mamini@modares.ac.ir, mamini@ipm.ir}

\address{Current Address: STEM Complex, 150 Louis-Pasteur Pvt,
	Ottawa, ON, Canada K1N 6N5}


\thanks{}

\begin{abstract}
We introduce and study various notions of amenability for continuous (Borel) partial actions of locally compact (Borel) groups on  topological (standard Borel) spaces.  
We also study amenability of partial representations of a locally compact group in a Banach space and show that a partial action on a measure space is amenable iff the corresponding Koopman partial representation on the corresponding $L^2$-space is amenable. We introduce the notion of induced partial representation from a closed subgroup and explore perseverance of amenability type properties under induction.      
\end{abstract}

\maketitle

\tableofcontents

\section{Introduction}\label{sec:intro}

Classically a dynamics on a space (a manifold, topological space, or measure space) describes the time evolution of the points of the space. In modern terms, this could be described by a (smooth, continuous, or Borel) action of the additive group of reals (the so called one parameter group of transformations). In discrete time, we may describe the action of the subgroup of integers by a single invertible transformation.     

In practice, this picture is too restrictive to describe the flow of differential equations. The initial value problem for a given Lipschitz vector field on an open subset of the Euclidean space admits a unique parametric solution (for any initial point in that open set), defined on some open subset about zero (in the parameter space, say time). Extending this to the maximal interval, we get a one parameter family of partial diffeomorphisms of the original open set in the Euclidean space. This is no longer a one parameter group of diffeomorphisms, as the flow at time $t+s$ may be a proper extension of the combination of flows at times $t$ and $s$. In technical terms, this is a {\it partial action} of the additive group of reals. 
   
\par Partial actions (at least in topological case) are now quite well studied (see \cite{exl} and references therein) and have natural extensions to  partial actions on \css  with natural ties to notions such as Fell bundles (see, \cite{excirc}, 
\cite{mc} and \cite{extwist}). 

\par Group actions are related to the notion of amenability for groups via the celebrated Day’s fixed point theorem: A discrete group is amenable if and only if {\it any} of its actions  by continuous affine transformations on a compact convex subset of a (separable) locally convex topological vector space has a fixed point \cite{d}. Relaxing the notion of amenability, one could define amenability of an {\it specific} action of a group. This notion is known (at least for discrete groups) to be related to analogs of all sort of properties characterizing amenability of groups, such as fixed point properties \cite{dw}, exactness \cite{o}, invariant means and cohomology \cite{bnnw}, and approximate type invariance properties (such as F{\o}lner and Reiter conditions) \cite{kl}. Amenability of actions is also studied for locally compact groups (see for instance, \cite{aeg} and \cite{AR00}). Finally, each action on a measure space induces a representation of the group on the corresponding $L^2$-space (the so called Koopman representation) and dynamical properties of the action  have counterparts for representations (see for instance \cite{kl}). In particular, there is also a notion of amenability for representations due to Bekka \cite{b}. An action also gives a groupoid structure on the Cartesian product of the group by the ambient space, whose amenability is known to be related to the amenability of the action \cite{aeg}.             

In this paper we study various notions of amenability for \pas. The motivation of  paper is twofold. First, while many equivalent conditions are known for amenability of actions of discrete groups, some of these equivalences are not known in the locally compact case. Here we prove such equivalences in the most general case (not only for actions, but also) for \pas. Second, the interplay between \pas and partial representations seems to be not well explored, and the notion of amenability provides a crucial aspect of such a relationship.     

\par The structure of the paper is as follows. In Section \ref{sec:enve}  we review the notion of enveloping actions of \pas. The contents of this section are not new and we provides proofs only for the sake of completeness. In section \ref{sec:am}, we define and study various amenability notions for \pas of topological (as well as Borel) groups. We relate the amenability of a \pa in the sense of Zimmer to amenability of its enveloping action. We also relate the amenability of a \pa in the sense of Greenleaf to approximate type invariance conditions, such as F{\o}lner and Reiter conditions. 
In Section \ref{sec:pr}, we introduce and study amenability of partial representations and show that a continuous (Borel) partial action is amenable in the sense of Greenleaf iff the corresponding Koopman partial representation is amenable in the sense of Bekka. Also we show that all continuous partial actions (representations) of a locally compact amenable group on a standard Borel space (in a tracial Banach space) are amenable in the sense of Greenleaf (Bekka). In Section  \ref{sec:ind}, we introduce the notions of induced partial representations (from a closed subgroup) and weak containment for partial representations and study perseverance of amenability type properties under induction and weak containment. 

Though many of the constructions in this paper resemble the global case, there are many serious technicalities to overcome (as one could say by following the commonly involved proofs presented here) before one could use similar arguments. Some notions (such as induced partial representations) are completely new and are discussed for the first time here. 

\section{Enveloping actions}\label{sec:enve} 

\par In this section we review the existence, uniqueness and properties \eas for \pas of topological groups on (not necessarily Hausdorff)
topological spaces, illustrated in some concrete examples. In particular, we observe that  $(\al,X)$ is a \pa, where $X$ is a Hausdorff space, and if $(\beta,Y)$ 
is its \ea, then $Y$ is a Hausdorff space if and only if the graph of $\al$ 
is closed. For the rest of this paper, $G$ is a topological group with identity element $e$.
\ve

Following \cite{fav5}, let us define partial actions of topological groups on topological spaces, and review their properties. 

\begin{df}\label{df:pa}
A {\it partial action} \al\  of a topological group $G$ on a topological 
space $X$ is a pair $\big(\{X_s\}_{s\in G},\{\al_s\}_{s\in G}\big)$, such 
that,
\be
 \item $X_t$ is open in $X$, and $\al_t:X_{t^{-1}}\f X_t$ is a homeomorphism, 
       
 \item The subset $X\rtimes_{\al}G:=\big\{(t,x): t\in G, 
       x\in X_{t^{-1}}\big\}$ is open in $G\times X$, and the map$: X\rtimes_{\al}G\f X$;\ $(t,x)\ff\al_t(x),$ 
       is continuous,
 \item $X_e=X$, and $\al_{st}$ is an 
       extension of $\al_s\al_t$,
\ee
\noindent for all $s,t\in G$.

  Condition (3) above is equivalent to the following set of conditions \cite[Lemma 1.2]{qr},
  \be
  \item[(3-1)] $\al_e=id_X$ and $\al_{t^{-1}}=\al_t^{-1}$, 
  \item[(3-2)] $\al_t(X_{t^{-1}}\cap X_s)=X_t\cap X_{ts}$, 
  \item[(3-3)] $\al_s\al_t :X_{t^{-1}}\cap X_{t^{-1}s^{-1}}\f X_s\cap X_{st}$ 
  is a bijection, and $\al_s\al_t(x)=\al_{st}(x)$,
 for  $s,t\in G$ and $x\in 
  X_{t^{-1}}\cap X_{t^{-1}s^{-1}}$.
 \ee  
If $\al =\big(\{X_t\}_{t\in G},\{\al_t\}_{t\in G}\big)$ and $\beta
=\big(\{Y_t\}_{t\in G},\{\beta_t\}_{t\in G}\big)$ are 
\pas of $G$ on $X$ and $Y$, we say that a continuous function 
$\phi: X\f Y$ is a morphism $\phi :\al\f\beta$ if $\phi(X_t)\subseteq Y_t$, 
 and for each  $t\in G$, the following diagram commutes,      
\[ \xymatrix
{X_{t^{-1}}\ar@{->}[r]^-{\phi}\ar[d]_-{\al_t}&Y_{t^{-1}}\ar[d]^-{\beta_t}\\
X_t\ar@{->}[r]_-{\phi}&Y_t} \] 
\end{df}

A partial action of a Borel group $G$ on a (standard) Borel 
space $X$ is defined as above, where the terms open, homeomorphism, and continuous are replaced by the terms Borel measurable, Borel isomorphism, and Borel, respectively.

\begin{ex}\label{rest}
$(i)$ Let $\beta:G\times Y\f Y$ be a continuous global action
and let $X$ be an open subset of $Y$. Consider $\al=\beta\r{X}$, the 
restriction of $\beta$ to $X$, that is: $X_t=X\cap\beta_t(X)$, and 
$\al_t:X_{t^{-1}}\f X_t$ such that $\al_t(x)=\beta_t(x)$, $t\in G$, 
$x\in X_{t^{-1}}$. Then $\al$ is a \pa on $X$ and any \pa arises in this way \cite{fav5}. In particular, $\beta$ may be identified with the \pa $\beta\r{Y}$.

$(ii)$ The flow of a differentiable vector field is a partial action. More precisely, 
consider a smooth vector field $\mathbf{v}:X\f TX$ on a manifold $X$, and for 
$x\in X$ let $\ga_x$ be the corresponding integral curve through $x$, defined on its maximal interval $(a_x,b_x)$. Let us define, 
for $t\in\R$: $X_{-t}=\big\{ x\in X:\, t\in (a_x,b_x)\big\}$, 
$\al_t:X_{-t}\f X_t$ such that $\al_t(x)=\ga_x(t)$, and 
$\al =\big(\{X_t\}_{t\in\R},\{\al_t\}_{t\in\R}\big)$. Then $\al$ is a partial 
action of $\R$ on $X$.
\end{ex}

It is observed by F. Abadie  that \pas on compact spaces restrict to a global action on an open subgroup \cite[Proposition 1.1]{fav5}. Here we sketch the proof, as the idea of the proof is used in the reminder of this section.   

\begin{lem}\label{cont}
Let $\al$ be a \pa\ of $G$ on a compact space $X$. Then there exists
an open subgroup $H$ of $G$ such that $\al$ restricted to $H$ is a  
global action. 
\end{lem}
\begin{proof}
Let $G_x=\{ t\in G:\, x\in X_{t^{-1}}\}$, and $G_0=\bigcap_{x\in X}A_x$. Then $e\in G_0$ and $st\in G_0$ whenever $s$, $t\in G_0$. Also, given $x\in X$ there exist open neighborhood 
$U_x$ of $x$ and symmetric neighborhood $V_e^x$ of $e$ such that 
$V_e^x\times U_x\subseteq X\rtimes_{\al}G$. By compactness of $X$, 
there exist $x_1,\ldots ,x_n\in X$ with $X=\bigcup_{j=1}^n U_{x_j}$. 
For the symmetric neighborhood $V=\bigcap_{j=1}^nV_{x_j}\subseteq G_0$ of $e$, $H:=\bigcup_{n=1}^\infty V^n$ is an 
open subgroup of $G$ contained in $G_0$, and $\al$ restricts to a global action on $H$.
\end{proof}
\par In particular, if $G$ is connected, each \pa on $G$ is a global action (since the unique open subgroup of a connected group is the group itself).  The next result is \cite[Theorem 1.1]{fav5}, whose proof is sketched here.

\begin{lem}\label{env}
Let $\al$ be a \pa of $G$ on $X$. Then there exists a pair  
$(\iota,\env{\al})$ (unique, up to isomorphism) such that $\env{\al}$ is a continuous action of $G$ on 
a topological space $\env{X}$, and $\iota :\al\f\env{\al}$ is a morphism, 
such that for any continuous 
action $\beta$ of $G$ and any morphism $\psi:\al\f\beta$, there exists a unique morphism $\env{\psi}:\env{\al}\f\beta$ 
making the  diagram   
 \[ \xymatrix{
       \al\ar[rr]^{\iota}
\ar[dr]_{\psi}&\ar @{}[d]
|{\circlearrowleft}&\hal\ar@{-->}[dl]^{\env{\psi}}\\ &\beta&}\]
\par commutative. Moreover,
\be
 \item $\iota(X)$ is open in $\env{X}$.
 \item $\iota:X\f\iota(X)$ is a homeomorphism.
 \item $\env{X}$ is the $\env{\al}$--orbit of $\iota(X)$.
\ee
\end{lem}
\begin{proof}
The continuous action $\ga:G\times (X\times G)\f X\times G$; \ $\ga_s(x,t)=(x,st)$, observes the equivalence relation, $$(x,r)\sim (y,s)\iff x\in X_{r^{-1}s}, \ \al_{s^{-1}r}(x)=y,$$ 
inducing a continuous action $\env{\al}$ on $\env{X}=(X\times G)/\sim$, defined by, $$\env{\al}_s([x,t])=[x,st];\ \ (s,t\in G, x\in X).$$ Next, the quotient map $q:X\times G\f\env{X}$ induces a continuous injection $\iota:X\inc\env{X}$, which  is also open, since $$q^{-1}\big(\iota(U)\big)
=\{ (x,t):\, (x,t)\sim (y,e),\text{ some }y\in U \} 
=\{ (x,t):\, \al_t(x)\in U\}$$  is open in $X\rtimes_{\al}G,$ for $U\subseteq X$ open. Also, $\env{X}$ is the $\env{\al}$--orbit of $\iota(X)$, as  $q(x,t)=\env{\al}_t\big(\iota(x)\big)$.
\par To see that $\alpha^e$ is continuous and $\iota:\alpha\to\alpha^e$ is a morphism, note that $q$ is an open map, since, 
$$q(U\times V)
  =\bigcup_{t\in V}q(U\times\gamma_t(\{e\}))
  =\bigcup_{t\in V}\alpha_t^e(\iota (U)),$$ 
for $V\subseteq G$ and $U\subseteq X$ open. Next, $(id\times q)^{-1}((\alpha^e)^{-1}(W))=\gamma^{-1}(q^{-1}(W))$ is  open in $G\times(X\times G)$, for $W\subseteq X^e$ open, thus,  $(\alpha^e)^{-1}(W)$ is open in $G\times X^e$, because $id\times q$ is an open surjection. Finally, for $x\in X_{t^{-1}}$, 
\[\iota\big(\al_t(x)\big)=q\big(\al_t(x),e\big)=q(x,t)
                         =q\big(\ga_t(x,e)\big)
                         =\env{\al}_t\big(q(x,e)\big)
                         =\env{\al}_t\big(\iota(x)\big).\]
\par To prove the universal property, given $\beta:G\times Y\f Y$ and $\psi:X\f Y$, the map 
$\psi':X\times G\f Y$;\ $\psi'(x,t)=\beta_t\big(\psi(x)\big)$ observes the equivalence relation $(x,r)\sim (y,s)$ given by $\al_{s^{-1}r}(x)=y$, since,
\[ \beta_{s^{-1}}\big(\psi'(x,r)\big)
                    =\beta_{s^{-1}}\big(\beta_r(\psi (x))\big)
                    =\beta_{s^{-1}r}\big(\psi(x)\big)
                    =\psi\big(\al_{s^{-1}r}(x)\big)
                    =\psi(y), \]
and  $\psi'(x,r)=\beta_s\big(\psi (y)\big) =\psi'(y,s)$. Thus 
it induces a continuous map $\env{\psi}:\env{X}\f Y$, with
$\env{\psi}\big(q(x,t)\big)=\beta_t(\psi(x))$, for $t\in G$, $x\in X$.
Then, $\env{\psi}\iota(x)=\env{\psi}\big(q(x,e)\big)=\psi (x)$, and 
$\env{\psi}:\env{\al}\f\beta$ is a morphism, uniquely 
determined by  $\env{\psi}\iota=\psi$. 
\end{proof}
\par Since  $(\iota,\env{\al})$ is characterized by a universal 
property, it is unique up to isomorphisms (c.f., \cite{rowen}).  
The action $\env{\al}$ above is called an {\it enveloping action} of $\al$ (or simply, a {\it globalization} of $\al$). Note that as 
$\env{X}$ is the $\hal$--orbit of $X$,  $X$ and $\hex$ share the 
same local properties.

\begin{ex}\label{rk:susp}
$(i)$ Assume that $h:X\f X$ is a homeomorphism, so we have an action of $\Z$ on 
$X$. We may think of this action as a partial action of $\R_d$ on $X$, where 
$\R_d$ denotes the real numbers with the discrete topology. Indeed, define 
$X_s=X$ if $s\in \Z$, $X_s=\emptyset$ if $s\notin\Z$, and 
$\al_s:X_{-s}\f X_s$ as $\al_s=h^s$ if $s\in\Z$, $\al_s=\emptyset$ otherwise. 
Note that $\al$ is not a partial action of $\R$ on $X$, because $\Z\times X$ 
is not open in $\R\times X$. However, we can imitate the construction of the 
enveloping action made in the proof of \ref{env} above, using $\R$ 
instead of $\R_d$, to obtain 
 a global continuous action 
$\beta:\R\times Y\to Y$, where $Y:=(X\times \R)/\sim$, such that 
$\beta_n(x)=\al_n(x)$, $n\in\Z$, $x\in X$. This action $\beta$ 
is called the {\em suspension of $h$}, and its construction is well known in 
dynamical systems theory \cite[Page 45]{tomi}.

$(ii)$ Consider the action $\beta:\Z\times S^1\f S^1$ given by the 
rotation by an irrational angle $\theta$: 
$\beta_k(z)=e^{2\pi ik\theta}z$, for $k\in\Z$, $z\in S^1$.
Let $U$ be a nonempty open arc of $S^1$, $U\neq S^1$, and consider the \pa 
given by the restriction $\al$ of $\beta$ to $U$ (\ref{rest}). 
Since the action $\beta$ is minimal, it follows that $\beta$ is the 
enveloping action of $\al$. This example shows that, even when $X$ and 
$\hex$ are similar locally, their global properties may be deeply different.
In this case, for instance, the first homotopy groups of $U$ and $S^1$ are 
different (c.f., \cite{fav5}).

$(iii)$ Consider the \pa $\al$ of $\Z_2$ on the unit interval $X=[0,1]$, given by  
$\al_1=id_X$, $\al_{-1}=id_V$, where $V=(a,1]$, $a>0$. Let 
$\hal :G\times\hex\f\hex$ be the enveloping action of $\al$. 
Consider $J=J^-\cup J^+\subseteq\R^2$ with the relative topology, where 
$J^\pm=\{\pm 1\}\times [0,1]$. Then
$\hex$ is the topological quotient space obtained from $J$ by identifying 
the points $(1,t)$ and 
$(-1,t)$, $t\in (a,1]$ \cite{fav5}. Therefore, $\hex$ is not a Hausdorff space: 
$(1,a)$ and $(-1,a)$ do not have disjoint neighborhoods. Note also that 
$\hal_{-1}$ permutes $(1,t)$ and $(-1,t)$ for $t\in [0,a]$, and is the 
identity in the rest of $\hex$.
\end{ex}

\begin{rk}\label{prop:t2}
$(i)$ Let $\al$ be a \pa\ of $G$ on the Hausdorff space $X$. Let $\gr(\al )$
be the graph of $\al$, that is 
$\gr(\al )=\{(t,x,y)\in G\times X\times X:\ x\in X_{t^{-1}},\, \al_t(x)=y \}.$ 
Then $\hex$ is a Hausdorff space if and only if $\gr(\al )$ is a closed 
subset of $G\times X\times X$ \cite[Proposition 1.2]{fav5}.

$(ii)$ If $G$ is a discrete group, then $\gr(\al )$ is closed in 
$G\times X\times X$ if and only if  
$\gr(\al_t)$ is closed in $X\times X$, for $t\in G$.

$(iii)$ As already seen in \ref{rest}, the flow of a smooth vector field on a 
manifold is a \pa, indeed a smooth \pa. The enveloping space inherits a 
natural manifold structure, although not always Hausdorff, by translating 
the structure of the original manifold through the enveloping action. It 
would be interesting to characterize those vector fields whose flows have 
closed graphs. For such a vector field, one obtains a Hausdorff manifold that 
contains the original one as an open submanifold, and a vector field whose 
restriction to this submanifold is the original vector field. Note, however,  
that the inclusion of the original manifold in its enveloping one could be 
 complicated \cite{fav5}. 
\end{rk}
\ve
Many of the algebraic and even dynamical notions related to global 
actions may be easily extended to the context of \pas. For instance, it is 
possible to make sense of expressions such as transitive \pas or minimal \pas.
To give an example, we could  say that a \pa $\al$ on a topological space $X$ is 
{\it minimal} when each $\al$--orbit is dense in $X$, that is, when 
$X=\ov{\{\al_t(x):\, t\in G_{x}\}}$, for each $x\in X$, where $G_x:=\{t\in G: x\in X_{t^{-1}}\}$. The dynamical properties of $\al$ and $\env{\al}$ are 
in general the same, for instance, 
it is not hard to see that $\al$ is minimal if and only if $\env{\al}$ is 
minimal. However, there are notions for partial actions which happen to be trivial for global actions. We would define a new notion of this type for transitivity of \pas in the next section.

\section{Amenable partial actions}\label{sec:am} 
In this section we define amenability of partial actions in terms of amenability of the corresponding groupoid of germs.
Recall that a {\em groupoid} is a set $\g$ with a distinguished subset $\g^{(2)} \subset \g \times \g$, called the set of {\em composable pairs}, a product map $\g^{(2)} \to \g$ with $(\gamma, \eta)\mapsto \gamma\eta$, and an inverse map from $\g$ to $\g$ with $\gamma \mapsto \gamma^{-1}$ such that
\begin{enumerate}
	\item $(\gamma^{-1})^{-1} = \gamma$ for all $\gamma\in \g$,
	\item If $(\gamma, \eta), (\eta, \nu)\in \g^{(2)}$, then $(\gamma\eta,\nu), (\gamma, \eta\nu)\in \g^{(2)}$ and $(\gamma\eta)\nu = \gamma(\eta\nu)$,
	\item $(\gamma, \gamma^{-1}), (\gamma^{-1},\gamma)\in \g^{(2)}$, and $\gamma^{-1}\gamma\eta = \eta$, $\xi\gamma\gamma^{-1} = \xi$ for all $\eta, \xi$ with $(\gamma, \eta), (\xi,\gamma) \in \g^{(2)}$.
\end{enumerate}
The set of {\em units} of $\g$ is the subset $\g^{(0)}$ of elements of the form $\gamma\gamma^{-1}$. The maps $r: \g\to \g^{(0)}$ and $s:\g\to \g^{(0)}$ given by $
r(\gamma) = \gamma\gamma^{-1}, \ s(\gamma) = \gamma^{-1}\gamma
$
are called the {\em range} and {\em source} maps respectively. One sees that $(\gamma, \eta)\in \g^{(2)}$ is equivalent to
$r(\eta) = s(\gamma)$. 

A map $\varphi: \g\to \mathcal{H}$ between groupoids is called a {\em groupoid homomorphism} if $(\gamma, \eta)\in \g^{(2)}$ implies that $(\varphi(\gamma), \varphi(\eta))\in \mathcal{H}^{(2)}$ and $\varphi(\gamma\eta) = \varphi(\gamma)\varphi(\eta)$. This implies that $\varphi(\gamma^{-1}) = \varphi(\gamma)^{-1}$, and so $\varphi(\g^{(0)}) \subset \mathcal{H}^{(0)}$, $r\circ\varphi = \varphi\circ r$, and $s\circ\varphi = \varphi \circ s$.

A {\em topological groupoid} is a groupoid which is a topological space where the inverse and product maps are
continuous, where we are considering $\g^{(2)}$ with the product topology inherited from $\g\times\g$. A topological
groupoid is called {\em \'etale} if it is locally compact, its unit space is Hausdorff, and the range and source maps are local homeomorphisms. These properties imply that $\g^{(0)}$ is open. Furthermore, in a second countable \'etale groupoid, the spaces
$
\g_x : = s^{-1}(x), \ \g^x := r^{-1}(x)
$
are discrete for all $x\in \g^{(0)}$. We note that an \'etale groupoid may not be Hausdorff, even though we always assume the unit space is.

The following result from \cite{AR00} will be used  as our  definition of amenability for a second countable groupoids.

\begin{lem}\label{am}
	Let $\mathcal{G}$ be a second countable locally compact groupoid with a continuous Haar system $\lambda$. The following are equivalent:
	\begin{enumerate}
		\item $\mathcal{G}$ is amenable.
		\item There exists a net $(g_i)$ of of positive Borel (continuous) functions (of compact support) on $\g$ 
		such that,
		\begin{enumerate}
			\item $\lambda(g_i)(u)\leq 1$, for $u\in\mathcal{G}^{(0)}$;
			\item $\lambda(g_i)(u)\to 1$, for $u\in\mathcal{G}^{(0)}$; and
			\item for all $\gamma\in\mathcal{G}$, 
			\[
			\int_{\g}\left|g_i(\gamma^{-1}\eta) - g_i(\eta)\right|d\lambda^{r(\gamma)}(\eta)\to 0.
			\]
		\end{enumerate}
		\item There exists a net $(m_i)$ of  families of Radon measures $m_i=(m_i^u)_{u\in\mathcal{G}^{(0)}}$ on $\g$ with ${\rm supp}(m_i^u)\subseteq \mathcal{G}^{u}$,
		such that,
		\begin{enumerate}
			\item $m_i^u(\g)\leq 1$, for $u\in\mathcal{G}^{(0)}$;
			\item $m_i^u(\g)\to 1$, for $u\in\mathcal{G}^{(0)}$; and
			\item  $\|\gamma\cdot m_i^{s(\gamma)}-m_i^{r(\gamma)}\|\to 0,$ for  $\gamma\in\mathcal{G}$.
		\end{enumerate}
	\end{enumerate}
\end{lem}

Following \cite{es}, we define a type of morphism which arises naturally when considering \pas.
Let $\g$ and $\mathcal{H}$ be topological groupoids. We say that a morphism $\rho:
	\g\to \mathcal{H}$ is {\em $s$-bijective} if for all $x\in \g^{(0)}$, the restriction
	$\rho: \g_x \to \mathcal{H}_{\varphi(x)}$ is bijective. Similarly, $\rho$ is {\em $r$-bijective} if for all $x\in \g^{(0)}$, the restriction $\rho: \g^x \to \mathcal{H}^{\rho(x)}$ is bijective.
	A morphism $\rho:\g\to \mathcal{H}$ is $s$-bijective if and only if it is $r$-bijective. 
	
	\begin{lem}\label{am2}
	If there is an $s$-bijective  Borel map $\rho: \g\to \mathcal{H}$, then amenability of  $\mathcal{H}$ implies that of $\g$.
	\end{lem}
	\begin{proof}
Let $(g_i)$ be a net of positive Borel maps on $\mathcal{H}$ as in Lemma \ref{am}(2). Then we get a net of positive Borel maps $h_i:=g_i\circ\rho$ on $\g$, satisfying, 
$$\lambda_\g(h_i)(u)=\lambda_\mathcal{H}(g_i)(\rho(u)),\ \ (u\in \g^{(0)}),$$
and
\begin{align*}
\int_{\g}\left|h_i(\gamma^{-1}\eta) - h_i(\eta)\right|&d\lambda^{r(\gamma)}(\eta)\\&=\int_{\mathcal{H}}\left|g_i(\rho(\gamma)^{-1}\rho(\eta)) - g_i(\rho(\eta))\right|d\lambda_\mathcal{H}^{r(\rho(\gamma))}(\rho(\eta)),
\end{align*}
	
therefore, amenability of  $\g$ follows from that of $\mathcal{H}$, by Lemma \ref{am}.
	\end{proof}
	
\par Given a continuous \pa $\alpha$ of a topological group $G$ on a topological space $X$ we can form a groupoid which encodes the action. The {\em groupoid of germs} the action is
	\begin{equation}\label{germs}
	X\rtimes_{\al}G:=\big\{(x,t): t\in G, 
	x\in X_{t^{-1}}\big\}.
	\end{equation}
	The groupoid operations are given by $r((x,t)) = \alpha_t(x),  s((x,t)) = x,$ and,
	\[
	(x,t)^{-1} = (\alpha_t(x), t^{-1}),  \ \ (\alpha_s(x), t)(x,s) = (x, ts).
	\]
	For $s\in \s$ and an open set $U\subset X_{s^{-1}}$, the sets
	\[
	U_s := \{(x, s): x\in U\},
	\]
	generate a topology on $X\rtimes_{\al}G$, and under this topology $	X\rtimes_{\al}G$ is a locally compact groupoid with continuous Haar system, when both $G$ and $X$ are locally compact. It is \'etale, when $G$ is discrete and second countable, if $G$ is so. 
	
	To define different notions of amenability of \pas, we need some preparation. Let $B(X)$ be the set of bounded Borel functions on $X$. This is a unital C*-algebra under pointwise  operations and sup norm. A {\it mean} on $B(X)$ is a state on $B(X)$. For each open subset $U\subseteq X$ and $f\in C_b(U)$, we may identify $f$ by the extension  by zero of $f$, thereby regarding $f$ as an element of $B(X)$. Given a \pa $\al$ of $G$ on $X$, we get  isomorphisms,
	$$\al_s: C_b(X_s)\to C_b(X_{s^{-1}}); \ \ \al_s(f)(x):=f(\al_{s}(x)),\ (x\in X_{s^{-1}}).$$
	
	Next, let $(X,\nu)$ be a standard Borel space with a Borel \pa of $G$. We assume that $\nu$ is {\it quasi-invariant}, in the sense that, for each $s\in G$ and each Borel subset $E\subseteq X_{s^{-1}}$, $\nu(E)=0$ implies $\nu(\al_s(E))=0$. Let $[\frac{d(\nu\circ\al_s)}{d\nu}]\in L^1(X_{s^{-1}}, \nu)$ be the corresponding Radon-Nikodym derivative, then $$\sig_{\rm RN}: X\rtimes_\al G\to \mathbb R^+;\ \ \sig_{\rm RN}(x,s):=[\frac{d(\nu\circ\al_s)}{d\nu}](x),\ \ (x\in X_{s^{-1}}),$$
	is called the Radon-Nikodym cocycle of $\al$.
	
	For a Borel group $M$, a {\it partial 2-cocycle} $\sig: X\rtimes_{\al}G\to M$ is a Borel map satisfying,
	$$\sig(x,ts)=\sig(x,s)\sig(\al_s(x),t),\ \ (x\in X_{s^{-1}}\cap X_{s^{-1}t^{-1}}).$$
	For a separable Banach space, the closed unit ball $E^*_1$ is a compact metrizable space in weak$^*$-topology. Let Iso$(E)$ and Hom$(E^*_1)$ be the groups of isometric isomorphisms of $E$ with strong operator topology and the group of hmeomorphisms of $E^*_1$ with the topology of uniform convergence. Then Iso($E$) is a standard Borel group and the canonical map: Iso$(E)\to$ Hom$(E^*_1)$ is continuous (and so Borel) \cite{zi1}. A family $\{A_x\}$ of non-empty compact convex subset $A_x\subseteq E^*_1$ is called a Bore1
	field if $\{(x, \phi):\phi\in A_x\}$ is a Bore1 subset of $X\times E^*_1$. For a partial 2-cocycle $\sig: X\rtimes G\to$ Iso$(E)$, the induced partial 2-cocycle $\sig^*: X\rtimes G\to$ Hom$(E^*_1)$ is defined by $\sig^*(x,s): = (\sig(x, s)^{-1})^*$. A Bore1 field  is called $\sig$-{\it invariant} if 
	$\sig^*(x,s)A_x=A_{\al_s(x)}$, for a.a. $x\in X_{s^{-1}}$.      
	    	 
	\begin{df}\label{amen}
		We say that a continuous (Borel) partial action $\alpha$ of a locally compact group  $G$ on a locally compact (standard Borel) space $X$ is {\it amenable} (or simply, $X$ is an {\it amenable} partial $G$-space) 
		
		$(i)$ in the sense of Delaroche, if the groupoid of germs $X\rtimes_{\al}G$ is amenable,
		
		$(ii)$ in the sense of Greenleaf,  if there mean $m$ on $B(X)$ which is $\al$-invariant, that is, $m(\al_s(f))=m(f)$, for each $s\in G$ and $f\in C_b(X_{s^{-1}})$,
		
		$(iii)$ in the sense of Zimmer, if  for
		each separable Banach space $E$, each partial 2-cocycle $\sig: X\rtimes_\alpha G\to$ Iso($E$), and each $\sig$-invariant
		Bore1 field $\{A_x\}$, there is a Bore1 $\sig$-invariant Borel section $\eta: X\to E^*_1$, that is, a Borel map such that $\eta(x)\in A_x$ a.e. on $X$,
		and $\sig^*(x, s)\eta(\al_s(x)) = \eta(x)$ a.e. on $X_{s^{-1}}$. 		
	\end{df}

First note that $(ii)$ and $(iii)$ are quite different, even for global actions: for a closed subgroup $H$, the canonical action of $G$ on $G/H$ is amenable in Greenleaf sense iff $H$ is co-amenable (c.f., \cite{ey}, \cite{mp}), where as, it is  amenable in Zimmer sense iff $H$ is amenable \cite[Theorem 1.9]{zi1}. When $H$ is also normal in $G$, it is co-amenable iff $G/H$ is amenable. In particular, the trivial action on a singleton is always amenable in Greenleaf sense, but is amenable in Zimmer sense iff $G$ is amenable. 

The co-amenability of $H\leq G$ is also known to be equivalent to the so-called {\it conditional fixed point property} of Eymard \cite{ey}: each  continuous affine $G$-action on a convex compact subset of a locally convex space with an
$H$-fixed point has a $G$-fixed point. 

Using the idea of Example \ref{rest}$(i)$, one could build a partial version of the canonical action on the homogeneous spaces. 

\begin{ex} \label{homog}
Let $H\leq G$ be a closed subgroup of a (locally compact) topological group $G$ and $U\subseteq G$ be an open subset. Then $X:=G/H$ consisting of right cosets of $H$ is a (locally compact) topological space in the quotient topology, which is Hausdorff if $G$ is so. Consider the open subsets  $X_t:=UH/H\cap UHt/H$ of $X$, for $t\in G$, and let $G$ act partially on $G/H$ by $$\al_t: X_{t^{-1}}\to X_t; \ Hx\mapsto Hxt, \ (xH\in UH/H\cap UHt/H).$$
When $K=G$, this is the canonical global action of $G$ on $G/H$.   
\end{ex}

\begin{df}
A \pa $\al$ of $G$ on $X$ is called {\it partially transitive} if for each $s,t\in G$, $X=\bigcup_{r\in G} X_{sr}\cap X_{tr}$.  
\end{df}

When $\al$ is a global action, then it is automatically partially transitive. On the other hand, a typical example of a partially transitive \pa is the restriction of a weakly transitive global action to an open subset. Recall that a global action is called weakly transitive if it satisfies $X=\bigcup_{r\in G} U\cdot r$, for each open set $U\subseteq X$. A weakly transitive global action separating the compacts always has an invariant measure (see \cite[Definition 4.1, Theorem 4.4]{st} for definitions and details).

\begin{ex} \label{pt}
The restricted partial action on an open subset $A\subseteq X$ with $A_t:=A\cap (A\cdot t)$ is partially transitive, as for given $s,t\in G$, we have, 
\begin{align*}
\bigcup_{r\in G} (A_{sr}\cap A_{tr})&=\bigcup_{r\in G} A\cap (A\cdot sr)\cap (A\cdot tr)\\&=A\cap (\bigcup_{r\in G} (A\cdot s\cap A\cdot t)\cdot r=A\cap X=A.
\end{align*} 
\end{ex}

\begin{lem} \label{lift1}
Given a partially transitive partial action $\alpha$ of a locally compact group  $G$ on a standard Borel space $X$ with enveloping action $\env{\al}$ on $\env{X}$, there is a $G$-factor $Y$ of $X$ such that,  given separable Banach space $E$, each partial 2-cocycle $\sig: Y\rtimes_\alpha G\to$ Iso($E$) lifts to a 2-cocycle $\env{\sig}: \env{X}\times G\to$ Iso($E$). 
\end{lem}	
\begin{proof}
We freely use the notations of the proof of Lemma \ref{env}. Given $([x,t],s)$ in $\env{X}\times G$, by partial transitivity, there is $r\in G$ such that 
$$x\in X_{r^{-1}ts^{-1}}\cap X_{r^{-1}t}=\al_{r^{-1}t}(X_{s^{-1}}\cap X_{t^{-1}r}).$$
Choose $y\in X_{s^{-1}}\cap X_{t^{-1}r}$ with $\alpha_{r^{-1}t}(y)=x$. It follows that $[x,t]=[y,r]$ and $(y,s)\in X\rtimes_{\al} G$. Define $y\approx y^{'}$ if there are $r,r^{'}\in G$ such that $[y,r]=[y^{'}, r^{'}]$. This is an equivalence relation observed by the $G$-partial action, thus $G$ partially acts on $Y$, via $$\tilde\al_s([y]):=[\al_s(y)];\ \ (s\in G, y\in Y).$$ Put $Y=\frac{X}{\approx}$ and set $\env{\sig}([x,t],s):=\sig([y],s)$. 
Next,
for $([x,st],u)\in \env{X}\times G$, choose $v\in G$ and $z\in X_{u^{-1}}\cap X_{t^{-1}s^{-1}v}$ with  $\alpha_{v^{-1}st}(z)=x$, then, $$[z,v]=[x,st]=\env{\al}_s([x,t])=\env{\al}_s([y,r])=[y, sr]=[\al_{s}(y),r],$$ thus, $[z]=[\al_{s}(y)]=\tilde\al_{s}([y]),$ therefore, 
\begin{align*}
\env{\sig}([x,t],s)\env{\sig}(\env{\al}_s([x,t]),u)&=\env{\sig}([x,t],s)\env{\sig}([x,st],u)\\&=\sig([y],s)\sig([z],u)\\&=\sig([y],s)\sig(\tilde\al_{s}([y]),u)\\&=\sig([y],us)\\&=\env{\sig}([x,t],us),    
\end{align*}
as required.
\end{proof}

\begin{df}
Let $G$ acts partially on $(X,\nu)$ by $\al$. A Borel subset $B\subseteq X$ is called $\al$-{\it invariant} if $\al_t(B\cap X_{t^{-1}})=B\cap X_t$, for $t\in G$. We say that a \pa $\al$ is {\it ergodic} if each $\al$-invariant Borel subset is null or co-null.
\end{df}

\begin{df}
Let $(X,\nu)$ and $(Y,\mu)$ are standard Borel spaces and there is a Borel surjection $p: X\twoheadrightarrow Y$, with $p_*\nu=\mu$. If $G$ acts partially on both $X$ and $Y$ by $\al$ and $\beta$ in a such a way that $X_t=p^{-1}(Y_t)$ and $\beta_t(p(x))=p(\al_t(x))$, for $t\in G$ and $x\in X$, we say that $Y$ is a {\it factor} of $X$ and write $\beta=:p_*\al$. 
\end{df}
\begin{lem} \label{factor}
If $X$ is an  ergodic partial $G$-space and $Y$ is a factor of $X$, then amenability of $Y$ in the sense of Zimmer implies that of $X$. 	
\end{lem}
\begin{proof}
As in the proof of \cite[Theorem 2.4]{zi1}, we may assume that $X\subseteq I\times Y$ is co-null and Borel, and $\mu=m\times \nu$, where $(I,m)$ is the unit interval with Lebesgue measure. Let $E$ be a separable Banach space, $\sig: X \rtimes_\al G\to {\rm Iso}(E)$ be a partial 2-cocycle, and $\{A_x\}$ be a $\sig$-invariant
Borel field in $E^*_1$. Let $F:= L^1(I, E)$. Define $p_*\sig: Y \rtimes_{p_*\al} G\to \mathbb B(F)_1$ by,
$$p_*\sig(y,s)f(\theta):=\sig_{\rm RN}(\theta, y, s)\sig(\theta, y,s)f(p_1(\al_s(\theta,y))); \ \ (s\in G, (\theta, y)\in X_{s^{-1}}),$$
where $\sig_{\rm RN}$ is the Radon-Nikodym cocycle of $\al$ and $p_1: I\times Y\to I$ is the orthogonal projection on the first leg. As in the proof of \cite[Theorem
2.1]{zi1}, one could show that $p_*\sig$ is Borel. Next, for, $$\{y,s\}: I\to I; \ \theta\mapsto p_1(\al_s(\theta,y)),$$
let us observe that,
$$\{y,s\}^{-1}=\{p_*\al_s(y), s^{-1}\}, \ \ (y\in Y_{s^{-1}}).$$ 
Indeed, 
\begin{align*}
\{p_*\al_s(y), s^{-1}\}(\{y,s\}(\theta))&=p_1(\al_{s^{-1}}(p_1(\al_s(p_1(\al_s(\theta,y)),p_*\al_s(y)))))\\&=p_1(p_1(\al_s(p_1(\theta,y)),p_*\al_{s^{-1}}p_*\al_s(y)))\\&=p_1(p_1(\al_s(p_1(\al_s(\theta,y)),y)))\\&=p_1(p_1(p_1(\theta,p_*\al_sy)),p_*\al_sy)\\&=p_1(p_1(\theta,p_*\al_sy),p_*\al_sy)\\&=p_1(\theta,,p_*\al_sy)\\&=\theta,
\end{align*}
and the same for the reverse composition. Now, a similar argument as in the proof of \cite[Theorem 2.4]{zi1} shows that $p_*\sig(y, s)$ is in
Iso($F$), and after a suitable co-null set, we may  $p_*\sig$ is a Borel partial 2-cocycle. Put,
$$B_y:= \{f:I\to E^*: f(\theta)\in A_{(\theta,y)} \ ({\rm a.a.}\ \theta)\}.$$ By \cite[Proposition
2.2.]{zi1}, this is a compact convex set, and $(B_y)$ is a Bore1 field by \cite[Lemmas 1.7, 2.5]{zi1}. Next, since $dm(\{y,s\}(\theta))=\sig_{\rm RN}(\theta, y,s)dm(\theta),$
\begin{align*}
	\langle p_*\sig^*(y,s)g, f\rangle &=\langle g, p_*\sig(y,s)f\rangle \\&=\int_I g(\theta)\sig_{\rm RN}(\theta,y,s)\sig(\theta,y,s)f(\{y,s\}(\theta))dm(\theta)\\&=\int_I g\circ\{y,s\}^{-1}(\theta)\sig(\theta,y,s)f(\theta)dm(\theta)\\&=\int_I g\circ\{p_*\al_s(y),s^{-1}\}(\theta)\sig(\theta,y,s)f(\theta)dm(\theta)\\&=\langle \sig^*(y,s)g\circ\{p_*\al_s(y),s^{-1}\}, f\rangle,
\end{align*} 
for $f\in L^1(I, E)$ and $g\in L^\infty(I, E^*)$, thus, $$p_*\sig^*(y,s)g=\sig^*(y,s)g\circ\{p_*\al_s(y),s^{-1}\}.$$ This shows that $\{B_y\}$ is $p_*\sig$-invariant. Now amenability of $Y$  implies the existence of a
$p_*\sig$-invariant Bore1 section $\eta: Y \to F^*_1$ with $\eta(y)\in B_y$ a.e. Then $$\tilde\eta(\theta,y) :
=\eta(y)(\theta),\ \ (y\in Y, \theta\in I),$$ defines a $\sig$-invariant Borel section
$\tilde\eta: X \to E^*_1$, showing the amenability of $X$.  
\end{proof}

\begin{ex} \label{lift2}
$(i)$ Back to Example \ref{pt}, Given a weakly transitive global action on $X$ and open subset $A\subseteq X$, the restricted partial action $\al$ on $A$ is partially transitive. With the notation of Lemma \ref{lift1}, for $y,z\in A$, $[y]=A\cap \mathcal O_y$, where $\mathcal O_y$ is the orbit of $y$ under the global action of $G$. Thus we have a \pa of $G$ on the orbit space $A/G$ such that each partial 2-cocycle $\sig: (A/G)\rtimes G\to {\rm Iso}(E)$ lifts to a global 2-cocycle $\env{\sig}: \env{A}\times G \to {\rm Iso}(E)$. When the original global action is transitive, the orbit space $A/G$ is trivial and $\sig$ is nothing but a representation of $G$ on $E$. On the other hand, $[x,s]=[y,t]$ iff $x\cdot s^{-1}t=y$ iff $x\cdot s^{-1}=y\cdot t^{-1}$, thus the map $[x,t]\mapsto x\cdot t^{-1}$ is well-defined and injective from $\env{A}$ to $X$, which is also onto, as $\cup_{t\in G} A\cdot t^{-1}=X$, by weak transitivity. 

$(ii)$ A concrete example of the latter case in $(i)$ is the (transitive) action of $G$ on $G/H$, with $A=UH/H$, as in Example \ref{homog}.   
\end{ex}

In the last statement of the next result, $Y:=\frac{X}{\approx}$, where $\approx$ is the equivalence relation defined in \ref{factor}.

\begin{thm}\label{main1}
Let $X$ be a partially transitive partial $G$-space with enveloping $G$-space $\env{X}$. If $\env{X}$ is amenable in the sense of Zimmer, then so is $X$. Conversely, if $Y:=\frac{X}{\approx}$ is amenable in the sense of Zimmer, then so is $\env{X}$.    
\end{thm}
\begin{proof}
To prove the first statement, we freely use the notations of the proof of Lemmas \ref{env} and \ref{lift1}, in particular, we let $Y$ be the factor of $X$ constructed in \ref{lift1}. Let $\sig: Y\rtimes_{\tilde \al}G\to {\rm Iso}(E)$ be a partial 2-cocycle and $\{A_{[y]}\}$ be a $\tilde\sig$-invariant Borel field on $Y$. Put 
$B_{[y,s]}:=A_{[\al_s(y)]}$, then,
$$\{([y,s], \phi): \phi\in B_{[y,s]}\}=\{([y,s], \phi): \phi\in A_{[\al_s(y)]}\}\subseteq \env{X}\times E^*_1,$$ is Borel, that is, $\{B_{[y,s]}\}$ is a Borel field. By amenability of $\env{X}$, there is a $\env{\sig}$-invariant Borel section $\env{\eta}: \env{X}\to E^*_1$, that is,
$${\env{\sig}}^{*}([x,t],s)\env{\eta}(\env{\al}_s([x,t]))=\env{\eta}([x,t]),\ \ (x\in X, s,t\in G),$$
and so for $t=e$,
$${\sig}^{*}([x],s)\env{\eta}(\iota(\al_s(x)))=\env{\eta}(\iota(x)),\ \ (x\in X, s\in G).$$ 
Put $\eta:=\env{\eta}\circ\iota$, then for $x\approx y$ in $X$, we have $x=\al_{r}(y)$, for some $r\in G$, hence by $\env{\sig}$-invariance of $\env{\eta}$,
$$\eta(x)=\eta(\iota(x))=\env{\eta}(\iota(\al_r(y)))=\env{\eta}(\env{\al}_r(\iota(y)))=\env{\eta}(\iota(y))=\eta(y),$$
which means that $\eta$ lifts to a Borel section $\tilde\eta: Y\to E^*_1; \ \tilde\eta([y]):=\eta(y)$, for which we have,
\begin{align*}
{\sig}^{*}([x],s)\tilde\eta(\tilde\al_s([x]))={\sig}^{*}(\iota(x),s)\env{\eta}(\iota(\al_s(x)))=\env{\eta}(\iota(x))=\tilde\eta([x]),   
\end{align*}
for $x\in X, s\in G,$ that is, $\tilde\eta$ is $\sig$-invariant. Therefore, $Y$ is amenable, and so is $X$, by Lemma \ref{factor}.

Conversely, to prove the second statement, we freely use the notations of the proof of Lemma \ref{factor}, in particular, since $Y$ is a $G$-facor of $X$, we may identify $X$ with a co-null subset of $I\times Y$. Put  $F:=L^1(I,E)$. Given a 2-cocycle $\env{\sig}:\env{X}\times G\to {\rm Iso}(E)$, identifying $X$ with $\iota(X)\subseteq \env{X}$, we let $\sig$ be the restriction of $\env{\sig}$ to $X\rtimes_{\al}G\subseteq \env{X}\times G$. Let $\{C_{[x,t]}\}$ be a Borel field in $E^*_1$ and put $A_x:=C_{[x,e]}$. Then, for $x=:(\theta, y)$, let $B_y\subseteq F^*_1$ be the set associated to $A_x=A_{(\theta,y)}$ as in \ref{factor}. Let $p: X\to Y$ be the orthogonal projection onto the second leg and consider the induced partial 2-cocycle $p_*\sig: Y\rtimes_{\tilde\al} G\to {\rm Iso}(F)$. By amenability of $Y$, there is a $p_*\sig$-invariant Borel section $\eta: Y\to F^*_1$. For $x\in X$ and $t\in G$, put $\env{\eta}([x,t]):=\eta([x])\circ\{p_*\al_t([x]), t^{-1}\}$. This is well-defined, since if $[x,t]=[x^{'},t^{'}]$, then $[x]=[x^{'}]$. Now for $x\in X$ and $t\in G$, we have,
\begin{align*}
 {\env{\sig}}^{*}([x,t],s)\env{\eta}(\env{\al}_s([x,t]))&=\sig^*([x],s)\env{\eta}([x,st])\\&=\sig^*([x],s)\eta([x])\circ\{p_*\al_{st}([x]),s^{-1}t^{-1}\}\\&=\sig^*([x],t)^{-1}\sig^*([x],st)\eta([x])\circ\{p_*\al_{st}([x]),s^{-1}t^{-1}\}\\&=\sig^*([x],t)^{-1}p_*\sig^*([x],st){\eta}([x])\\&=\sig^*([x],t)^{-1}p_*\sig^*([\al_{t}(x)],s){\eta}([x])\\&=\eta([x])\circ\{p_*\al_t([x]), t^{-1}\}\\&= \env{\eta}([x,t]),  
\end{align*}
for each $x\in X, s,t\in G$, that is, $ \env{\eta}$ is $\env{\sig}$-invariant.      
\end{proof}

The next definition uses the notion of partial representation  (c.f., Definition \ref{df:pr}).

\begin{df}
	Given a Borel partial action $\alpha$ of a locally compact group  $G$ on a standard Borel space $X$, we say that $(X, G)$ is an {\it amenable pair} in the sense of Eymard
	if  for
	each separable Banach space $E$, each partial representation  $\pi: G\to$ Iso($E$),  and each $\al$-invariant compact convex set $A\subseteq E^*_1$, the existence of an $\pi$-invariant section $\eta: X\to A$ implies the existence of a $\al$-fixed point in $A$, where $G$ acts on $A$ via $\pi$. Here, the section $\eta$ is $\pi$-invariant, if $\pi^*_s\eta(x)=\eta(\al_s(x))$, if $x\in X_{s^{-1}}$, and 0, otherwise. 
\end{df}

\begin{prop}
If $X$ is an amenable partial $G$-space in the sense of Greenleaf, then $(X,G)$ is amenable pair in the sense of Eymard.
\end{prop}
\begin{proof}
Given a partial representation $\pi: G\to$ Iso($E$),  $\al$-invariant compact convex set $A\subseteq E^*_1$, and  $\al$-invariant section $\eta: X\to A$, let $m$ be an $\al$-invariant mean on $B(X)$ as in Definition \ref{amen}$(ii)$, and put,
$$\langle a,\xi\rangle:=\int_X \langle \eta(x),\xi\rangle dm(x),\ \ (\xi\in E),$$ 
then by convexity and compactness of $A$, we get $a\in A$. Moreover, 
\begin{align*}
\langle \pi^*_sa,\xi\rangle &=\langle a,\pi_s\xi\rangle\\&=\int_X \langle \eta(x),\pi_s\xi\rangle dm(x)\\&=\int_X \langle \pi^*_s\eta(x),\xi\rangle dm(x)	\\&=\int_{X^{s^-1}} \langle \eta(\al_s(x)),\xi\rangle dm(x)
\\&=\int_{X} \langle \eta(x),\xi\rangle dm(x)\\&=\langle a,\xi\rangle,
\end{align*}
that is, $a$ is an $\al$-fixed point, as required.
\end{proof}

Next, we prove a version of the Reiter's condition for partial actions. 

\begin{df}
	For $1\leq p<\infty$, we say that a \pa $\al$  on a Borel measure space $(X,\nu)$ satisfies {\it Reiter condition} $(P_p)$, if 
	for each compact subset $K\subseteq G$ and $\varepsilon>0$, there is a norm one positive function  $f\in L^p(X,\nu)$ such that,
	$$\sup_{t\in K}\int_{X_{t^{-1}}} |f(x)-\sigma_{\rm RN}^{\frac{1}{p}}(x, t)f(\al_{t}(x))|^pd\nu(x)<\varepsilon.$$ 
\end{df}

When $G$ acts on itself my multiplication, the left and right multiplication actions commute. We need the following analog for \pas.

\begin{df}
A Borel (continuous) \pa $\al$ on a Borel (topological) space $X$ with partial sets $\{X_t\}_{t\in G}$ is called {\it symmetric} if there is a Borel (continuous) \pa $\beta$ on $X$ with the same partial sets $\{X_t\}_{t\in G}$ such that,
$\al_s\beta_t=\beta_t\al_s,\ {\rm on}\ X_{s^{-1}}\cap X_{t^{-1}}.$
\end{df}

\begin{prop}
If $G$ is an amenable locally compact group, then all symmetric continuous \pas of $G$ have Reiter condition $(P_1)$.  
\end{prop}
\begin{proof}
Let $\al$ be a continuous \pa of $G$ on a topological space $X$ with open partial sets $\{X_t\}_{t\in G}$. By the symmetry condition, choose a \pa $\beta$ with,
$$\al_s\beta_t(x)=\beta_t\al_s(x), \ \ (x\in X_{s^{-1}}\cap X_{t^{-1}}).$$
For $t\in G$, consider the operator $T_t$  from $L^1(X_{t},\nu)$ to $L^1(X_{t^{-1}},\nu)$, defined by,
$$T_tf(x):=\sigma^\beta_{\rm RN}(x, t)f(\beta_t(x)), \ \ (x\in X_{t^{-1}}, f\in L^1(X_{t},\nu)).$$
Then, \begin{align*}
	\|T_tf\|_1&=\int_{X_{t^{-1}}} \sigma^\beta_{\rm RN}(x, t)|f(\beta_t(x))|d\nu(x)=\int_{X_{t^{-1}}} |f(\beta_t(x))|d(\nu\circ\beta_t)(x)\\&=\int_{X_{t}} |f(x)|d\nu(x)=	\|f\|_1,
\end{align*}
for $f\in L^1(X_{t},\nu)$, and 
\begin{align*}
	T_sT_tf(x)&=\sigma^\beta_{\rm RN}(x, s)T_tf(\beta_s(x))\\&=\sigma^\beta_{\rm RN}(\beta_s(x), t)\sigma^\beta_{\rm RN}(x, s)f(\beta_t\beta_s(x))\\&=\sigma^\beta_{\rm RN}(x, ts)f(\beta_{ts}(x))=T_{ts}f(x),
\end{align*}
for $x\in X_{s^{-1}}\cap X_{s^{-1}t^{-1}}$, and $f\in L^1(X_{s},\nu)$.
For $f\in L^1(X,\nu)$, let $\mathfrak G(f)$ be the set of all finite convex combinations of the form,
$\sum_n c_nT_{t_n}(f|_{X_{t_n}}), $ regarded as an element of $L^1(X,\nu)$ via extending by zero, with $t_n\in G,$ and $J=J_f$ be the closed linear subspace of $L^1(X,\nu)$ spanned by vectors of the form $T_t(f|_{X_{t}})-f|_{X_{t^{-1}}}$, extended by zero. We claim that,
$${\rm dist}(0, \mathfrak G(f))=\|q(f)\|, \ \ (f\in L^1(X,\nu)),$$
where $q: L^1(X,\nu)\to L^1(X,\nu)/J$ is the canonical quotient map: let us denote the LHS and RHS of the above equality by $d$ and $d^{'}$, then clearly $d\geq d^{'}$, as $\mathfrak{G}(f)\subseteq f+J$. For the reverse inequality, we may assume that $d>0$. Since the map $t\mapsto \langle T_t(f|_{X_{t}}), g|_{X_{t^{-1}}}\rangle$ is bounded Borel measurable for $f\in L^1(X,\nu)$ and $g\in L^\infty(X,\nu)$, $\phi_{g,f}(t):=\langle T_t(f|_{X_{t}}), g|_{X_{t^{-1}}}\rangle$ defines an element in $L^\infty(G)$. Let $m$ be a left invariant mean on $L^\infty(G)$ and put,
$\langle \tilde g, f\rangle:=m(\phi_{g,f}).$
This associates to each $g\in L^\infty(X,\nu)$ an elements $\tilde g\in L^\infty(X,\nu)$ with $\|\tilde g\|_\infty\leq\|g\|_\infty.$ Since $d>0$, by a Hahn-Banach argument (or by applying \cite[Lemma 8.6.5]{r} to $B:=L^1(X,\nu)$ and $C:=\mathfrak G(f)$), there is $g\in L^\infty(X,\nu)$ with $\|g\|_\infty=1/d,$ and 
${\rm Re}(\langle  g, h\rangle)\geq 1,$ for $h\in\mathfrak G(f).$    
Then $\|\tilde g\|_\infty\leq 1/d$, and 
$$d\leq d{\rm Re}(\langle \tilde g, h\rangle)\leq d|\langle \tilde g, h\rangle|\leq d\|\tilde g\|_\infty\|h\|_1\leq\|h\|_1,$$
for $h\in\mathfrak G(f).$ 

Next, let us observe that $\langle\tilde g, h\rangle$ is independent of the choice of $h$: for $h=\sum_n c_nT_{t_n}(f|_{X_{t_n}}),$ 
\begin{align*}
	\phi_{g,h}(t)&= \langle T_t(h|_{X_{t}}), g|_{X_{t^{-1}}}\rangle\\&=\int_{X_{t^{-1}}}\sigma^\beta_{\rm RN}(x, t)h(\beta_t(x))g(x)d\nu(x)\\&=\sum_n c_n\int_{X_{t^{-1}}}\sigma^\beta_{\rm RN}(x, t)T_{t_n}(f|_{X_{t_n}})(\beta_t(x))g(x)d\nu(x)\\&=\sum_n c_n\int_{X_{t^{-1}}\cap \beta^{-1}_{t}(X_{{t^{-1}_n}})}\sigma^\beta_{\rm RN}(x, t)\sigma^\beta_{\rm RN}(\beta_t(x), t_n)f(\beta_{t_n}\beta_t(x))g(x)d\nu(x)\\&=\sum_n c_n\int_{X_{(t_nt)^{-1}}}\sigma^\beta_{\rm RN}(x, t_nt)f(\beta_{t_nt}(x))g(x)d\nu(x)\\&= \sum_n c_n \phi_{g,f}(t_nt),
\end{align*} 
where the equality before the last follows from the fact that
$$X_{t^{-1}}\cap \beta^{-1}_{t}(X_{{t^{-1}_n}})=\beta^{-1}_{t}(X_{t}\cap (X_{{t^{-1}_n}})\subseteq X_{(t_nt)^{-1}},$$
and the functions are extended by zero where undefined,
thus, 	
$$\langle \tilde g, h\rangle:=\sum_n c_n m(t^{-1}_n\cdot\phi_{g,f})=m(\phi_{g,f})\sum_n c_n =m(\phi_{g,f})=\langle \tilde g, f\rangle.$$
On the other hand, $\|h\|_1$ could be chosen arbitrarily closed to $d$. It follows that,
$$1={\rm Re}(\langle \tilde g, h\rangle)=|\langle \tilde g, h\rangle|,$$
and so, $\langle \tilde g, h\rangle=1$, for $h\in \mathfrak G(f)$. On the other hand, for a typical element $k:=T_t(f|_{X_{t}})-f|_{X_{t^{-1}}}\in J$, by a similar calculation as above we have,
$$
\phi_{g,k}(s)=\phi_{g,f}(ts)-\phi_{g,f}(s),\ \ (s\in G),
$$
thus, $$\langle \tilde g, k\rangle=m(t^{-1}\cdot\phi_{g,f})-m(\phi_{g,f})=0,$$
thus $\langle\tilde g, f+k\rangle=1$, for each $k\in J$, by linearity and continuity. But, $\|f+k\|_1$ could be chosen arbitrarily close to $d^{'}$, for a proper choice of $k$, and $\|\tilde g\|_\infty\leq 1/d$, thus, $d^{'}(1/d)\geq 1$, that is, $d\leq d^{'}$, and therefore, $d=d^{'}$, as claimed.   

Next, let $h\in L^1(X,\nu)_{+}$ be of norm one, then since the map,
$$t\mapsto \int_{X_{t^{-1}}}|h(x)-\sigma^\beta_{\rm RN}(x,t)h(\beta_t(x))|d\nu(x)$$ is continuous  vanishing at $e$, given $\varepsilon>0$, there is an open neighborhood $V$ of $e$ in $G$ such that,
$$\int_{X_{t^{-1}}}|h(x)-\sigma^\beta_{\rm RN}(x,t)h(\beta_t(x))|d\nu(x)<\varepsilon, \ \ (t\in V).$$ 
Similarly, there is an open neighborhood $W$ of $e$ in $G$ such that,
$$\int_{X_{t^{-1}}}|h(x)-\sigma^\al_{\rm RN}(x,t)h(\al_t(x))|d\nu(x)<\varepsilon, \ \ (t\in W).$$   
Put $U:=V\cap W$. Let $K\subseteq G$ be a compact subset and cover $K$ by finitely many translates of $U$, say $K=\bigcup_{n=1}^{N} s_nU$. We claim that there is a finite convex combination $T=\sum_n c_nT_{t_n}$ with $\|Th_i\|_1<\varepsilon,$ for $1\leq i\leq N$, where $h_i:=h-\sigma^\beta_{\rm RN}(\cdot,s_i)h\circ\beta_{s_i}$ on $X_{s^{-1}_i}$, and 0 elsewhere, for $1\leq i\leq N$. Here we use the convention that when $T$ is applied to a function, it is appropriately restricted to the corresponding domain, that is, $Tf:= \sum_n c_nT_{t_n}(f|_{X_{t_n}})$. We prove the claim by induction on $N$: if $N=1$, then since $h_1\in J_h$, we have $q(h_1)=0$, and so there is an element in  $\mathfrak G(h)$ which is $\varepsilon$-close to 0, which is a restatement of the claim for $N=1$. Now assume that the claim holds for $N-1$ and choose a finite convex combination $T^{'}:=\sum_i b_iT_{u_i}$ with $\|T^{'}h_k\|_1<\varepsilon$, for $1\leq k\leq N-1$. Put $g:=T^{'}h_N$. Then an element in $\mathfrak G(g)$ is a finite convex combination of the form,
$$(\sum_i d_jT_{t_j})(T^{'}h_N)=\big(\sum_{i,j} b_id_jT_{t_j}\big)T_{u_i}(h_N)=\sum_{i,j} b_id_jT_{u_it_j}(h_N),$$   
which is again a convex combination, thus $\mathfrak G(g)\subseteq \mathfrak G(h_N),$ therefore,
$${\rm dist}(0, \mathfrak G(g))\leq{\rm dist}(0, \mathfrak G(h_N))=\|q(h_N)\|=0,$$
where $q$ is the quotient map onto the quotient of $L^1(X,\nu)$ by $J_{h}$. Choose a convex combination $T^{''}$  as above with $\|T^{''}g\|_1<\varepsilon$. Then $T:=T^{''}T^{'}$ is again a convex combination as above with 
$$\|Th_i\|_1=\|T^{''}T^{'}h_i\|_1\leq \|T^{'}h_i\|_1<\varepsilon,\ \ (1\leq i\leq N-1),$$
and $\|Th_N\|_1=\|T^{''}T^{'}h_N\|_1=\|T^{''}g\|_1<\varepsilon,$ finishing the proof of the claim. Next, put $f:=Th$, This is a norm one positive element of $L^1(X,\nu)$ and, for each $s\in K$, there is $t\in U$ and $1\leq i\leq N$ with $s=s_it$. Define,
$$S_sf(x)=\sigma^\al_{\rm RN}(x,s)f(\al_s(x)),\ \ (x\in X_{s^{-1}}, f\in L^1(X_s,\nu)).$$  
Let us observe that $S_sT_t=T_tS_s$, for $s,t\in G$. For this, first observe that,
\begin{align*}
	\sigma^\beta_{\rm RN}(\al_s(x),t)\sigma^\al_{\rm RN}(x,s)d\nu(x)&=\sigma^\beta_{\rm RN}(\al_s(x),s)d\nu(\al_s(x))\\&=d\nu(\beta_t\al_s(x))\\&=d\nu(\al_s\beta_t(x))\\&=\sigma^\al_{\rm RN}(\beta_t(x),s)d\nu(\beta_t(x))\\&=\sigma^\al_{\rm RN}(\beta_t(x),s)\sigma^\beta_{\rm RN}(x,t)d\nu(x),  	
\end{align*}
that is, 
$$\sigma^\beta_{\rm RN}(\al_s(x),t)\sigma^\al_{\rm RN}(x,s)=\sigma^\al_{\rm RN}(\beta_t(x),s)\sigma^\beta_{\rm RN}(x,t), \ \ (x\in X_{s^{-1}}\cap X_{t^{-1}}).$$
Therefore, 
\begin{align*}
	T_tS_sf(x)&=\sigma^\beta_{\rm RN}(x,t)S_sf(\beta_t(x))\\&=\sigma^\beta_{\rm RN}(x,t)\sigma^\al_{\rm RN}(\beta_t(x),s)f(\al_t\beta_t(x))\\&=\sigma^\al_{\rm RN}(x,s)\sigma^\beta_{\rm RN}(\al_s(x),t)f(\beta_t\al_s(x))\\&=\sigma^\beta_{\rm RN}(x,t)T_tf(\al_s(x))\\&=S_sT_tf(x),
\end{align*} 
for $x\in X_{s^{-1}}\cap X_{t^{-1}}$ and $f\in L^1(X_{s^{-1}}\cap X_{t^{-1}},\nu)$. Back to the above calculations, since $T$ is a convex combination of operators of the form $T_t$ (followed by restriction on the corresponding $L^2$ subspace), $S_s$ commutes with $T$, therefore, for $f:=Th$ and $s=s_it$ as above,
\begin{align*}
	S_sf-f&=S_{s_i}S_tTh-Th=S_{s_i}TS_th-Th\\&=S_{s_i}T(
	S_th-h)+T(S_{s_i}h-h),
\end{align*} 
thus,
$$\|S_sf-f\|_1\leq \|S_{s_i}T\|\|
S_th-h\|+\|T\|\|S_{s_i}h-h\|<\|
S_th-h\|+\|S_{s_i}h-h\|<2\varepsilon,$$	
that is, $\al$ satisfies condition $(P_1)$, as required.
\end{proof}

In the next definition, the symmetric difference of sets is defined by, $$A\Delta B:=(A\backslash B)\cup(B\backslash A).$$
 
\begin{df}
We say that a Borel \pa $\al$ on a Borel space $(X,\nu)$ has {\it F{\o}lner property} if for each compact subset $K\subseteq G$ and $\varepsilon>0$, there is a Borel subset $F\subseteq X$ such that,
$$\nu\big(\al_t(F\cap X_{t^{-1}})\Delta (F\cap X_{t^{-1}})\big)<\varepsilon\nu(F\cap X_{t^{-1}}), \ \ (t\in K).$$
\end{df}

\begin{df}
Given a Borel \pa $\al$ on a Borel space $(X,\nu)$, 
two Borel subsets $C,D\subseteq X$ are said to be $\al$-{\it equi-decomposable}, writing $C\sim_\al D$, if there are $n\geq 1$, elements $t_1,\cdots, t_n\in G$, Borel subsets $C_1,\cdots, C_n\subseteq C$,  such that $C_i\subseteq X_{t^{-1}_i}$, for $1\leq i\leq n$, $C=\sqcup_{i=1}^n C_i$ and $D=\sqcup_{i=1}^n \al_{t_i}(C_i)$. A \pa  $\al$ is called {\it paradoxical} if there are disjoint Borel subsets $C,D\subseteq X$ with  $C\sim_\al X\sim_\al D$.   
\end{df}

\noindent Note that $C\sim_\al D$ implies $\mu(C)=\mu(D)$, for each invariant measure $\mu$.

\vspace{.3cm}
Let $G$ partially act by $\al$ on $X$ and let $S_\infty$ be the group of
permutations of $\mathbb N_0:=\mathbb N\cup\{0\}$. Then the group $\tilde G := G \times S_\infty$ partially acts canonically on
$\tilde X:= X\times \mathbb N_0$, say by $\tilde\al$. Given a Borel subset $F\subseteq \tilde X$, the {\it levels} of $F$ is the set of those $n\in\mathbb N$ such that $(x, n)\in F$, for some
$x\in X$.  If $F$ has only finitely many levels, we say that $F$ is {\it bounded}.
For a bounded Borel subset $F$ the equivalence class of $F$ with respect to $\tilde \al$ is called
the type of $F$ and is denoted by $[F]$. For a Borel subset $E\subseteq X$, we set $[E] := [E \times \{0\}]$.
Given bounded Borel subsets $A, B \subseteq \tilde X$, there is $k\geq 0$ such that
$B':= \{(b, n + k): (b, n)\in B\}$
does not meet $A$ (by boundedness). Define $[A] + [B] :=[A\cup B']$. 
This is a well defined operation turning
$S := \{[A] : A\subseteq \tilde X\ {\rm bounded} \}$
into a commutative semigroup (with the same argument as in the classical case; c.f.,  \cite[0.2.4]{ru}), called the {\it type semigroup} of $\al$.

\begin{lem} \label{bij}
If $C\sim_\al D$ there is a bijection $\phi: C\to D$ with $A\sim_\al \phi(A)$, for each Borel subset $A\subseteq C$.  
\end{lem}
\begin{proof}
Let  $C=\sqcup_{i=1}^n C_i$ and $D=\sqcup_{i=1}^n \al_{t_i}(C_i)$, with $C_i\subseteq X_{t^{-1}_i}$, for $1\leq i\leq n$. Consider the bijection, $$\phi_i: C_i\to \al_{t_i}(C_i);\ x\mapsto \al_{t_i}(x),$$
and put $\phi=\phi_i$ on $C_i$, for $1\leq i\leq n$. 
\end{proof}

Next, put $C\preceq_\al D$ if $C\sim_\al D_0$, for some Borel subset $D_0\subseteq D$. In this case, we write $[C]\leq [D]$. As in the classical case \cite[Theorem 0.1.9]{ru}, it follows from Lemma \ref{bij}, that the Cantor–Bernstein type theorem holds for \pas: for Borel subsets $C,D$, if
$C\preceq_\al D$ and $C\preceq_\al D$ , then $C\sim_\al D$.

\begin{thm} \label{Reiter}
	For a  Borel \pa $\al$ of a Borel (topological) group $G$ on a Borel measure space $(X,\nu)$, the following are equivalent:
	
	$(i)$ $\al$ is amenable on $X$ in the sense of Greenleaf,
	
	$(ii)$ There is an $\al$-invariant finitely additive probability measure $\mu$ on $X$. 
	
	$(iii)$ $\al$ is not paradoxical,
	
	$(iv)$ $\al$ has F{\o}lner property,	
	
	$(v)$ $\al$ satisfies Reiter condition $(P_1)$,
	
	$(vi)$ $\al$ satisfies Reiter condition $(P_p)$, for some $p\geq 1$,
	
	$(vii)$ $\al$ satisfies Reiter condition $(P_p)$, for all $p\geq 1$.

\end{thm}
\begin{proof}
	$(i)\Rightarrow (ii).$ Let $m$ be an $\al$-invariant mean on $B(X)$ and put $\mu(C):=m(\mathds{1}_C).$ 
	
	$(ii)\Rightarrow (iii).$ If  $\mu$ is an $\al$-invariant finitely additive probability measure on $X$ and $X$ is paradoxical with disjoint Borel subsets $C,D\subseteq X$ satisfying  $C\sim_\al X\sim_\al D$, then,
	$$1=\mu(X)\geq\mu(C\sqcup D)=\mu(C)+\mu(D)=\mu(X)+\mu(X)=2,$$
	a contradiction.
	
	$(iii)\Rightarrow (ii).$ 
	Suppose that $X$ is not $\al$-paradoxical. Let us observe  that in this case, 
	$(n + 1)[X]\nleq n[X]$, for all $n\in\mathbb N$, since otherwise, 
	$$2[X] + n[X] = (n + 1)[X] + [X] \leq n[X] + [X] = (n + 1)[X] \leq n[X],$$
	for some $n\geq 1$, and repeating this argument, 
	$$ n[X] \geq n[X] + n[X] = 2n[X]\geq n[X],$$ that is, $n[X]=2n[X]$. But then using K\"onig’s Theorem on bipartite graphs \cite[Theorem 0.2.5]{ru}, by an argument exactly as in \cite[Theorem 0.2.6]{ru}, one could conclude that $[X]=2[X]$, which is not possible, as $X$ is not $\al$-paradoxical.   Next, \cite[Theorem 0.2.9]{ru} on commutative monoids implies that there is
	an additive map $\mu: S \to [0, \infty]$ with  $\mu([X]) = 1$. We may regard $\mu$ as a  probability measure on $X$ via $\mu(A):=\mu([A])$, which is then finite additive by the definition of operation in $S$. Finally, since $[\al_t(A)] = [A]$, for $t\in G$ and $A\subseteq X_{t^{-1}}$, it follows that $\mu$ is $\al$-invariant.
	
	$(iv)\Rightarrow (v).$ For a compact subset $K\subseteq G$ and $\varepsilon>0$, let $F\subseteq X$ be the  $(K,\varepsilon)$-F{\o}lner set of positive finite measure. Put $f:=\frac{1}{\nu(F)}\mathds{1}_F$, then $f$ is a norm one positive function in $L^1(X,\nu)$ which fulfills Reiter condition $(P_1)$ for $(K,\varepsilon)$. 
	
	$(v)\Rightarrow (vi).$ Given compact subset $K\subset G$ and $\varepsilon>0$, let $f\in L^1(X,\nu)$ be the function which fulfills the corresponding condition $(P_1)$ for $(K,\varepsilon)$, and put $g:=f^{\frac{1}{p}}$.  Then, since, $$|a-b|^p\leq |a^p-b^p|, \ \ (a,b\geq 0, p\geq 1),$$ we get, 
	\begin{align*}
		\int_{X_{t^{-1}}} |g(x)-\sigma^{\frac{1}{p}}_{\rm RN}(x, t)g(\al_{t}(x))|^p &d\nu(x)\\&\leq \int_{X_{t^{-1}}} |f(x)-\sigma_{\rm RN}(x, t)f(\al_{t}(x))|d\nu(x),
	\end{align*}
that is,  $g\in L^p(X,\nu)$ fulfills the condition $(P_p)$ for $(K,\varepsilon)$. 
	
	$(vi)\Rightarrow (vii).$ Given compact subset $K\subset G$ and $\varepsilon>0$, let $g\in L^1(X,\nu)$ be the function which fulfills the corresponding condition $(P_p)$ for $(K,(\varepsilon/2p)^{p})$, and put $f:=g^p$.  Let $h_t:=\mathds{1}_{X_{t}}$ be the characteristic function of the Borel set $X_t$. Then, since, $$|a^p-b^p|\leq p|a-b|(a^{p-1}+b^{p-1}), \ \ (a,b\geq 0, p\geq 1),$$ by Holder inequality,  for $r:=p-1$ and $1/p+1/q=1$, we have the following inequalities, with the convention that wherever $\al_t(x)$ is not defined, we put $g(\al_{t}(x))=0$,  
	\begin{align*}
		&\int_{X_{t^{-1}}} |f(x)-\sigma_{\rm RN}(x, t)f(\al_{t}(x))|d\nu(x)\\&\ \ =\int_{X} h_{t^{-1}}(x)|f(x)-\sigma_{\rm RN}(x, t)f(\al_{t}(x))|d\nu(x)\\&\ \ =\int_{X} h^p_{t^{-1}}(x)|g^p(x)-\sigma_{\rm RN}(x, t)g^p(\al_{t}(x))|d\nu(x)\\&\ \ \leq p\int_{X} h_{t^{-1}}(x)|g(x)-\sigma^{\frac{1}{p}}_{\rm RN}(x, t)g(\al_{t}(x))|(h^{r}_{t^{-1}}(x)[g^{r}(x)\\&\ \ +\sigma^{\frac{1}{r}}_{\rm RN}(x, t)g^{r}(\al_{t}(x))]d\nu(x) \\&\ \ \leq p \|h_{t^{-1}}(g-\sigma^{\frac{1}{p}}_{\rm RN}(\cdot, t)(g\circ\al_t))\|_p\big(\|h^{r}_{t^{-1}}g^{r}\|_q+\|h^{r}_{t^{-1}}(\sigma^{\frac{1}{r}}_{\rm RN}(\cdot, t)(g\circ\al_t))^{r}\|_q\big)\\&\ \ <p(\varepsilon/2p)(1+1)=\varepsilon,
	\end{align*}
that is,  $f\in L^1(X,\nu)$ fulfills the condition $(P_1)$ for $(K,\varepsilon)$.

$(vii)\Rightarrow (ii).$ For compact subset $K\subseteq G$ and $\varepsilon>0$, let $\mathfrak F_{K,\varepsilon}$ be the set of  norm one functions $f\in L^1(X,\nu)_{+}$ with  $\sup_{t\in K}\int_{X_{t^{-1}}} |f(x)-f(\al_{t}(x))|d\nu(x)<\varepsilon$ and $\int_K\int_{X_{t^{-1}}} f(x)d\nu(x)dt>2\varepsilon$,  where the first integral is against a left Haar measure on $G$. Then,
$$0<\varepsilon<M_f:=\int_K\int_{X_{t^{-1}}} f(\al_t(x))d\nu(x)dt<\mu(K)\|f\|_1<\infty.$$
Let $\mu_{f,K,\varepsilon}$ be the probability finitely additive measure on $X$ defined  by,   $$\int_X\varphi d\mu_{f,K,\varepsilon}:= \frac{1}{M_f}\int_K\int_{X_{t^{-1}}} \varphi(x)f(\al_t(x))d\nu(x)dt,\ \ ( \varphi\in B(X)),$$
for   $K\subseteq G$ compact, $\varepsilon>0$, and $f\in \mathfrak F_{K,\varepsilon}$.  We claim that,
$$\int_{X_{s^{-1}}}\varphi d\mu_{f,K,\varepsilon}\circ\al_s\approx_{2\varepsilon} \int_{X_{s}}\varphi d\mu_{f,K,\varepsilon},\ \ (s\in K),$$
for each $f\in \mathfrak F_{K,\varepsilon}$. Indeed, let us for simplicity of notation assume that $M_f=1$, then for $\phi$ supported in $X_{s}$,  
\begin{align*}
\int_{X_{s^{-1}}}&\varphi d\mu_{f,K,\varepsilon}\circ\al_s=\int_K\int_{X_{t^{-1}}\cap {X_{s^{-1}}}} \varphi(\al_{s^{-1}}(x))f(\al_t(x))d\nu(x)dt\\&=\int_K\int_{\al_s(X_{t^{-1}}\cap {X_{s^{-1}}})} \varphi(x)f(\al_t(\al_{s}(x))d\nu\circ\al_s(x)dt
\\&=\int_K\int_{X_{s^{-1}t^{-1}}\cap {X_{s}}} \varphi(x)f(\al_{ts}(x))\sigma_{\rm RN}(x,s)d\nu(x)dt\\&\approx_\varepsilon\int_K\int_{X_{s^{-1}t^{-1}}\cap {X_{s}}} \varphi(x)f(x)\sigma^{-1}_{\rm RN}(x,ts)\sigma_{\rm RN}(x,s)d\nu(x)dt\\&=\int_K\int_{X_{s^{-1}t^{-1}}\cap {X_{s}}} \varphi(x)f(x)\sigma^{-1}_{\rm RN}(\al_s(x),t)d\nu(x)dt\\&\approx_\varepsilon\int_K\int_{X_{t^{-1}}\cap {X_{s}}} \varphi(x)f(\al_t(x))\sigma^{-1}_{\rm RN}(\al_s(x),t)\sigma_{\rm RN}(\al_s(x),t)d\nu(x)dt\\&=\int_{X_{s}}\varphi d\mu_{f,K,\varepsilon}.
\end{align*}  
Since the set of probability finitely additive measures on $X$ is weak$^*$-compact, the net $\{\mu_{f,K,\varepsilon}\}$ clusters to a probability finitely additive measure $\mu$ on $X$, which is then invariant by the above estimates.
 
	$(ii)\Rightarrow (i).$ Let $\mu$ be an invariant finitely additive probability measure on $X$. Let us define a functional $m$ on the subspace of $B(X)$ consisting of simple functions by $=m(\sum_i \lambda_i\mathds{1}_{E_i}):=\sum_i \lambda_i\mu(E_i).$ Then,
	$$\big|\sum_i \lambda_i\mu(E_i)\big|=\Big|\int_X(\sum_i \lambda_i\mathds{1}_{E_i})d\mu\Big|\leq \Big\|\sum_i \lambda_i\mathds{1}_{E_i}\Big\|_\infty,
$$
Thus $m$ extend to a mean on $B(X)$ and the extension is clearly $\al$-invariant.

	$(i)\Rightarrow (v).$ Let $m$ be an $\al$-invariant mean on $B(X)$, then there is a net $(f_i)$ of positive norm one functions in $L^1(X,\nu)$ such that $f_i\to m$, in the weak$^*$-topology of $B(X)^*$. For $f\in L^1(X, \nu)$, let $$(\delta_t*f)(x):=\sigma_{\rm RN}(x,t)f(\al_t(x)),\ \ (x\in X_{t^{-1}}),$$ 
	extended by zero to a measurable function on $X$. 
	
Next, let us observe that $L^1(X,\nu)$ could be regarded as an $L^1(G)$-module. For $x\in X$, consider the following measurable subsets of $G$,
$$G_x:=\{s\in G: x\in X_{s^{-1}}\},\ \  (x\in X).$$
 Define the action of $L^1(G)$ on $L^1(X,\nu)$ by the following ``convolution'' product against the Haar measure of $G$,
 $$(h*f)(x):=\int_{G_x} h(t)\sigma_{\rm RN}(x,t)f(\al_t(x))dt,\ \ (h\in L^1(G), f\in L^1(X,\nu)).$$
 Put,
 $$(L_th)(s):=h(t^{-1}s),\ \ (h\in L^1(G)).$$
 
Let $\mathds{1}_x$ and $\mathds{1}_t$  denote the characteristic functions of $G_x$ and $X_{t^{-1}}$, regarded as elements in $L^\infty(G)$ and $L^\infty(X,\nu)$, respectively. Then,
$$\delta_t*f(x)=\mathds{1}_t(x)\sigma_{\rm RN}(x,t)f(\al_t(x))=\mathds{1}_x(t)\sigma_{\rm RN}(x,t)f(\al_t(x)),$$
for $f\in L^1(X,\nu)$, with the convention that these terms are zero whenever $\la_t(x)$ is not defined. With this convention, we may write,
$$h*f=\int_G h(t)\mathds{1}_t(\cdot)(\delta_t*f)(\cdot)dt,$$
as an $L^1(X,\nu)$-valued Bochner integral. Now by linearity and continuity of left convolution by $\delta_t$ we have, 
\begin{align*}
	\delta_t*(h*f)&=\delta_t*\int_{G} h(s)\mathds{1}_s(\delta_s*f)ds
	=\int_{G} h(s)\delta_t*(\mathds{1}_s(\delta_s*f))ds\\&=\int_{G} h(s)\mathds{1}_{ts}(\cdot)\delta_t*\delta_s*fds=\int_{G} h(s)\mathds{1}_{ts}(\delta_{ts}*f)ds\\&=\int_{G} h(t^{-1}s)\mathds{1}_{s}(\delta_s*f)ds=(L_th)*f,
\end{align*}
for $t\in G$, $h\in L^1(G),$ and $f\in L^1(X,\nu)$, where the third equality follows from the fact that,
$X_{ts}=\{x\in X_{s^{-1}}: \al_s(x)\in X_{t^{-1}}\},$ and the fourth equality follows from the cocycle identity for
$\sigma_{\rm RN}$.

Let $h\in P^1(G)$, which is the set of positive norm one functions in $L^1(G)$. Since the mean $m$ is continuous, 
$$m(h*f)=\int_G h(t)(\delta_t*f)dt=\int_G h(t)m(\delta_t*f)dt=m(f)\int_G h(t)dt=m(f),$$
for each $f\in L^1(X,\nu)$. It follows that, $h*f_i-f_i\to 0$, weakly in $L^1(X,\nu)$.

Consider $\mathfrak X:=\prod_{h\in P^1(G)}L^1(X,\nu)$ with product topology. A continuous linear functional $\phi\in\mathfrak X^*$ is of the form,
$$\langle\phi, (f_h)\rangle=\sum_{i=1}^{n}\langle\phi_i, f_{h_i}\rangle,$$
for some $n\geq 1$, $h_1,\cdots,h_n\in P^1(G)$, and $\phi_1,\cdots, \phi_n\in L^1(X,\nu)^*$. This plus the above observation means that for $f^i:=(h*f_i-f_i)_{h\in P^1(G)}\in\mathfrak X$, $f^i\to 0$ weakly in $\mathfrak X$. Let,
$$\mathfrak K:=\{(h*f-f)_{h\in P^1(G)}\in\mathfrak X: f\in L^1(X,\nu)_+, \|f\|_1=1\}.$$
This is a convex subset of $\mathfrak X$, whose weak closure contains $0$, and so does its closure in the product topology (since $\mathfrak X$ is locally convex topologcal vector space in product topology). This simply means that there is a net $(\tilde f_i)$ of positive norm one functions in $L^1(X,\nu)$ such that,
$h*\tilde f_i-\tilde f_i\to 0,$
in norm, as $i\to \infty$, for each $h\in P^1(G)$. Of course, at this point, the convergence is not uniform in $h$. With a slight abuse of notation, we denote this new net again by $(f_i)$. 

Next, let us fix $\varepsilon>0$ and $K\subseteq G$ compact. Without loss of generality, we may assume that $e\in K$. Fix a positive norm one function $h\in P^1(G)$. Then the subset 
$L := \{L_th : t\in K\}$ is norm compact in $L^1(G)$ \cite[Proposition D.2.1]{ru} is compact, thus,
$$\sup_{g\in L}\|g*f_i-f_i\|_1\to 0,$$
as $i\to \infty$. Choose an index $i=i(\varepsilon) $ with $\sup_{g\in L}\|g*f_{i(\varepsilon)}-f_{i(\varepsilon)}\|_1\leq \varepsilon.$
Put $f_\varepsilon:=h*f_{i(\varepsilon)}\in L^1(X,\nu)$, then since $e\in K$, $\|f_\varepsilon-f_{i(\varepsilon)}\|_1\leq \varepsilon$, thus,
\begin{align*}
	\sup_{t\in K}\|\delta_t*f_{\varepsilon}-f_{\varepsilon}\|_1&=\sup_{t\in K}\|\delta_t*(h*f_{i(\varepsilon)})-f_{i(\varepsilon)}\|_1+\|f_\varepsilon-f_{i(\varepsilon)}\|_1\\&=\sup_{t\in K}\|L_th*f_{i(\varepsilon)}-f_{i(\varepsilon)}\|_1+\varepsilon\\&=\sup_{g\in L}\|g*f_{i(\varepsilon)}-f_{i(\varepsilon)}\|_1+\varepsilon\\&\leq 2\varepsilon,
\end{align*}
which says that $\al$ satisfies condition $(P_1)$.

$(v)\Rightarrow (iv)$. We prove this in three steps. In the first step, let us consider any positive norm one simple function $\psi$ on $X_{t^-1}$. We could always find a finite increasing sequence of positive scalars $\beta_i$ and mutually disjoint non-null measurable subsets $B_i\subseteq X_{t^{-1}}$ of finite measure such that $\psi=\sum_{j=1}^{n} \beta_i\mathds{1}_{B_j}$. Put $A_i=\bigcup_{j=1}^{n} B_j$. Then $A_i$'s form a finite decreasing sequence of non-null measurable subsets of $X_{t^{-1}}$ of finite measure, and for $\gamma_1:=\beta_1\nu(A_1)$, and $\gamma_i:=(\beta_i-\beta_{i-1})\nu(A_i)$ we could write $\psi$ as a convex combination $\psi=\sum_{i=1}^{n}\gamma_i\mathds{1}_{A_i}$. For arbitrary indices $1\leq i,j\leq n$, either $A_i\subseteq A_j$ or $A_j\subseteq A_i$, thus, either $\al_t(A_i)\backslash A_i\subseteq \al_t(A_j)$ or $A_j\backslash\al_t(A_j)\subseteq A_i$, and in both cases, $$\big(\al_t(A_i)\backslash A_i\big)\Delta \big(A_j\backslash\al_t(A_j)\big)=\emptyset.$$ Next, for each $i$, 
$$\nu(\al_t(A_i)\Delta A_i)=\big\|\mathds{1}_{\al_t(A_i)}-\mathds{1}_{A_i}\big\|_1,$$
and,
\begin{align*}
\int_{X}\big(\mathds{1}_{A_i}(x)-\delta_{t}*&\mathds{1}_{A_i}(x)\big)d\nu(x)=\int_{X}\mathds{1}_{A_i}(x)d\nu(x)-\int_X(\delta_{t}*\mathds{1}_{A_i})(x)d\nu(x)\\&=\int_{X}\mathds{1}_{A_i}(x)d\nu(x)-\int_X\sigma_{\rm RN}(x,t)\mathds{1}_{A_i}(\al_{t}(x))d\nu(x)\\&=\int_{X}\mathds{1}_{A_i}(x)d\nu(x)-\int_X\mathds{1}_{A_i}(\al_{t}(x))d\nu\circ\al_{t}(x)\\&=\nu(A_i)-\nu(\al_{t}(A_i)).
\end{align*}
Next, since $\big(\al_t(A_i)\backslash A_i\big)\Delta \big(A_j\backslash\al_t(A_j)\big)=\emptyset,$ for $i\neq j$, we could write,  
\begin{align*}
\sum_{j=1}^{n}\gamma_j&\frac{\nu(\al_t(A_j)\Delta A_j)}{\nu(A_j)}=\sum_{j=1}^{n}\frac{\gamma_j}{\nu(A_j)}\big\|\mathds{1}_{\al_t(A_j)}-\mathds{1}_{A_j}\big\|_1\\&=\big\|\sum_{j=1}^{n}\frac{\gamma_j}{\nu(A_j)}(\mathds{1}_{A_j}-\delta_{t}*\mathds{1}_{A_j})-\sum_{j=1}^{n}\frac{\gamma_j}{\nu(A_j)}(\mathds{1}_{A_j}-\delta_{t}*\mathds{1}_{A_j})\Big\|_1\\&=\big\|\sum_{j=1}^{n}\frac{\gamma_j}{\nu(A_j)}(\mathds{1}_{A_j}-\delta_{t}*\mathds{1}_{A_j})\big\|_1+\Big\|\sum_{j=1}^{n}\frac{\gamma_j}{\nu(A_j)}(\mathds{1}_{A_j}-\delta_{t}*\mathds{1}_{A_j})\Big\|_1\\&=\|\psi-\delta_t*\psi\|_1.
\end{align*}  
In the second step, let us observe that a weak version of the F{\o}lner condition follows: For each $\varepsilon, \delta>0$, and each compact subset $K\subseteq G$, there are measurable subsets $N\subseteq G$ and $F\subseteq X$ with $0<\nu(F)<\infty$, $m(N)<\delta$, and,
$$\nu\big(\al_t(F\cap X_{t^{-1}})\Delta (F\cap X_{t^{-1}}))<\varepsilon\nu(F\cap X_{t^{-1}}),\ \ (t\in K\backslash N),$$
where $m$ is a left Haar measure on $G$. If $m(K)=0$, we may choose $N=K$ and nothing is left to be proved. If $m(K)>0$, Put $\varepsilon_0 := \delta\varepsilon(\delta\varepsilon + 3m(K))^{-1}<1$ and by condition $(P_1)$ choose a norm one positive function $f\in L^1(X,\nu)$ with $\|f-\delta_t*f\|_1<\varepsilon_0,$ for $t\in K$. Choose a positive simple function $\varphi$ with $\|f-\varphi\|_1<\varepsilon_0$, then $\|\psi\|_1>1-\varepsilon_0>0$. Put $\psi:=\varphi/\|\varphi\|_1$. Then,
\begin{align*}
\|\psi-\delta_t*\psi\|_1&\leq \|\varphi-\delta_t*\varphi\|_1/\|\varphi\|_1\\&\leq 1/\|\varphi\|_1\big(\|\delta_t*(f-\varphi)\|_1+\|f-\delta_t*f\|_1+\|\varphi-f\|_1\big)\\&<\frac{3\varepsilon_0}{1-\varepsilon_0}=\delta\varepsilon/m(K).
\end{align*} 
Next, write $\psi$ as a convex combination $\psi=\sum_{i=1}^{n}\gamma_i\mathds{1}_{A_i}$, as in step one. It follows that,
$$\sum_{j=1}^{n}\gamma_j\frac{\nu(\al_t(A_j)\Delta A_j)}{\nu(A_j)}<\delta\varepsilon/m(K),$$
for $t\in K$. Integrating against the Haar measure, we get,
$$\sum_{j=1}^{n}\gamma_j\int_K\frac{\nu(\al_t(A_j)\Delta A_j)}{\nu(A_j)}dt<\delta\varepsilon,$$
and the LHS being a convex combination,
$\int_K\frac{\nu(\al_t(A_i)\Delta A_i)}{\nu(A_i)}dt<\delta\varepsilon,$
for some $i$, which proves the claim with $F\cap X_{t^{-1}}:=A_i$ and $N:=\{x\in K: \frac{\nu(\al_t(A_i)\Delta A_i)}{\nu(A_i)}\geq \varepsilon\}.$ Finally, in the third step, we conclude the F{\o}lner condition from this apparently weaker version. Given $\varepsilon>0$ and $K\subseteq G$ compact, choose a compact 
symmetric neighborhood $V$ of $e\in G$ containing $K$. Then, for $t\in K$, $tV\subseteq V^2\cap tV^2$, thus, 
$$m(V^2\cap tV^2)\geq m(tV)=m(V).$$
Applying the above weak F{\o}lner condition to $\delta:=\frac{1}{2}m(V)$, $\frac{1}{2}\varepsilon$ and $V^2$, we get  measurable subsets $N\subseteq V^2$, $F\subseteq X_{t^{-1}}$ with $m(N)<\delta$, $0<\nu(F)<\infty$, and, 
$$ \frac{\nu\big(\al_t(F)\Delta F\big)}{\nu(F)}<\frac{\varepsilon}{2}, \ \ (t\in V^2\backslash N).$$ 
Then,
\begin{align*}
2\delta&=m(V)\leq m(V^2\cap tV^2)\\&\leq m\big((V^2\backslash N)\cap t(V^2\backslash N)\big)+m(N)+m(tN)\\&< m\big((V^2\backslash N)\cap t(V^2\backslash N)\big)+2\delta,
\end{align*}
which implies that $(V^2\backslash N)\cap t(V^2\backslash N)$ is non-null, and in particular, non-empty. Now $V$ is symmetric and we may also choose $N$ to be symmetric. Choose $s,u\in V^2\backslash N$ with $t=s^{-1}u$ (using symmetry), then,
\begin{align*}
\|\delta_s*f\|_1&=\int_{X^{s^{-1}}}|f(\al_s(x))|\sigma_{\rm RN}(x,s)d\nu(x)\\&=\int_{X^{s^{-1}}}|f(\al_s(x))|d\nu\circ\al_s(x)=\int_{X^{s}}|f(x)|d\nu(x)=\|f\|_1,
\end{align*}
		where norms are calculated in the corresponding $L^1$-spaces, thus,
\begin{align*}
\nu(\al_t(F)\Delta F)&=\|\delta_t*\mathds{1}_F-\mathds{1}_F\|_1=\|\delta_{s^{-1}u}*\mathds{1}_F-\mathds{1}_F\|_1\\&=\|\delta_{s^{-1}}*\big(\delta_u*\mathds{1}_F-\delta_s*\mathds{1}_F\big)\|_1=\|\delta_u*\mathds{1}_F-\delta_s*\mathds{1}_F\|_1\\&\leq\|\delta_u*\mathds{1}_F-\mathds{1}_F\|_1+\|\delta_s*\mathds{1}_F-\mathds{1}_F\|_1<2(\varepsilon/2)=\varepsilon,
\end{align*} 
as required.
\end{proof}

We have the following result on amenability in the sense of Delaroche. In the next two results, $Y:=\frac{X}{\approx}$, as in Lemma \ref{lift1}.

\begin{lem}
Let $X$ be a partial $G$-space. Assume further that both $X$ and $G$ are second countable. If $Y$ is amenable in the sense of Delaroche, then so is $X$.   
\end{lem}
\begin{proof}
By Lemma \ref{am2}, we only need to observe that the Borel map $$\rho: \g:=X\rtimes_\al G\to Y\rtimes G:=\mathcal H; \ (x,s)\mapsto ([x],s),\ \ (s\in G, x\in X_{s^{-1}}),$$
is an $s$-bijective morphism. Since,
$$\tilde{\al}_s[x]=[\al_s(x)],\ \  (s\in G, x\in X_{s^{-1}}),$$
it follows that $\rho$ is a groupoid homomorphism. 

Next, for $x\in X$, $\g_x=\{(x,s): s\in G_x\}$, where $G_x:=\{s\in G: x\in X_{s^{-1}} \}$, whereas, $\mathcal H_{[x]}=\{([x],s): s\in \tilde G_x\}$, where  $\tilde G_x:=\{s\in G: [x]\in Y_{s^{-1}} \}$. Since $X_s=q^{-1}(Y_s)$, where $q:A\twoheadrightarrow Y$ is the quotient map, $\rho$ maps $\g_x$ onto $\mathcal H_{[x]}$. This map is also clearly injective.
\end{proof}

\begin{cor}
Let $X$ be a partially transitive partial $G$-space with enveloping $G$-space $\env{X}$. Assume further that both $X$ and $G$ are second countable. 

$(i)$ If $Y$ is amenable in the sense of Zimmer, then $\env{X}$ is amenable in the sense of Delaroche,

$(ii)$ If $\env{X}$ is amenable in the sense of Delaroche, then $X$ is amenable in the sense of Zimmer.    
\end{cor}
\begin{proof}
$(i)$  If $Y$ is amenable in the sense of Zimmer, then $\env{X}$ is amenable in the sense of Zimmer by Theorem \ref{main1}, and then it is amenable in the sense of Delaroche, by \cite[Theorem A]{aeg}.

$(ii)$  If $\env{X}$ is amenable in the sense of Delaroche, then it is also amenable in the sense of Zimmer by \cite[Theorem A]{aeg}, thus ${X}$ is amenable in the sense of Zimmer by Theorem \ref{main1}.
\end{proof}

\section{amenable partial representations}\label{sec:pr} 

\par In this section we explore amenability properties of partial representations of  topological (Borel) groups. 

Recall that a bounded operator $T$ on a Banach space $E$ is called a {\it partial isometry} if $T$ is a contraction and there is a contraction $S$  on $E$ (called a generalized inverse of $T$) satisfying $STS=S,\ TST=T$ \cite[Definition 4.1]{ma}. Note that the contractive generalized inverse of $T$ is not unique in general \cite[page 776]{ma}. Moreover, 
a bounded operator $T$ is a partial isometry iff $\ker(T)$ is complemented in $E$ such that $T$ acts isometrically on the complement, and there exists a norm one projection onto ${\rm Im}(T)$ \cite[Proposition 4.2]{ma}. The range of an iometry is a closed subspace, but we warn the reader that a (non surjective) isometry on $E$ is not necessarily a partial isometry in the above sense (unless there is a norm one projection onto its range). However,  an isometry with a generalized inverse is always a partial isometry.

For a Banach space $E$, we denote the semigroups of isometries and partial isometries on $E$ by ${\rm Iso}(E)$ and ${\rm PIso}(E)$, respectively.

\begin{df}\label{df:pr}
	A {\it partial representation} of a topological group $G$ in a Banach space $E$ is a map $\pi: G\to {\rm PIso}(E)$, such 
	that,
	\be
	\item the map $t\in G\mapsto \pi_t\xi\in E$, is norm continuous, for each $\xi\in E$ (that is, $\pi$ is continuous in the strong operator topology), 
	
	\item $\pi_e={\rm id}_E$,
	\item $\pi_t\pi_s\pi_{s^{-1}}=\pi_{ts}\pi_{s^{-1}}$,\ $\pi_{t^{-1}}\pi_t\pi_{s}=\pi_{t^{-1}}\pi_{ts}$, 
	\ee
	\noindent for all $s,t\in G$.
	\noindent If moreover, each $\pi_t$ is an isometry, we say that $\pi: G\to {\rm Iso}(E)$ is a  representation.
\end{df}

The same definition applies to Borel groups, where continuous is replaced by Borel.  
The following basic lemma (whose topological version is also valid, with almost verbatim proof) relates \pas to partial representations. The partial  representation $\kappa^\al$ in the next result is called the {\it partial Koopman representation} of the \pa $\al$.   

\begin{lem}\label{key}
Let $G$ be a Borel group,

$(i)$ to any partial action $\al$ of $G$ on a Borel measure space $(X,\nu)$ with quasi-invariant measure $\nu$, one could associate a partial representation $\kappa^\al$ in $E:=L^2(X,\nu)$,

$(ii)$ to any Borel partial representation $\pi$ of $G$ in a Banach space $E$ one could associate two Borel \pas $\al^\pi$ on ${\rm PHomeo}(E^*_1)$ with compact-open-topology and $\al_\pi$ on $E$ with norm-topology.  
\end{lem}
\begin{proof}
$(i)$ Let $\sig_{\rm RN}$ be the corresponding Radon-Nikodym cocycle of $\al$. Observe that the map,
$\kappa^\al_s: L^2(X_{s^{-1}},\nu)\to L^2(X_{s},\nu)$, defined by, $$ \kappa^\al_sf(x):=\sig_{\rm RN}^{\frac{1}{2}}(x, s^{-1})f(\al_{s^{-1}}(x)),\ \ (s\in G, x\in X_s),$$
is a surjective isometry, since,
\begin{align*}
\|\kappa^\al_sf\|_2^2&=\int_{X_s}|\kappa^\al_sf(x)|^2d\nu(x)=\int_{X_s}\sig_{\rm RN}(x, s^{-1})|f(\al_{s^{-1}}(x))|^2d\nu(x)\\&=\int_{X_s}|f(\al_{s^{-1}}(x))|^2d(\nu\circ\al_{s^{-1}})(x)=\int_{X_{s^{-1}}}|f(x)|^2d\nu(x)=\|f\|_2^2,
\end{align*}
for $s\in G$ and $f\in L^2(X_s,\nu)$. Thus, $\kappa^\al_s$ extends to a partial isometry on $E:=L^2(X,\nu)$, defined by $\kappa^\al_s(g)=\kappa^\al_s(f)^{\tilde{}}$, where $f$ is the restriction of $g$ to $X_s$ and $\tilde h$ is extension of $h$ by zero, for  $h\in L^2(U,\nu)$, and Borel subset $U\subseteq X$. Next, with the above notations, the assignment,
$$s\mapsto \kappa^\al_s(g)=\kappa^\al_s(f)^{\tilde{}}=\big(\sig_{\rm RN}^{\frac{1}{2}}(\cdot, s^{-1})f\circ\al_{s^{-1}}\big)^{\tilde{}}, $$ 
defines a Borel map. Also, as $X_e=X$, $\al_e={\rm id}_X$, and $\sig_{\rm RN}(e,\cdot)=1$, we get $\kappa^\al_e={\rm id}_E$. Finally, 
\begin{align*}
\kappa^\al_t\kappa^\al_s\kappa^\al_{s^{-1}}(g)(x)&=\kappa^\al_t\kappa^\al_s\kappa^\al_{s^{-1}}(f)(x)=\sig_{\rm RN}^{\frac{1}{2}}(x, t^{-1})\kappa^\al_s\kappa^\al_{s^{-1}}f(\al_{t^{-1}}(x))\\&=\sig_{\rm RN}^{\frac{1}{2}}(x,t^{-1})\sig_{\rm RN}^{\frac{1}{2}}(\al_{t^{-1}}(x), s^{-1})\kappa^\al_{s^{-1}}f(\al_{s^{-1}}\al_{t^{-1}}(x))\\&=\sig_{\rm RN}^{\frac{1}{2}}(x,t^{-1})\sig_{\rm RN}^{\frac{1}{2}}(\al_{t^{-1}}(x), s^{-1})\sig_{\rm RN}^{\frac{1}{2}}(x,s)f(\al_{s}\al_{s^{-1}}\al_{t^{-1}}(x))\\&=\sig_{\rm RN}^{\frac{1}{2}}(x,s^{-1}t^{-1})\sig_{\rm RN}^{\frac{1}{2}}(x,s)f(\al_{s}\al_{s^{-1}}\al_{t^{-1}}(x))\\&=\sig_{\rm RN}^{\frac{1}{2}}(x,s^{-1}t^{-1})\kappa^\al_{s^{-1}}f(\al_{(ts)^{-1}}(x))=\kappa^\al_{ts}\kappa^\al_{s^{-1}}(g)(x),
\end{align*}
for $g\in L^2(X,\nu)$, and $s,t\in G$, where $f$ is  the restriction of $g$ to $X_{s^{-1}}$, and we have identified certain functions with their extension by zero. The equality $\kappa^\al_{t^{-1}}\kappa^\al_t\kappa^\al_{s}=\kappa^\al_{t^{-1}}\kappa^\al_{ts}$ is proved similarly. Finally, since here $E$ is a Hilbert space, each $\kappa^\al_t$ is  automatically a partial isometry (since closed subspaces of Hilbert spaces are always range of a norm one projection). 

$(ii)$ Given a partial representation $\pi$ of $G$ in a Banach space $E$, ${\rm Im}(\pi_t)$ is a closed (complemented) subspace of $E$ with surjective isometry $\pi_t: {\rm Im}(\pi_{t^{-1}})\to {\rm Im}(\pi_t)$. Taking adjoint, we get a homeomorphism $\pi^*_t: {\rm Im}(\pi_{t})^*\to {\rm Im}(\pi_{t^{-1}})^*$. Consider $X:={\rm PHomeo}(E^*_1)$,  the space of partial homeomorphisms of the compact metric space $E^*_1$, with compact-open topology, and regard each $X_t:={\rm PHomeo}({\ker}(\pi_{t})^\perp_1)$ as a closed (and so Borel) subspace of $X$.  Define the Borel \pa action $\al^\pi$ by,
$$\al^\pi_t: X_{t^{-1}}\to X_t;\ \ x\mapsto \pi^*_tx\pi^*_{t^{-1}}.$$   
  Then clearly, $\al^\pi_e={\rm id}_X$ and $\al^\pi_{t^{-1}}=(\al^\pi_t)^{-1},$ also, 
  \begin{align*}
\al^\pi_t\al^\pi_s\al^\pi_{s^{-1}}(x)&=\pi^*_{t}\pi^*_{s}\pi^*_{{s^{-1}}}x\pi^*_{s}\pi^*_{{s^{-1}}}\pi^*_{{t^{-1}}}\\&=\pi^*_{ts}\pi^*_{{s^{-1}}}x\pi^*_{s}\pi^*_{{s^{-1}}t^{-1}}\\&=\al^\pi_{ts}\al^\pi_{s^{-1}}(x),
  \end{align*}   
for $x\in X_s$, and similarly, 
$\al^\pi_{t^{-1}}\al^\pi_t\al^\pi_{s}(x)=\al^\pi_{t^{-1}}\al^\pi_{ts}(x)$
for $x\in X_{s^{-1}}$. Thus, $\al^\pi$ is a partial action on $X$ by \cite[Proposition 4.5]{exl}.

Next, let us put $X=E$ and $X_t:=\pi_t\pi_{t^{-1}}(E)$, and let $(\al_{\pi})_t: X_{t^{-1}}\to X_t$ be the restriction of $\pi_t$ to $X_{t^{-1}}$. Then, $(\al_\pi)_e={\rm id}_X$  $(\al_\pi)_{t^{-1}}=(\al_\pi)^{-1}_t,$ and, 
\begin{align*}
	(\al_{\pi})_t(\al_{\pi})_s(\al_{\pi})_{s^{-1}}=\pi_t\pi_s\pi_{s^{-1}}=\pi_{ts}\pi_{s^{-1}}=(\al_{\pi})_{ts}(\al_{\pi})_{s^{-1}},\ \ 
\end{align*}   
on $X_s$, and similarly for the other relation. 
\end{proof}

Next, we adapt \cite[Definitions 1.14, 1.17]{dad} to the Borel (topological) setting. 

\begin{df}
$(a)$ Let $\rho$ be a Borel (continuous) representation of a Borel (topological) group $G$ on a Banach space $F$, let $E$ be another Banach space such that there are bounded linear maps $\iota: E\to F$ and $q: F\to E$,  satisfying, 

$(i)$ $q\circ\iota(x)=x$,

$(ii)$ $q(\rho_{t^{-1}}\iota q(\rho_{t}\iota q(\rho_{s}(\iota(x)))))=q(\rho_{t^{-1}}\iota q(\rho_{ts}(\iota(x)))),\newline 
\indent \ \ \ \ \  q(\rho_{t}\iota q(\rho_{s}\iota q(\rho_{s^{-1}}(\iota(x)))))=q(\rho_{ts}\iota q(\rho_{s^{-1}}(\iota(x)))),$

$(iii)$ $\iota(q(\rho_t(\iota(x))))=\rho_t(\iota(x))$,

$(iv)$ $\|q(\rho_{t}(\iota(x))\|\leq \|x\|,$

$(v)$ the map $t\mapsto q(\rho_t(\iota(x)))$ is Borel (continuous) from $G$ to $E$,

\noindent for $x\in E$, and $s,t\in G$, then we say that the partial representation $\pi$ defined by partial isometries $\pi_t:=q(\rho_t\iota(\cdot))$, is the {\it restriction} of $\rho$  to $E$ via $\iota$ and $q$, and write $\pi={\rm Res}_\iota^q(\rho)$. 

$(b)$ Let $\pi$ be a Borel (continuous) partial representation of a Borel (topological) group $G$ on a Banach space $E$, let $F$ be another Banach space and $\rho$ is a Borel (continuous)  representation of  $G$ on $F$ such that there are bounded linear maps $\iota: E\to F$ and $q: F\to E$,  with $\pi={\rm Res}_\iota^q(\rho)$, then we say that $\rho$ is the {\it induction} of $\pi$  to $F$ via $\iota$ and $q$, and write $\rho={\rm Ind}_\iota^q(\pi)$.  

$(c)$   
A {\it globalization} of a partial representation $\pi$ of a Borel (topological) group $G$ on a Banach space $E$ is a quadruple $(F, \rho, \iota, q)$, where $\rho={\rm Ind}_\iota^q(\pi)$ acting on $F$ has the following universal property: for every quadruple $(F^{'},\rho^{'}, \iota^{'}, q^{'})$ with  $\rho^{'}={\rm Ind}_{\iota^{'}}^{q^{'}}(\pi)$ acting on $F^{'}$,  there exists a unique  bounded linear map $\psi: F\to F^{'}$ satisfying, $\psi\circ\iota=\iota^{'}$, $q^{'}\circ\psi=q$, and $\psi\circ\rho_t=\rho^{'}_t\circ\psi$, for each $t\in G$.

$(d)$   
Two globalizations $(F, \rho, \iota, q)$ and $(F^{'},\rho^{'}, \iota^{'}, q^{'})$ of a partial representation $\pi$ are said to be {\it isomorphic} if there exists an isomorphism $\psi: F\to F^{'}$ of Banach spaces with $\psi\circ\iota=\iota^{'}$, $q^{'}\circ\psi=q$, and $\psi\circ\rho_t=\rho^{'}_t\circ\psi$, for each $t\in G$.  
\end{df}

The next result extends \cite[Theorem 1.18]{dad}, based on the ideas in \cite{zi1}. For the notion of vector measures, we refer the reader to the monograph \cite{du}.

\begin{thm}
Every partial representation $\pi$ of a Borel (topological) group $G$ on a Banach space $E$ has a unique (up to isomorphism) globalization to a Borel (continuous) representation of $G$ on a quotient space $F$ of the Banach space $M(G, E)$ of $E$-valued bounded Borel measures on $G$.
\end{thm}
\begin{proof}
Put $E_t := \pi_t\pi_{t^{-1}}(E)$, for $t\in G$, and consider the 
linear subspace of the algebraic tensor product $M(G)\odot E$, generated by the vectors of the form $$\delta_t\otimes x- \delta_s\otimes \pi_{s^{-1}t}(x),$$ for $s,t\in G$ and $x\in E_{t^{-1}s}$, and let $Z$ be its closure in $M(G,E)=M(G)\otimes_\gamma E$, where the right hand side is the   projective tensor product \cite{du}. 

The left multiplication on the first leg makes of
$M(G)\odot E$ a $G$-space, that is, we have an algebraic representation $\rho$ of $G$ in $M(G)\odot E$, defined by,  $$\rho_t(\sum_{i=1}^{n}\mu_i\otimes x_i):=\sum_{i=1}^{n}(t\cdot \mu_i)\otimes x_i,$$
where $t\cdot\mu(A)=\mu(t^{-1}A)$, for each Borel subset $A\subseteq G$ and $t\in G$. Let us observe that each $\rho_t$ is isometric in the projective norm: Given $z\in M(G)\odot E$, by \cite[Proposition 3.2]{df}, we have,
\begin{align*}
\|\rho_t(z)\|_\gamma&=\inf\{\sum_{n=1}^{N}\|\nu_n\|\|x_n\|: \rho_t(z)=\sum_{n=1}^{N}\nu_n\otimes x_n\}\\&=\inf\{\sum_{n=1}^{N}\|\rho_{t^{-1}}\nu_n\|\|x_n\|: z=\sum_{n=1}^{N}\nu_n\otimes x_n\}\\&=\inf\{\sum_{n=1}^{N}\|\nu_n\|\|x_n\|: z=\sum_{n=1}^{N}\nu_n\otimes x_n\}=\|z\|_\gamma,
\end{align*}
In particular, $\rho_t$ extends to a surjective  self isometry of $M(G)\otimes_\gamma E$. Now since,
\begin{align*}
\rho_t\big(\delta_u\otimes x- \delta_v\otimes \pi_{u^{-1}v}(x)\big)&= \delta_{su}\otimes x- \delta_{sv}\otimes \pi_{u^{-1}v}(x)\\&= \delta_{su}\otimes x- \delta_{sv}\otimes \pi_{{(su)}^{-1}(sv)}(x),
\end{align*}
$\rho_t(Z)\subseteq Z$, and changing $t$ to $t^{-1}$ we get that indeed, $\rho_t(Z)= Z$. Inparticular, $\rho_t$ lifts to a surjective  self isometry of $(M(G)\otimes_\gamma E)/Z=:F$.

Consider  the linear map $\iota: E \to F$ defined by $\iota(x):=(\delta_e\otimes x)+Z$, which is clearly bounded.  Also, put, $$q\big((\sum_{n=1}^{N}\nu_n\otimes x_n)+Z\big) :=\sum_{n=1}^{N} \int_{G_{x_n}} \pi_t(x_n)d\nu_n(t),$$ where $G_x:=\{t\in G: x\in E_{t^{-1}}\}$. To see this is well defined, note that for $x\in E_{t^{-1}}$, $y:=\pi_{s^{-1}t}(x)$, and $\delta_t\otimes x- \delta_s\otimes \pi_{s^{-1}t}(x)\in Z$,
$$q\big((\delta_t\otimes x- \delta_s\otimes \pi_{s^{-1}t}(x)+Z\big)=\int_{G_{x}} \pi_u(x)d\delta_t(u)-\int_{G_{y}} \pi_v(x_n)d\delta_s(u),$$
which is equal to, 
$\pi_t(x)-\pi_s\pi_{s^{-1}t}(x)=0$,
if $x\in E_{t^{-1}}$, and is equal to, $0-0=0$, otherwise. Next, for $z:=\sum_{n=1}^{N}\nu_n\otimes x_n\in M(G)\odot E,$ 
\begin{align*}
\|q(z+Z)\|&=\big\|\sum_{n=1}^{N} \int_{G_{x_n}} \pi_t(x_n)d\nu_n(t)\big\|\leq\sum_{n=1}^{N} \int_{G_{x_n}} \|\pi_t(x_n)\|d|\nu_n|(t)\\&\leq\sum_{n=1}^{N} \int_{G_{x_n}} \|x_n\|d|\nu_n|(t)\leq\sum_{n=1}^{N} \|x_n\|\|\nu_n\|,
\end{align*}
that is, $\|q(z+Z)\|\leq\|z\|_\gamma$. We may choose $z_0\in M(G)\odot E,$ with $\|z_0\|_\gamma<2\|z+Z\|$ and $z-z_0\in Z$, thus, 
$$\|q(z+Z)\|=\|q(z_0+Z)\|\leq \|z_0\|_\gamma<2\|z+Z\|,$$
thus $q$ has an extension to a bounded linear map $q: F\to E$. Now, 
$$q(\rho_t\iota(x))=q\big(\rho_t(\delta_e\otimes x)+Z\big)=q\big((\delta_t\otimes x)+Z\big)=\pi_t(x),$$ and
$$q(\iota(x))=q\big((\delta_e\otimes x)+Z\big)=\pi_e(x)=x,$$
for each $x\in E$, verifying condition $(i)$ in part $(a)$ of the above definition and that $\rho$ restricts to $\pi$. Next, condition $(ii)$ just says that $\rho$ is a representation, and condition $(iii)$ says that each $\pi_t$ is a partial isometry. Finally, $(v)$ says that $\pi$ is a Borel partial representation. 

It remains to prove the universal property for $\rho$. Given a quadruple $(F^{'},\rho^{'}, \iota^{'}, q^{'})$ with  $\rho^{'}={\rm Ind}_{\iota^{'}}^{q^{'}}(\pi)$ acting on $F^{'}$, 
Define,
$$\psi: M(G)\odot E\to F^{'}; \ \ \sum_{n=1}^{N}\nu_n\otimes x_n\mapsto \int_G\rho^{'}_t(\iota^{'}(x_n))d\nu_n(t),$$
which is bounded in the projective norm by an argument as above, and for $\delta_t\otimes x- \delta_s\otimes \pi_{s^{-1}t}(x)\in Z$, 
\begin{align*}
\psi\big(\delta_t\otimes x- \delta_s\otimes \pi_{s^{-1}t}(x)\big)&=\rho^{'}_t(\iota^{'}(x))-\rho^{'}_s(\iota^{'}(\pi_{s^{-1}t}(x)))\\&=\rho^{'}_t(\iota^{'}(x))-\rho^{'}_s\iota^{'}q^{'}\rho^{'}_{s^{-1}t}(\iota^{'}(x))\\&=\rho^{'}_t(\iota^{'}(x))-\rho^{'}_s\rho^{'}_{s^{-1}t}(\iota^{'}(x))=0,
\end{align*}
thus $\psi$ extends to a bounded linear map $\psi: F=(M(G)\otimes_\gamma E)/Z\to F^{'}$, satisfying, $$\psi(\iota(x))=\psi\big((\delta_e\otimes x)+Z\big)=\rho^{'}_e(\iota^{'}(x))=\iota^{'}(x),$$ for each $x\in E$, and 
\begin{align*}
q^{'}(\psi(z))&=q^{'}\big(\psi(\sum_{n=1}^{N}\nu_n\otimes x_n)\big)=q^{'}\big(\int_G\rho^{'}_t\iota^{'}(x_n))d\nu_n\big)\\&=\int_Gq^{'}\rho^{'}_t\iota^{'}(x_n))d\nu_n(t)=\int_{G_{x_n}} \pi_t(x_n)d\nu_n(t)\\&=q\big(\sum_{n=1}^{N}\nu_n\otimes x_n\big)=q(z),
\end{align*}
and 
\begin{align*}
\psi(\rho_t(z))&=	\psi\big(\rho_t(\sum_{n=1}^{N}\nu_n\otimes x_n)\big)=	\psi\big(\sum_{n=1}^{N}t\cdot \nu_n\otimes x_n\big)\\&=\int_G \rho^{'}_s\iota^{'}(x_n)d(t\cdot\nu_n)(s)=\int_G \rho^{'}_s\iota^{'}(x_n)d\nu_n(t^{-1}s)\\&=\int_G \rho^{'}_{ts}\iota^{'}(x_n)d\nu_n(s)=\rho^{'}_{t}\big(\int_G \rho^{'}_{s}\iota^{'}(x_n)d\nu_n(s)\big)\\&=\rho^{'}_{t}\big(\psi(\sum_{n=1}^{N}\nu_n\otimes x_n)\big)=\rho^{'}_t(\psi(z)),
\end{align*}
for each $z=\sum_{n=1}^{N}\nu_n\otimes x_n\in M(G)\odot E$.

Finally, to prove uniqueness, from $\psi(\iota(x))=\iota^{'}(x)$,  we get $\psi\big((\delta_e\otimes x)+Z\big)=\iota^{'}(x)$ which implies $\psi\big((\delta_t\otimes x)+Z\big)=\rho^{'}_t\iota^{'}(x)$, for $t\in G$ and $x\in E$. Thus, by continuity,
$$\psi\big((\nu\otimes x)+Z\big)=\int_G\psi\big((\delta_t\otimes x)+Z\big)d\nu(t)=\int_G\rho^{'}_t\iota^{'}(x)d\nu(t),$$
which shows that $\psi$ has to be the above map, by linearity and continuity. 

Up to here, we showed that $\pi$ on $E$ has a globalization $\rho$ on $F:=(M(G)\otimes_\gamma E)/Z$, whose uniqueness (up to isomorphism) follows from the universal property.   
\end{proof}

We denote the above globalization of $\pi$ by $\env{\pi}$ and call it the {\it enveloping representation} of $\pi$, and write $\env{E}:=(M(G)\otimes_\gamma E)/Z$. 

\vspace{.3cm}
The next definition extends \cite[Definition 1.1]{b}. Recall that for a partial representation $\pi$, the partial isometry $\pi_t$ and projection $p_t:=\pi_t\pi_{t^{-1}}$ have the same range. For a closed subspace $E_0\subseteq E$, $\mathbb B(E_0)$ could be identified with the cut-down of $\mathbb B(E)$ with projection onto $E_0$.

\begin{df}
A partial representation $\pi$ on a Banach space $E$ is called {\it amenable} in the sense of Bekka, if  there exists
a $\pi$-invariant mean $\phi$ on $\mathbb B(E)$,  that is a norm one linear functional with $\phi({\rm id}_E) = 1$ and
$\phi(\pi_{t^{-1}}T\pi_{t})=\phi(T),$  for $t\in G$ and $T\in{\mathbb B}(E_t)=p_t\mathbb B(E)p_t$, where $E_t:={\rm Im}(\pi_t)$, and $p_t=\pi_t\pi_{t^{-1}}$. 
\end{df}

Note that the $\pi$-invariance in the above definition could also be stated as the equality $\phi(\pi_{t^{-1}}T\pi_{t})=\phi(p_tTp_t),$  for $t\in G$ and $T\in{\mathbb B}(E)$.

\begin{prop} \label{Koopman}
$(a)$ For a  Borel (continuous) \pa $\al$ of a Borel (topological) geoup $G$ on a Borel measure space $(X,\nu)$, the following are equivalent:

$(i)$ $\al$ is amenable on $X$ in the sense of Greenleaf,

$(ii)$ the partial Koopman representation $\kappa^\al$ on $L^2(X,\nu)$ is amenable in the sense of Bekka.

 $(b)$ For a  partial representation $\pi$ on a Banach space $E$, if $\pi$ is amenable on $E$ in the sense of Bekka, the partial actions $\al_\pi$ on $X=E$ and $\al^\pi$ on $X={\rm PHomeo}(E^*_1)$ are amenable in the sense of Greenleaf.

\end{prop}
\begin{proof} First let us prove the equivalence in part $(a)$.
	
$(i)\Rightarrow (ii).$ If $\al$ is amenable in the sense of Greenleaf, then by Theorem \ref{Reiter}, there is a net $(f_i)$ of positive, norm one functions in $L^2(X,\nu)$ such that,
$$\sup_{t\in K}\int_{X_{t^{-1}}} |f_i(x)-\sigma^{\frac{1}{2}}_{\rm RN}(x,t)f_i(\al_{t}(x))|^2d\nu(x)\to 0, $$
as $i\to \infty$. Put, $$\phi_i(T):=\langle Tf_i, f_i\rangle, \ \ (T\in \mathbb B(L^2(X,\nu))).$$
Then $(\phi_i)$ is a net in the unit ball of $\mathbb B(L^2(X,\nu))^*$, which has a weak$^*$-cluster point $\phi$ in the same ball. We claim that $\phi$ is a $\pi$-invariant mean. Passing to a subnet, we may assume that $\phi_i\to \phi$ in weak$^*$-topology. Given $t\in G$, the above limit condition on the net $(f_i)$ implies that,
$$\|f_i|_{X_{t}}-\kappa^\al_{t^{-1}}(f_i|_{X_{t}})\|_2\to 0,$$
as $i\to\infty$, in $L^2(X_{t^{-1}}, \nu)$. For a bounded operator $T$ on  $L^2(X,\nu)$, by a triangle inequality argument we have,
\begin{align*}
\phi(\kappa^\al_{t}T\kappa^\al_{t^{-1}})&=\lim_i \langle \kappa^\al_{t}T\kappa^\al_{t^{-1}}(f_i|_{X_{t}}), f_i|_{X_{t}}\rangle=\lim_i \langle T\kappa^\al_{t^{-1}}(f_i|{X_{t}}), \kappa^\al_{t^{-1}}(f_i|_{X_{t}})\rangle\\&=\lim_i \langle T(f_i|_{X_{t}}), f_i|_{X_{t}}\rangle=\lim_i \langle TP_{t}(f_i), P_{t}(f_i)\rangle\\&=\lim_i \langle  P_{t}TP_{t}(f_i),f_i\rangle=\lim_i\phi_i(P_{t}TP_{t})=\phi(P_{t}TP_{t}),
\end{align*}  
where $P_t$ is the orthogonal projection onto $L^2(X_t,\nu)$.

$(ii)\Rightarrow (i).$ Let $\phi$ be a $\kappa^\al$-invariant mean on $\mathbb B(L^2(X,\nu))$. Let $t\in G$. To each $\varphi\in B(X)$ we could associate the multiplication operator $M_\varphi$ on $L^2(X_{t^{-1}},\nu)$, and for $f\in L^2(X_{t^{-1}},\nu)$ and $x\in X_{t^{-1}}$, 
\begin{align*}
\kappa^\al_{t^{-1}}M_\varphi\kappa^\al_tf(x)&=\sig^{\frac{1}{2}}_{\rm RN}(x,t)M_\varphi\kappa^\al_tf(\al_t(x))\\&=\sig^{\frac{1}{2}}_{\rm RN}(x,t)\varphi(\al_t(x))\kappa^\al_tf(\al_t(x))\\&=\sig^{\frac{1}{2}}_{\rm RN}(x,t)\sig^{\frac{1}{2}}_{\rm RN}(\al_t(x),t^{-1})\varphi(\al_t(x))f(\al_{t^{-1}}\al_t(x))\\&=\sig^{\frac{1}{2}}_{\rm RN}(x,e)\varphi(\al_t(x))f(x)\\&=\varphi(\al_t(x))f(x),
\end{align*} 
that is,
$$\kappa^\al_{t^{-1}}M_\varphi\kappa^\al_t=M_{{t^{-1}}\cdot\varphi}.$$  
Let us define $m(\varphi):=\phi(M_\varphi)$. Then, $$m({t^{-1}}\cdot\varphi)=\phi(\kappa^\al_{t^{-1}}M_\varphi\kappa^\al_t)=\phi(M_\varphi)=m(\varphi),$$
i.e., $m$ is an $\al$-invariant mean on $B(X)$, and so $\al$ is amenable in the sense of Greenleaf.

Let us prove the statements in part $(b)$.
Let $\phi$ be a $\pi$-invariant mean on $\mathbb B(E)$. Let $E_t:={\rm Im}(\pi_t)$ and $X_t:={\rm Im}(\pi_t\pi_{t^{-1}})$, for $t\in G$. To each $\varphi\in B(X_{t^{-1}})$ we associate the multiplication operator $M_\varphi: E_t\to E_t$, defined by, 
$$M_\varphi(x):=\varphi\big(\pi_{t^{-1}}(x)\big)x,\ \  (x\in E_{t}),$$ 
then for $x\in E_t$, we have,
\begin{align*}
	\pi_{t^{-1}}M_\varphi\pi_t(x)&=\pi_{t^{-1}}M_\varphi(\pi_t(x))\\&=\pi_{t^{-1}}\big(\varphi(\pi_t(\pi_{t^{-1}}(x)))\pi_t(x)\big)\\&=\varphi\big((\al_\pi)_{t^{-1}}(\pi_{t^{-1}}(x))\big)\pi_{t^{-1}}(\pi_t(x))\\&=(t\cdot\varphi)\big(\pi_{t^{-1}}(x)\big)x\\&=M_{t\cdot\varphi}(x),
\end{align*} 
that is,
$$\pi_{t^{-1}}M_\varphi\pi_t=M_{{t}\cdot\varphi}.$$  
Let us define $m(\varphi):=\phi(M_\varphi)$. Then, $$m({t}\cdot\varphi)=\phi(\pi_{t^{-1}}M_\varphi\pi_t)=\phi(M_\varphi)=m(\varphi),$$
i.e., $m$ is an $\al_\pi$-invariant mean on $B(X)$, and so $\al_\pi$ is amenable in the sense of Greenleaf.

Next, let $\phi$ be a $\pi$-invariant mean on $\mathbb B(E)$. Let, $$X_t:={\rm PHomeo}(\ker(\pi_t)^\perp_1)={\rm PHomeo}({\rm Im}(\pi^*_t)_1)={\rm PHomeo}((E^*_t)_1),$$ where $E^*_t:=(E_{t^{-1}})^*$, for $t\in G$. To each $x\in E_{t^{-1}}$ we associate $x_t\in X_t$, defined by,
$$x_t: (E^*_t)_1\to (E^*_t)_1; \ \ y\mapsto \langle x,y\rangle y, \ \ (y\in (E^*_t)_1).$$ 
Given $\varphi\in B(X_{t})$, we define  the corresponding multiplication operator by, 
 $$M_\varphi: E_{t^{-1}}\to E_{t^{-1}}; \ \ M_\varphi(x):=\varphi(x_t)x,\ \  (x\in E_{t^{-1}}),$$ 
 then for $x\in E_{t^{-1}}$ and $y\in (E^*_t)_1$, we have,
 \begin{align*}
 \al^\pi_{t^{-1}}(x_{t^{-1}})(y)&= \pi^*_{t^{-1}}x_{t^{-1}}\pi^*_t(y)=\pi^*_{t^{-1}}(x_{t^{-1}}(\pi^*_t(y)))\\&=\pi^*_{t^{-1}}\big(\langle x,\pi^*_t(y)\rangle\pi^*_t(y)\big)=\langle x,\pi^*_t(y)\rangle y\\&=\langle \pi_t(x),y\rangle y=\pi_t(x)_t(y), 
 \end{align*}
thus,
  \begin{align*}
 	\pi_{t^{-1}}M_\varphi\pi_t(x)&=\pi_{t^{-1}}M_\varphi(\pi_t(x))=\pi_{t^{-1}}\big(\varphi(\pi_t(x)_t)\pi_t(x)\big)\\&=\varphi( \al^\pi_{t^{-1}}(x_{t^{-1}}))x=(t\cdot\varphi)\big(x_{t^{-1}}\big)x=M_{t\cdot\varphi}(x),
 \end{align*} 
 that is,
 $$\pi_{t^{-1}}M_\varphi\pi_t=M_{{t}\cdot\varphi}.$$  
 Let us define $m(\varphi):=\phi(M_\varphi)$. Then, $$m({t}\cdot\varphi)=\phi(\pi_{t^{-1}}M_\varphi\pi_t)=\phi(M_\varphi)=m(\varphi),$$
 i.e., $m$ is a $\al^\pi$-invariant mean on $B(X)$, and so $\al^\pi$ is amenable in the sense of Greenleaf.
\end{proof}

\begin{rk}
As for the converse implications in part $(b)$, if $\pi$ is a partial representation in a {\it Hilbert} space $E$ and $\al_\pi$ is amenable in the sense of Greenleaf, by Theorem \ref{Reiter}, there is an $\al_\pi$-invariant finitely additive probability measure $\mu$ on $X:=E$. Let  $X_t=E_t={\rm Im}(\pi_t)={\rm Im}(\pi_{t^{-1}}\pi_t)$. Also, there an $\al_\pi$-invariant mean $m$ on $B(X)$. For a bounded Borel function $\varphi\in B(X)$ we use the notation $m_x(\varphi(x)):=m(\varphi)$, to show the steps of calculations easier. We set, 
$$\phi(T):=m_x\Big(\int_{X} \langle Tx,y\rangle d\mu(y)\Big),\ \ (x\in X_t, T\in \mathbb B(E)).$$
Then,
$\phi({\rm id}_E)=m_x\big(\int_X 1d\mu\big).$
If we could guarantee that the above value is {\it non zero}, then after multiplying with an appropriate constant, we may assume that $\phi({\rm id}_E)=1$. The rest is easy, since,
\begin{align*}
	\phi(\pi_{t^{-1}}T\pi_{t})&=m_x\Big(\int_{X} \langle \pi_{t^{-1}}T\pi_{t}(x),y\rangle d\mu(y)\Big)\\&=m_x\Big(\int_{X} \langle  T\pi_{t}(x),\pi_{t}(y)\rangle d\mu(y)\Big)\\&=m_x\Big(\int_{X} \langle  T(\al_\pi)_{t}(x),(\al_\pi)_{t}(y)\rangle d\mu(y)\Big)\\&=m_x\Big(\int_{X} \langle Tx,y\rangle d\mu(y)\Big)=\phi(T),
\end{align*}         
for $T\in\mathbb B(E_t)$, as both $m$ and $\mu$ are $\al_\pi$-invariant. 

A similar argument could be designed when $E$ is a {\it Banach} space and the partial actions $\al_\pi$ on $X=E$ and $\al_{\bar\pi}$ on $E^*$ are amenable in the sense of Greenleaf: we just need a slight modification of the above proof as follows. Let $X:=E$ and $Y:=E^*$, choose  an $\al_\pi$-invariant mean on $B(X)$ and an $\al_\pi$-invariant finitely additive probability measure $\mu$ on $Y$. Set, 
$$\phi(T):=m_x\Big(\int_{Y} \langle Tx,y\rangle d\mu(y)\Big),\ \ (x\in X_t, T\in \mathbb B(E)).$$
Now if one could guarantee that $\phi({\rm id}_E)$ is {\it non zero}, the rest of the proof goes as above.

As for the converse implication for $\al^\pi$, let us consider the case that $G$ is {\it compact}. Let $X={\rm PHomeo}(E^*_1)$. By definition, there an $\al^\pi$-invariant mean $m$ on $B(X)$. Fix non zero elements $y\in E^*$ and $z\in E$. Set, 
$$\phi(T):=m_x\Big(\int_G\int_G \langle x\circ T^*(\pi^*_s(y)), \pi_u(z)\rangle dsdu\Big),\ \ (x\in X, T\in \mathbb B(E)),$$
where the integrals are against a normalized Haar measure on $G$. Then, 
$\phi({\rm id}_E)=m_x\big(\int_G\int_G 1dsdu).$
Again, if we could guarantee that for some choice of $y$ and $z$, the above value is {\it non zero}, we may let it be equal to one (after appropriate scaling), and for each $t\in G$ and $T\in\mathbb B(E)$, 
\begin{align*}
	\phi(\pi_{t^{-1}}T\pi_{t})&=m_x\Big(\int_G\int_G \langle x \pi^*_{t}T^*\pi^*_{t^{-1}}(\pi^*_s(y)), \pi_u(z)\rangle dsdu\Big)\\&=m_x\Big(\int_G\int_G \langle \pi^*_{t}\pi^*_{t^{-1}}x \pi^*_{t}T^*\pi^*_{t^{-1}}(\pi^*_s(y)), \pi_u(z)\rangle dsdu\Big)\\&=m_x\Big(\int_G\int_G \langle \al^\pi_{t^{-1}}(x) T^*(\pi^*_{st^{-1}}(y)), \pi_{tu}(z)\rangle dsdu\Big)\\&=m_x\Big(\int_G\int_G \langle \al^\pi_{t^{-1}}(x) T^*(\pi^*_{s}(y)), \pi_{u}(z)\rangle dsdu\Big)\\&=m_x\Big(\int_G\int_G \langle x\circ T^*(\pi^*_{s}(y)), \pi_{u}(z)\rangle dsdu\Big)=\phi(T),
\end{align*}         
as $m$ is $\al_\pi$-invariant, and the Haar measure on $G$ is both right and left translation invariant.

\end{rk}

We say that a unital Banach algebra $A$ is {\it tracial} if there is a trace on $A$, that is, a bounded (unital) linear functional $\Phi$ satisfying $\Phi(ab)=\Phi(ba)$, for $a,b\in A$. A Banach space $E$ is said to be tracial if the unital Banach algebra $\mathbb B(E)$ is tracial.

\begin{prop}
For a locally compact group $G$, the following are equivalent:

$(i)$ $G$ is amenable,

$(ii)$ every partial representation of $G$ in a tracial Banach space $E$ is amenable in the sense of Bekka,

$(iii)$ every partial representation of $G$ in a Hilbert space $H$ is amenable in the sense of Bekka,

$(iv)$ every Borel (continuous) \pa of $G$ on a standard Borel (topological) space $X$ is amenable in the sense of Greenleaf. 
\end{prop}
\begin{proof}
$(i)\Rightarrow (ii).$ Let $m$ be a left invariant mean on $L^\infty(G)$, that is, 
$m(t\cdot\varphi)=m(\varphi)$, for $t.\varphi(s):=\varphi(t^{-1}s),$ $s,t\in G$, $\varphi\in L^\infty(G)$. For a  partial representation of $G$ in a tracial Banach space $E$ with trace $\Phi$ on $\mathbb B(E)$, we associate to each $T\in \mathbb B(E)$ a bounded continuous function $\varphi_T\in C_b(G)$, defined by $\varphi_T(t):=\Phi(\pi_{t^{-1}}T\pi_{t})$. Let $\phi$ be the mean on $\mathbb B(E)$ defined by $\phi(T):=m(\varphi_T)$. Then, for $t\in G$, $T\in \mathbb B(E, E_t)$ and $S:=\pi_{t^{-1}}T\pi_{t}$, since the projection $p_t:=\pi_{t}\pi_{t^{-1}}$ acts as identity on the range of $T$, we have $p_tT=T$, thus, as $\Phi$ is a trace,
\begin{align*}
\varphi_S(s)&=\Phi(\pi_{s^{-1}}S\pi_{s})=\Phi(\pi_{s^{-1}}\pi_{t^{-1}}T\pi_{t}\pi_{s})\\&=\Phi(\pi_{s^{-1}}T\pi_{t^{-1}}\pi_{t}\pi_{s})=\Phi(\pi_{s^{-1}}T\pi_{t^{-1}}\pi_{ts})\\&=\Phi(\pi_{s^{-1}}p_tT\pi_{t^{-1}}\pi_{ts})=\Phi(\pi_{s^{-1}}\pi_{t}\pi_{t^{-1}}T\pi_{t^{-1}}\pi_{ts})\\&=\Phi(\pi_{s^{-1}}\pi_{t^{-1}}\pi_{t}T\pi_{t^{-1}}\pi_{ts})=\Phi(\pi_{s^{-1}t^{-1}}\pi_{t}T\pi_{t^{-1}}\pi_{ts})\\&=\Phi(\pi_{s^{-1}t^{-1}}\pi_{t}\pi_{t^{-1}}T\pi_{ts})=\Phi(\pi_{s^{-1}t^{-1}}p_tT\pi_{ts})\\&=\Phi(\pi_{s^{-1}t^{-1}}T\pi_{ts})=\varphi_T(ts),
\end{align*}
for each $s,t\in G$, that is, $\varphi_S=t^{-1}\cdot \varphi_T$, therefore, $$\phi((\pi_{t^{-1}}T\pi_{t})=m(\varphi_S)=m(t^{-1}\cdot \varphi_T)=m(\varphi_T)=\phi(T),$$
i.e., $\phi$ is a $\pi$-invariant mean on $\mathbb B(E)$.

$(ii)\Rightarrow (iii).$ A Hilbert space $H$ is always tracial, as for a trace class operator $S\in L^1(H)$, $\Phi_S(T)=Tr(ST)$ is  tracial on $\mathbb B(H)$, where $Tr$ is the canonical (possibly infinite) Trace of $\mathbb B(H)$.

$(iii)\Rightarrow (iv).$ For an  Borel \pa $\al$ of $G$ on a Borel space $(X,\nu)$  the corresponding Koopman representation $\kappa^\al$ is amenable in the sense of Bekka by $(iii)$, hence $\al$ is amenable in the sense of Greenleaf, by Proposition \ref{Koopman}.   
\end{proof}

Next, we study the permanence properties of amenability of partial representations in the sense of Bekka. For the rest of this section, amenability of partial representations is in the sense of Bekka.

\begin{df}
A partial {\it subrepresentation} of a partial representation $(\pi, E)$ is a pair $(\pi^F, F)$ consisting of a Banach subspace $F\leq E$ which is invariant under each $\pi_t$ and the restriction $\pi^F$ of $\pi$ to $F$. 

When $E=\oplus_i E_i$ such that $E_i$ is $\pi$-invariant and for $\pi_i:=\pi^{E_i}$, we have $\pi_t(\oplus_i x_i)=\oplus_i \pi_i(x_i)$, we write $\pi=\oplus_i \pi_i$, and say that $\pi$ is a {\it direct sum} of partial representations $\pi_i$.  
\end{df} 

The next lemma is an immediate consequence of the definition (c.f., \cite[Remark 1.2, Theorem 1.3]{b}).  

\begin{lem}
Let $\pi$ be a partial representation of a topological (Borel) group $G$ on a Banach space $E$.

$(i)$ If $H$ is a closed subgroup of $G$ and $\pi$ is amenable on $G$, so is its restriction to $H$, 

$(ii)$ if $H$ is
normal in $G$ and $H\subseteq \ker(\pi)$, then $\pi$ is amenable on $G$ if and only if $\dot\pi$ is  amenable on $G/H$, where $\dot\pi_{\dot t}(x):=\pi_t(x)$, for $\dot t:=tH$, $x\in E_{t^{-1}}=:E_{\dot t^{-1}}$,

$(iii)$ if $\pi$ is amenable in $E$, so is its conjugate representation $\bar\pi$ on $E^*$,

$(iv)$ $\pi$ is
amenable on $G$ if and only if $\pi$ is amenable on its discretisation $G_d$,
 
$(v)$ when $E$ is a Hilbert space, $\pi$ is amenable if and only if there exists a ad$_{\pi}$-invariant 
non-trivial bounded functional on $\mathbb B(E)$.

$(vi)$ if $\pi$ has an amenable partial subrepresentation whose invariant subspace is complemented in $E$, then $\pi$ is amenable.

$(vii)$ if a finite direct sum $\pi=\pi_1\oplus\cdots\oplus \pi_n$ of partial representations is
amenable, then $\pi_i$ is amenable, for some $1\leq i\leq n$.
\end{lem}

Following Bekka \cite{b}, we define an analog of the Reiter condition for partial representations on {\it Hilbert} spaces. 
Let $E$ be a Hilbert space. Let $L^p(E)$ denote the Schatten $p$-class on $E$, namely,
$$L^p(E):=\{T\in \mathbb B(E): {\rm Tr}(|T|^p)<\infty\},$$
for $1\leq p<\infty$, and put $L^\infty(E):=\mathbb B(E)$. We define the Schatten $p$-norm by $\|T\|_p:={\rm Tr}(|T|^p)^{\frac{1}{p}}$.

The algebra $L^1(E)$ of trace-class operators is 
a Banach $L^l(G)$-module via the action,
$$f\cdot T := \int_G f(t)\pi_tT\pi_{t^{-1}}dt,\ \  (f\in L^1(G), T\in L^1(E)).$$ Recall that $P^1(G)$ consists of norm one positive functions in $L^1(G)$. 

\begin{df} \label{rc}
We say that a partial representation $\pi$ on a Hilbert space $E$ satisfies {\it Reiter condition} $(P_p)$, for $1\leq p<\infty$, if for each $\varepsilon > 0$ and $K\subseteq  G$ compact,  there exists a positive norm one operator $T\in L^p(E)$ with $\|\pi_tT\pi_{t^{-1}}-T\|_p<\varepsilon$, for $t\in K$.
\end{df}

Note that when $E$ is a Banach space, it does not make sense to ask for the existence of a ``positive'' operator $T$ satisfying the above condition. We say that a partial representation $\pi$ on a {\it Banach} space $E$ satisfies {\it Reiter condition} $(\tilde P_p)$, for $1\leq p<\infty$, if for each $\varepsilon > 0$ and $K\subseteq  G$ compact,  there exists a  norm one operator $T\in L^p(E)$ with $\|\pi_tT\pi_{t^{-1}}-T\|_p<\varepsilon$, for $t\in K$. For global representations on a Hilbert space, conditions $(P_2)$ and $(\tilde P_2)$ are known to be equivalent (by Powers-St{\o}rmer inequality). In the case of partial representations on Hilbert spaces, $(P_2)$ might be stronger than $(\tilde P_2)$. 

\begin{prop} \label{Reiter2}
Let $\pi$ be a partial representation on a Hilbert space $E$. Then the following are equivalent.

$(i)$ $\pi$ is amenable in the sense of Bekka,

$(ii)$ there is a net $(S_i)$ of norm one, positive operators in $L^1(E)$ satisfying,  $$\|f\cdot S_i-S_i\|_1\to 0,\ \ (f\in P^1(G)),$$

$(iii)$ there is a net $(S_i)$ of norm one, positive operators in $L^1(E)$ satisfying,  $$\|\pi_tS_i\pi_{t^{-1}}-S_i\|_1\to 0,\ \ (t\in G),$$

$(iv)$ $\pi$ satisfies Reiter condition $(P_1)$.
\end{prop}
\begin{proof}
$(i)\Rightarrow (ii)$.  Let $\phi$ be an invariant mean on $\mathbb B(E)$. We have,
\begin{align*}
p_{t}(f\cdot T)p_t&=\int_G f(s)p_t\pi_{s^{-1}}T\pi_sp_tds=\int_G f(s)\pi_t\pi_{t^{-1}}\pi_sT\pi_{s^{-1}}\pi_t\pi_{t^{-1}}ds\\&=\int_G f(s)\pi_t\pi_{t^{-1}s}T\pi_{s^{-1}t}\pi_{t^{-1}}ds=\int_G f(ts)\pi_t\pi_{s}T\pi_{s^{-1}}\pi_{t^{-1}}ds\\&=\pi_t(L_{t^{-1}}f\cdot T)\pi_{t^{-1}},
\end{align*} 
thus, $\phi(f\cdot T)=\phi(L_{t^{-1}}f\cdot T),$ for $f\in P^1(G)$ and $T\in \mathbb B(E_t)$. Put $m_T(f)=\phi(f\cdot T)$, then $m_T(L_{t^{-1}})=m_T(f)$, for $f\in P^1(G)$, therefore, $m_T(f)=c(T)\int_G fdt,$ for a constant $c(T)$ depending only on $T$. Next, for $f,g\in L^1(G)$, 
\begin{align*}
\phi(g\cdot (f\cdot T))&=\int_G\int_G g(u)f(s)\phi(\pi_u\pi_sT\pi_{s^{-1}}\pi_{u^{-1}})duds\\&=\int_G\int_G g(u)f(s)\phi(\pi_{u^{-1}}\pi_u\pi_sT\pi_{s^{-1}}\pi_{u^{-1}}\pi_u)duds\\&=\int_G\int_G g(u)f(s)\phi(\pi_{u^{-1}}\pi_{us}T\pi_{s^{-1}u^{-1}}\pi_{u})dsdu\\&=\int_G\int_G g(u)f(u^{-1}s)\phi(\pi_{u^{-1}}\pi_{s}T\pi_{s^{-1}}\pi_{u})dsdu\\&=\int_G g*f(s)\phi(\pi_{s}T\pi_{s^{-1}})ds=\phi((g*f)\cdot T).
\end{align*}
If $T\in L^1(G)\cdot L^1(E)$, then as in \cite[Lemma 3.2]{b}, the map $t\mapsto \pi_tT\pi_{t^{-1}}$ is again into $L^1(G)\cdot L^1(E)$, and is norm continuous. Moreover, for $f\in P^1(G)$, $T=g\cdot S$ in this space, and approximate identity $(e_i)$ of $L^1(G)$ inside $P^1(G)$,   
\begin{align*}
\phi(f\cdot T)&=\lim_i\phi((f*e_i)\cdot T)=c(T)\lim_i\int_G(f*e_i)dt\\&=c(T)=c(g\cdot S)=\lim_i\phi((g*e_i)\cdot S)\\&=\phi(g\cdot S)=\phi(T).
\end{align*}
Next, by a Hahn-Banach extension argument, the set of norm one, positive operators in $L^1(E)$ is weak$^*$-dense in the state space of $\mathbb B(E)$ (c.f., \cite[page 391]{b}). Choose a net $(T_j)$ of norm one positive operators in $L^1(E)$ with $T_j\to \phi$ in weak$^*$-topology. By an argument similar the one used in the proof of the implication $(i)\Rightarrow(v)$ in Theorem \ref{Reiter}, there is another net $(S_j)$ of norm one positive operators in $L^1(E)$ with $\|f\cdot S_j-S_j\|_1\to 0$, for $f\in P^1(G)$. 

$(ii)\Rightarrow (iii)$. If $\phi$ is a weak$^*$-cluster point of the net $(S_i)$, then $\phi(f\cdot T)=\phi(T)$, for each $f\in P^1(G)$ and $T\in \mathbb B(E)$. For $f\in P^1(G)$ and $T\in \mathbb B(E)$,
 \begin{align*}
 	f\cdot(p_{t}Tp_t)&=\int_G f(s)\pi_{s^{-1}}p_tTp_t\pi_sds\\&=\int_G f(s)\pi_{s^{-1}}\pi_t\pi_{t^{-1}}T\pi_t\pi_{t^{-1}}\pi_{s}ds\\&=\int_G f(s)\pi_{s^{-1}t}\pi_{t^{-1}}T\pi_{t}\pi_{t^{-1}s}ds\\&=\int_G f(ts)\pi_{s^{-1}}\pi_{t^{-1}}T\pi_{t}\pi_{s}ds\\&=L_{t^{-1}}f\cdot (\pi_{t^{-1}}T\pi_t),
 \end{align*} 
Thus, $$\phi(p_{t}Tp_t)=\phi(f\cdot(p_{t}Tp_t))=\phi(L_{t^{-1}}f\cdot (\pi_{t^{-1}}T\pi_t))=\phi(\pi_{t^{-1}}T\pi_t),$$
that is, $\phi$ is invariant. 

$(iii)\Rightarrow (i)$. Any cluster point $\phi$ of the net $(S_i)$ is an invariant functional on $\mathbb B(E)$.

$(iii)\Rightarrow (iv)$.  Given $\varepsilon > 0$ and  $K\subseteq G$ compact, and $f\in P^1(G)$, choose a compact symmetric neighborhood
$U$ of $e$  such that $\|\frac{1}{m(U)}\mathds{1}_U*f - f\|_1 < \varepsilon$ and $\|L_sf - f\|_1 < \varepsilon$, for $s\in U$. Choose finitely many points $t_i\in K$ with $K= \bigcup_{i}  t_iU$. Choose norm one positive operator $S\in L^1(E)$ with,
$$ \|\pi_{t_i^{\pm 1}}S\pi_{t^{\mp 1}_i}-S\|_1<\varepsilon, \ \ \|f\cdot S-S\|_1<\varepsilon, \ \ \|\frac{1}{m(U)}(L_{t_i}\mathds{1}_U)\cdot S-S\|_1<\varepsilon,$$
for each of finitely many indices $i$. Put $T:=f\cdot S$, then,
from the first and second inequalities above, it follows that,
$$ \|\pi_{t_i^{\pm 1}}T\pi_{t^{\mp 1}_i}-T\|_1<3\varepsilon, \ \ \|p_{t_i^{\pm 1}}Tp_{t^{\pm 1}_i}-T\|_1<9\varepsilon, $$
for each index $i$. 

Also, for $s\in U$,
\begin{align*}
\|\frac{1}{m(U)}\mathds{1}_U\cdot T-\pi_{s^{-1}}T\pi_s\|_1&\leq \|\frac{1}{m(U)}\mathds{1}_U\cdot (f\cdot S)-f\cdot S\|_1\\&+\|\pi_{s^{-1}}(f\cdot S)\pi_s-f\cdot S\|_1\\&=\|(\frac{1}{m(U)}\mathds{1}_U\cdot f-f)\cdot S\|_1\\&+\|\pi_{s^{-1}}\big((L_sf-f)\cdot S\big)\pi_s\|_1<2\varepsilon. 
\end{align*}
Since $\|\frac{1}{m(U)}\mathds{1}_U\cdot (p_{t^{-1}_i}Tp_{t^{-1}_i}-T)-\pi_{s^{-1}}(p_{t^{-1}_i}Tp_{t^{-1}_i}-T)\pi_s\|_1$ is at most $18\varepsilon$, it follows from the above inequality that,
$$\|\frac{1}{m(U)}\mathds{1}_U\cdot (p_{t^{-1}_i}Tp_{t^{-1}_i})-\pi_{s^{-1}}(p_{t^{-1}_i}Tp_{t^{-1}_i})\pi_s\|_1<20\varepsilon.$$ 
Now, from the calculations in the proof of implication $(ii)\Rightarrow(iii)$,
$$\mathds{1}_U\cdot (p_{t^{-1}_i}Tp_{t^{-1}_i})=L_{t_i}\mathds{1}_U\cdot(\pi_{t_i}T\pi_{t^{-1}_i})=\mathds{1}_{t_iU}\cdot(\pi_{t_i}T\pi_{t^{-1}_i}),$$
and,
$$\pi_{s^{-1}}(p_{t^{-1}_i}Tp_{t^{-1}_i})\pi_s=\pi_{s^{-1}}\pi_{t^{-1}_i}\pi_{t_i}T\pi_{t^{-1}_i}\pi_{t_i}\pi_s=\pi_{(t_is)^{-1}}(\pi_{t_i}T\pi_{t^{-1}_i})\pi_{t_is},$$
and we could rewrite the above inequality as,
$$\|\frac{1}{m(U)}\mathds{1}_{t_iU}\cdot(\pi_{t_i}T\pi_{t^{-1}_i})-\pi_{(t_is)^{-1}}(\pi_{t_i}T\pi_{t^{-1}_i})\pi_{t_is}\|_1<20\varepsilon.$$ 
Again, since $\|\pi_{t_i}T\pi_{t^{-1}_i}-T\|_1<\varepsilon,$ it follows that,
$$\|\frac{1}{m(U)}\mathds{1}_{t_iU}\cdot T-\pi_{(t_is)^{-1}}T\pi_{t_is}\|_1<22\varepsilon.$$ 
Finally,
\begin{align*}
\|\pi_{(t_is)^{-1}}T\pi_{t_is}-T\|_1&\leq22\varepsilon+\|\frac{1}{m(U)}\mathds{1}_{t_iU}\cdot T-T\|_1\\&\leq22\varepsilon+\|\frac{1}{m(U)}\mathds{1}_{t_iU}\cdot S-S\|_1+2\|T-S\|_1\\&=22\varepsilon+
\|\frac{1}{m(U)}(L_{t_i}\mathds{1}_U)\cdot S-S\|_1+2\|f\cdot S-S\|_1\\&<25\varepsilon,
\end{align*}
for each $i$ and each $s\in U$. This means that $\|\pi_{t^{-1}}T\pi_{t}-T\|_1<25\varepsilon$, for each $t\in K$.

$(iv)\Rightarrow (iii)$. This is immediate. 
   \end{proof}

\begin{cor}
A \pa $\al$ satisfies  the Reiter condition $(P_1)$ iff its Koopman partial representation satisfies $(P_1)$. 
\end{cor}
 \begin{proof}
 This follows from Theorem \ref{Reiter}, Propositions \ref{Koopman} and \ref{Reiter2}
 \end{proof}
 \begin{rk}
 Recall that from the Cauchy-Schwartz and Powers-St{\o}rmer \cite{ps} inequalities, 
 $$\|SS^*-TT^*\|_1\leq\|S-T\|_2(\|S\|_2+\|T\|_2),\ \ \||S|-|T|\|^2_2\leq\||S|^2-|T|^2\|_1,$$ 
 for $S,T\in L^2(E)$. It is plausible that by these inequalities, Reiter conditions $(P_1)$ and $(P_2)$ are equivalent for partial representations, but the naive argument (as in the proof of \cite[Theorem 4.3]{b}) does not seem to work in this case.    
 \end{rk}

\begin{df}
A partial representation $\pi$ on a Hilbert space $E$ satisfies {\it F{\o}lner condition}
if  given $\varepsilon > 0$ and $K\subseteq  G$ compact, there exists
a nonzero finite rank projection $P\in\mathbb B(E)$ such that $\|\pi_tP\pi_{t^{-1}}-P\|_1<\varepsilon\|P\|_1$, for $t\in K$. It satisfies {\it weak F{\o}lner condition} if  given $\varepsilon > 0$, $\delta>0$  and $K\subseteq  G$ non-null compact, there exists
	a nonzero finite rank projection $P\in\mathbb B(E)$ and measurable subset $N\subseteq K$ such that $m(N)<\delta$ and  $\|\pi_tP\pi_{t^{-1}}-P\|_1<\varepsilon\|P\|_1$, for $t\in K\backslash N$.  
We say that $\pi$  has {\it Dieudonn\'e condition} $(D_p)$ if for the integral operator
	$$\Phi^p_\pi: M(G)\to \mathbb B(L^p(E));\ \ \Phi^p_\pi(\mu)T:=\int_G \pi_tT\pi_{t^{-1}}d\mu(t), \ \ (T\in L^p(E)),$$
	we have $\|\Phi^p_\pi(\mu)\|=\|\mu\|,$ for each positive bounded Radon measure $\mu\in M(G)_+$. 
\end{df}

The last result of this section is proved similar to \cite[Theorems 6.2, 6.3, 6.5]{b} (with the difference that for global action all these conditions are equivalent, whereas at this point we don't know this).

\begin{lem} \label{wf}
Let $\pi$ be a partial representation on a Hilbert space $E$ then,

$(i)$ if $\pi$ satisfies F{\o}lner condition, then it satisfies Reiter condition $(P_p)$, for each $1\leq p<\infty$,

$(ii)$ if $\pi$ satisfies Reiter condition $(P_2)$, it satisfies weak F{\o}lner condition,

$(iii)$ if $\pi$ satisfies Reiter condition $(P_p)$, it satisfies Dieudonn\'e condition $(D_p)$.
\end{lem}

\section{Induced partial representations and weak containment} \label{sec:ind}

In this section we introduce and study the notion of induced partial representations from a closed subgroup and use it to study the perseverance of amenability under weak containment. 

Let $H\leq G$ be a closed subgroup of a locally compact group $G$ and $\pi$ be a partial representation of $H$ in a Banach space $E$. Let $p_t:=\pi_t\pi_{t^{-1}}$ be the projection onto the closed subspace $E_t:={\rm Im}(\pi_t)$, for $t\in H$. 

Let $q: G\to G/H$ be the quotient map onto the homogeneous space of left cosets of $H$. Let $G=\bigsqcup_{x\in J} xH$ be a transversal decomposition of $G$ into disjoint union of left cosets of $H$. We call $J$ a {\it transversal set} for $H$ and fix it for the rest of our construction of the induced partial representation. 

We denote by $\mathcal F_\pi(G)$ the set of all norm continuous maps $\xi: G\to E$ with $q({\rm supp}\xi)\subseteq G/H$ compact, satisfying,
$$\pi_t\xi(st)=p_{t}\xi(s),\ \ \xi(xt)\in E_{t^{-1}}, \ \ (s\in G, x\in J, t\in H).$$ 
Note that for $x\in J$ and $t\in H$, the condition 
$\pi_t\xi(xt)=p_{t}\xi(x)$ is equivalent to $\xi(xt)=\pi_{t^{-1}}\xi(x)$, since $p_{t^{-1}}$ acts as identity on $E_{t^{-1}}$.  

Let us first observe that $\mathcal F_\pi(G)$ is non empty. For an $E$-valued continuous function $f\in C_c(G, E)$ of compact support, we put,
$$\xi_f(s):=\int_H \pi_t f(st)dt,$$
where the integral is against a left Haar measure on $H$. We denote the set of those $f\in C_c(G, E)$ satisfying, $f(xt)\in E_{t^{-1}}$ for $x\in J$, $t\in H$, with $C^\pi_c(G,E)$.

\begin{lem} For each closed subgroup $H$,
	
	$(i)$	$\mathcal F_\pi(G)\supseteq\{\xi_f: f\in C^\pi_c(G, E)\}=:\mathcal F^0_\pi(G),$
	
	$(ii)$ when $H=\{e\}$, 	$\mathcal F_\pi(G)=\mathcal F^0_\pi(G).$
\end{lem} 
\begin{proof}
	$(i)$ First let us observe that $\xi_f$ is norm continuous. Put $S:={\rm supp}f$. Since $f$ is uniformly continuous on $G$, given $\varepsilon>0$ and a compact neighborhood $V$ of $e$, there is a symmetric compact neighborhood with $U\subseteq V$ and $f(s)$ is $\varepsilon$-close to $f(t)$, whenever $s^{-1}t\in U$. For $s\in G\backslash VSH$ and $t\in Us$, $f(su)=f(tu)=0$, for $u\in H$, thus, $\xi_f(s)=\xi_f(t)=0$. On the other hand, for $s\in VSH$, $t\in Us$ and $u\in H\cap s^{-1}VS$, if $tuv\in S$, then $suv=(st^{-1})tuv\in VS,$ that is, $v\in (VS)^{-1}(VS)\cap H=:K$. Similarly, if $tuv\in S$, then $v\in K$. This means that $f(suv)=f(tuv)=0$, whenever $v\in H\backslash K$. Therefore,\begin{align*}
		\|\xi_f(s)-\xi_f(t)\|&=\Big\|\int_H\pi_v(f(sv)-f(tv))dv\Big\|\\&\leq\int_H\|\pi_v(f(sv)-f(tv))\|dv\\&=\int_H\|\pi_v(f(suv)-f(tuv))\|dv\\&\leq \int_K\|\pi_v(f(suv)-f(tuv))\|dv< m_H(K)\varepsilon,
	\end{align*}
	where $m_H$ is the left Haar measure on $H$, as required. Next, since $q({\rm supp}\xi_f)$ is a closed and so compact subset of $q({\rm supp}f)$. Finally, for $s\in G, t\in H$,
	\begin{align*}
		\pi_t\xi_f(st)&=\int_H \pi_t\pi_u f(stu)du=\int_H \pi_t\pi_{t^{-1}u} f(su)du\\&=\int_H \pi_t\pi_{t^{-1}}\pi_u f(su)du=\int_H p_t\pi_u f(su)du\\&=p_t\xi_f(s),
	\end{align*}  
and, for $x\in J, t\in H$,
\begin{align*}
	\xi_f(xt)&=\int_H \pi_u f(xtu)du=\int_H \pi_{t^{-1}u} f(xu)du\\&=\int_H \pi_{t^{-1}u} \pi_{u^{-1}}\pi_uf(xu)du=\int_H \pi_{t^{-1}}\pi_u f(xu)du\\&=\pi_{t^{-1}}\big(\xi_f(x)\big)\in E_{t^{-1}},
\end{align*}  	
where the third equality uses the fact that $f\in C^\pi_c(G, E)$ and $p_{u^{-1}}$ acts as the identity on $E_{u^{-1}}$. 
Summing up, $\xi_f\in\mathcal{F}_\pi(G)$. 
	
	$(ii)$ When $\pi$ is a global representation, for $\xi\in \mathcal{F}_\pi(G)$, we may choose $h\in C_c(G)$ with $\int_H h(st)dt=1$, for $s\in{\rm supp}\xi$ \cite[Proposition 1.9]{kt}. Put $f=h\xi$, and observe that,
	\begin{align*}
		\xi_f(s)&=\int_H h(st) \pi_t\xi(st)dt=\int_H h(st)\xi(s)dt\\&=\xi(s)\int_H h(st)dt=\xi(s),
	\end{align*}
	for $s\in G$. This in particular holds when $H$ is trivial.  
\end{proof}

In the above lemma, one way to make sure that $C^\pi_c(G,E)$ contains a nonzero element is to assume that $E_0:=\bigcap_{t\in H} E_t$ (which always contain zero) is a nonzero subspace of $E$. If $v\in E_0$ and $f_0\in C_c(G)$ is nonzero, then the map defined by $f(s):=f_0(s)v_0$ is a nonzero element of $C^\pi_c(G,E)$.
    
Let $\mu$ be a quasi-invariant regular Borel measure on $G/H$, and for $\xi\in \mathcal F_\pi(G)$, observe that, since $\pi_t$ acts as an isometry on $E_{t^{-1}}$,
$$\|\xi(xt)\|=\|\pi_t\xi(xt)\|=\|\xi(x)\|,\ \ (x\in J, t\in H),$$
thus the map $x\in J\mapsto \|\xi(x)\|\in\mathbb R$ could be regarded as a (bounded Borel) map on $G/H$. We put, $\|\xi\|:=\int_{G/H} \|\xi(x)\|d\mu(\dot x)$, where $\dot x:=q(x)$. The completion $\dot E$ of $\mathcal F_\pi(G)$ in this norm is a Banach space. When $E$ is a Hilbert space to start with, $\dot E$ is also a Hilbert space under the inner product $\langle\xi,\eta\rangle:=:=\int_{G/H} \langle\xi(x),\eta(x)\rangle d\mu(\dot x)$. 

As in the classical case \cite[Remark E.1.2]{bdv}, $\dot E$ could be identified with the space of all locally measurable maps $\xi: G\to E$, satisfying the specified conditions for all $t\in G$ and locally almost all $s\in G$ and $x\in J$, and with finite norm. When $G$ is $\sigma$-compact, we may drop ``locally'' in this statement.  

To define the {\it induced} partial representation, a natural approach is that, as in the classical case, for $s\in G$ and $\xi\in\mathcal F_\pi(G)$, we consider the left translation operators $L_s\xi(s'):=\sigma_\mu(s',s)\xi(s^{-1}s')$, for $s, s'\in G$, where on the RHS, the cocycle $\sigma_\mu(s',s):=[\frac{ds^{-1}_*\mu}{d\mu}]^{\frac{1}{2}}({\dot{s}}')$ is the Radon-Nikodym derivative of the push forward of $\mu$ w.r.t. $\mu$, and observe that,
\begin{align*}
\|L_s\xi\|&=\int_{G/H} \|\xi(s^{-1}s')\|\sigma_\mu(s', s)d\mu({\dot s}')\\&=\int_{G/H} \|\xi(s^{-1}s')\|d(s^{-1}_*\mu)({\dot s}')\\&=\int_{G/H} \|\xi(s')\|d\mu({\dot s}')=\|\xi\|.
\end{align*}
The problem here is that $L_s\xi\notin\mathcal F_\pi(G)$, in general. Indeed, the condition $L_s\xi(xt)\in E_{t^{-1}}$, for $x\in J$ and $t\in H$, is no longer satisfied. This is exactly why we get an induced {\it partial} representation on $G$.  
For $s\in G$, let $\dot E_s$ be the closure in $\dot E$ of those $\xi\in\mathcal F_\pi(G)$ which satisfy 
$\xi(sxt)\in E_{t^{-1}},$ for each $x\in J, t\in H$, then by the above calculation, $L_s$ extends to an isometric surjection from $\dot E_{s^{-1}}$ onto $\dot E_s$. If each $\dot E_s$ is {\it complemented} in $\dot E$ (which is the case when $E$ is a Hilbert space), the induced partial representation ${\rm Ind}^G_H\pi$ is defined by the extension of each $L_s$ to a partial isometry on $\dot E$. Now by the cocycle identity for $\sigma_\mu$, we have $L_{s_1}L_{s_2}=L_{s_1s_2}$ on $\dot E_{s^{-1}_2}={\rm Im}(L_{s^{-1}_2})$, that is, $L_{s_1}L_{s_2}L_{s^{-1}_2}=L_{s_1s_2}L_{s^{-1}_2},$  that is, ${\rm ind}^G_H\pi$ is a partial representation of $G$ on $\dot E$. Finally, it is straightforward to check that ${\rm ind}^G_H\pi$ is continuous (Borel) if $\pi$ is so. Summing up, we have the following result. 

\begin{prop}
Each continuous (Borel) partial representation $\pi$ of a closed subgroup $H\leq G$ in a Hilbert space $E$ induces a continuous (Borel) partial representation ${\rm Ind}^G_H\pi$ of $G$ in a Hilbert space $\dot E$. 
\end{prop}

For partial representations, there is an alternative construction of the induced representation (while these coincide for global representations) as follows: if we replace the condition $\xi(xt)=\pi_{t^{-1}}\xi(x)\in E_{t^{-1}}$, for $x\in J, t\in H$,  with the stronger condition that $\pi_t\xi(xt)=\xi(x)$, and $\xi(xt)\in E_{t^{-1}}$, for  $x\in J, t\in H$, we get a subspace $\mathcal F^0_\pi(G)$, and since we again have,
$$\|\xi(xt)\|=\|\pi_t\xi(xt)\|=\|\xi(x)\|,\ \ (x\in J, t\in H),$$
where the first equality follows from the fact that $\pi_t$ acts isometrically on $E_{t^{-1}}$. By the same argument, the completion $\dot E^0$ of $\mathcal F^0_\pi(G)$ is a Hilbert space (when $E$ is a Hilbert space) and $L_s$ lifts to a partial representation ${\rm ind}^G_H\pi$ on $\dot E^0$. As above $\dot E^0$ could be identified with the subspace of $\dot E$ consisting of all locally measurable maps $\xi: G\to E$, satisfying the specified conditions for all $t\in G$ and locally almost all $s\in G$ and $x\in J$ (with the stronger condition above replaced for its weaker version), and with finite norm. 

When $\pi$ is a global representations, these conditions are equivalent and we have ${\rm Ind}^G_H\pi={\rm ind}^G_H\pi$. In general, while $\dot E^0$ is a subspace of $\dot E$, since it is not necessarily invariant under operators $L_s$, ${\rm ind}^G_H\pi$ is not a partial subrepresentation of ${\rm Ind}^G_H\pi$. Also note that $\xi_f$ is not necessarily in $\mathcal F^0_\pi(G)$, for $f\in C^\pi_c(G, E)$. However, when $v\in E_0:=\bigcap_{t\in H} E_t$ and $f_0\in C_c(G)$, then $f(s):=f_0(s)v$ defines an element in $C^\pi_c(G, E)$ for which $\xi_f\in \mathcal F^0_\pi(G)$. Here, under some Urysohn type condition, we construct a total set in elements in $\dot E^0$. This is then used to prove a version of continuity of induction for ${\rm ind}^G_H\pi$.   

Recall that by Urysohn lemma, for respectively compact and open subsets $K$ and $U$ of $G$ with $K\subseteq U$, there is a continuous function $f$ of compact support on $G$ with $0\leq f\leq 1$ such that $f=1$ on $K$ and $f=0$ off $U$. We write $K\prec f\prec U$. 

Consider a partial representation $\pi$ of the subgroup $H$ on a Banach space $E$. Given $v\in E$, let $H_v:=\{(u,t)\in H\times H: \pi_u(v)\in E_{t^{-1}}\}.$ This is non empty, as $H\times \{e\}\subseteq H_v$. Put $H^v:=(H\times H)\backslash H_v$. This set could be empty, and indeed it is always empty when $\pi$ is a global representation. As before, we fix a transversal set $J$ with transversal decomposition $G=\bigsqcup_{x\in J} xH$. We put,
$$C^v_c(G):=\{f\in C_c(G): f(xtu)=f(xu)=0, \ (x\in J, (u,t)\in H^v)\}.$$
\noindent If $U\subseteq G$ is open, we write $C^v_c(G)$ to denote the set of those elements in $C^v_c(G)$ supported inside $U$. 

For $v\in E$ and $f\in C^v_c(G)$, we put,
$$\xi_{f,v}(s):=\int_H f(st)\pi_t(v)dt,\ \ (s\in G),$$ 
and observe that $\xi_{f,v}\in \mathcal F^0_\pi(G)$: for $s\in G$, $t\in H$,
\begin{align*}
	\pi_t\xi_{f,v}(st)&=\int_H  f(stu)\pi_t\pi_u(v)du=\int_H f(su)\pi_t\pi_{t^{-1}u}(v) du\\&=\int_H f(su)\pi_t\pi_{t^{-1}}\pi_u(v) du=\int_H f(su)p_t\pi_u(v) du=p_t\xi_{f,v}(s),
\end{align*}  
and, for $x\in J, t\in H$,
\begin{align*}
	\xi_{f,v}(xt)&=\int_H f(xtu)\pi_u(v) du=\int_H f(xtu)\pi_{t^{-1}} \pi_{t}\pi_u(v)du\\&=\int_H f(xtu)\pi_{t^{-1}}\pi_{tu}(v) du=\int_H f(xu)\pi_{t^{-1}}\pi_{u}(v) du\\&=\pi_{t^{-1}}\big(\xi_{f,v}(x)\big)\in E_{t^{-1}},
\end{align*}  	
where the second equality uses the fact that $f(xtu)=0$ whenever $\pi_u(v)$ is not in $E_{t^{-1}}$, and that $p_{t^{-1}}$ acts as the identity on $E_{t^{-1}}$. Finally, 
\begin{align*}
	\pi_t\xi_{f,v}(xt)&=\int_H f(xtu)\pi_t\pi_u(v) du=\int_H f(xtu) \pi_{t}p_{t^{-1}}\pi_u(v)du\\&=\int_H f(xtu)\pi_t\pi_{t^{-1}}\pi_{tu}(v) du=\int_H f(xu)\pi_t\pi_{t^{-1}}\pi_{u}(v) du\\&=\int_H f(xu)\pi_{u}(v)du=\xi_{f,v}(x),
\end{align*}  
for $x\in J, t\in H$, where the second and fifth equalities follow from the facts that $f(xtu)=f(xu)=0$ whenever $\pi_u(v)\notin E_{t^{-1}}$, and that $p_{t^{-1}}$ acts as the identity on $E_{t^{-1}}$.

\begin{df}\label{ury}
We say that $\pi$ has {\it Urysohn property} if  given $K\subseteq G$ compact,  and finite open cover $K\subseteq U_1\cup\cdots\cup U_n=:U$, for each $v_1,\cdots, v_n$ in $E$, there are functions $f_i\in C^{v_i}_c(U_i)_+$, $1\leq i\leq n$, with $K\prec f_1+\cdots+f_n\prec U$.   
\end{df} 

When $\pi$ is a global representation, it automatically has Urysohn property. Indeed,  by the Urysohn lemma, for a compact set $K$, open neighborhood $V$ of $e$,  and open cover $U:=\bigcup_{i=1}^{n}Vs_i$ of $K$, we may find $f_i\in C_c(Vs_i)$ with $K\prec f_1+\cdots+f_n\prec U$.

\begin{ex}
It is illustrative to see what the Urysohn property says for the Koopman partial representation of a given partial action $\al$ of $G$ on a standard Borel space $(X,\nu)$. In this case, $E:=L^2(X,\nu)$ and $E_t:=L^2(X_{t^{-1}},\nu)$, which is identified canonically with a closed subspace of $E$. Let $H\leq G$ be a closed subgroup with transversal set $J$. Given open subset $U\subseteq G$ and $v\in L^2(X,\nu)$,
$$C^v_c(U):=\{f\in C_c(U): f(xtu)=f(xu)=0, \ \ (x\in J, (u,t)\in E^v)\},$$
where $E^v$ is the complement of the set
$$E_v:=\{(u,t)\in H\times H: \kappa^\al_uv\in L^2(X_t, \nu)\},$$
where the above condition simply means that $\kappa^\al_uv=0$ off $X_t$, that is, $v(\al_{u^{-1}}(z))=0$, for $z\notin X_t$. Let us write ${\rm supp}^0v:=\{z\in X: v(z)\neq 0\}.$ Here there is no sense to take closure of this set, as $X$ is only a measure space (we use the same notation for continuous functions on $G$). Using this notation, the last condition could be written as, ${\rm supp}^0v\subseteq \al_{u^{-1}}(X_u\cap X^c_t)$. Thus, $f\in C^v_c(U)$ simply means that $Jtu\cup Ju\subseteq G\backslash {\rm supp}^0f$, whenever ${\rm supp}^0v\nsubseteq \al_{u^{-1}}(X_u\cap X^c_t)$, which is a void condition when $\al$ is a global action.
\end{ex}

\begin{lem}\label{total}
If $\pi$ has Urysohn property, the set $\{\xi_{f,v}: v\in E, f\in C^v_c(G)\}$ is a total set in $\dot E^0$. 
\end{lem}
\begin{proof}
Let $\xi\in\mathcal F^0_\pi(G)$. Choose $g\in C_c(G)$ with $\int_H g(st)dt=1$, whenever $\dot s:=sH\in q({\rm supp}(\xi))$, where $q: G\to G/H$ is the quotient map \cite[B.1.2]{bdv}. Set $\eta:=g\xi$, then $K:={\rm supp}(\eta)$ is compact and,
$$\int_H\pi_t \eta(xt)dt=\int_H g(xt)\pi_t\xi(xt)dt=\xi(x)\int_H g(xt)dt=\xi(x),
$$
for $x\in J$. Let $S:={\rm supp}\eta$ and $K$ be a compact neighborhood of $S$. Since $\eta$ is compactly supported, it is uniformly continuous, that is, given $\varepsilon>0$, there is a neighborhood $V$ of $e$ with,
$$\|\eta(us)-\eta(s)\|<\varepsilon, \ \ (u\in V, s\in G). 
$$ 
Cover $S$ by $U:=\bigcup_{i=1}^{n}Vs_i$, with $s_i\in S$, and use the assumption for $v_i:=\eta(s_i)$ to choose $f_i\in C^{v_i}_c(Vs_i)_+$, $1\leq i\leq n$, with $f_1+\cdots+f_n\leq 1$. The above inequality for $\eta$ guarantees that $\|\eta(s)-v_i\|<\varepsilon$, whenever $f_i(s)>0$, thus,
$$
\big\|\eta(s)-\sum_{i=1}^{n}f_i(s)v_i\big\|\leq \sum_{i=1}^{n}f_i(s)\|\eta(s)-v_i\|<\varepsilon,$$
for $s\in G$. 

Next, by the observation before Definition \ref{ury}, each $\xi_{f_i,v_i}$ is in $\mathcal F^0_\pi(G)$, and since $\eta$ and all $f_i$'s are supported in $K$, and $q(KH\cap J)=q(K)$ is compact in $G/H$, for the left Haar measure $m_H$ on $H$ and quasi-invariant measure $\mu$ on $G/H$, we have,
\begin{align*}
\Big\|\xi-\sum_{i=1}^{n}\xi_{f_i,v_i}\Big\|^2&\leq \int_{q(KH\cap J)}\Big(\int_H\big\|\pi_t\eta(xt)-\sum_{i=1}^{n}\xi_{f_i,v_i}(x)\big\|dt\Big)^2d\mu(\dot x)\\&= \int_{q(KH\cap J)}\Big(\int_H\big\|\pi_t\eta(xt)-\sum_{i=1}^{n}\pi_t\xi_{f_i,v_i}(xt)\big\|dt\Big)^2d\mu(\dot x)\\&\leq \int_{q(KH\cap J)}\Big(\int_H\big\|\eta(xt)-\sum_{i=1}^{n}\xi_{f_i,v_i}(x)\big\|dt\Big)^2d\mu(\dot x)\\&\leq \mu\big(q(KH\cap J)\big)m_H\big(K^{-1}K\cap H\big)^2\varepsilon^2,
\end{align*}
as required.
\end{proof}

\begin{rk} \label{total2}
In the above lemma, one could have a control on the norms of $v_i$'s and $f_i$'s. Indeed, by construction, we have $\|v_i\|\leq \|\eta\|_\infty$, and we may arrange (by multiplying $f_i$'s with appropriate constant factors) to have $\Big\|\xi-\sum_{1}^{n}\xi_{f_i,v_i}\Big\|<\varepsilon$, with $\|f_i\|_\infty<\varepsilon$, for $1\leq i\leq n$. Also, we may arrange (by multiplying $f_i$'s with appropriate continuous functions of norm at most one) that there are compact subsets $Q_i\subseteq H$ such that $f_i(xt)=0$, for $x\in J$ and $t\in H\backslash Q_i$.   
\end{rk}

\begin{df}
Let $\pi$ and $\sigma$ be partial representations of a topological (Borel) group $G$ in  Banach spaces $E$ and $F$, respectively. Let $E_t:={\rm Im}(\pi_t)$ and $F_t:={\rm Im}(\sigma_t)$, for $t\in G$. We say that $\pi$ and $\sigma$ are {\it equivalent}, and write $\pi\sim_u\sigma$, if there is an intertwining algebraic isomorphism $\phi: E\to F$ with $\phi(E_t)=F_t$, for each $t\in G$, and $\phi$ is isometric on each $E_t$. Here, being intertwining means that $\sigma_t\circ\phi=\phi\circ\pi_t$, for $t\in G$. We say that $\pi$ is {\it weakly contained} in $\sigma$, and write $\pi\preceq\sigma$, if given $\varepsilon>0$, $K\subseteq G$ compact, $x\in E$ and $x^*\in E^*$, there are finitely many elements $y_1,\cdots,y_n\in F$ and $y^*_1,\cdots,y^*_n\in F^*$ such that,
$$\big|\langle \pi_t(x), x^*\rangle-\sum_{i=1}^{n}\langle \sigma_t(y_i), y^*_i\rangle\big|<\varepsilon, \ \ (t\in K).$$
When $\pi\preceq\sigma$ and $\sigma\preceq\pi$, we say that $\pi$ and $\sigma$ are {\it weakly equivalent} and write $\pi\sim\sigma$.  
\end{df}

When $E$ is a Hilbert space, each isomorphism gives a family $\phi|_{E_t}: E_t\to F_t$ of  unitary operators, so $\pi\sim_u\sigma$ is indeed {\it unitary equivalence} in this case. Also, for partial representations on Hilbert spaces, $\pi\preceq\sigma$ is equivalent to the requirement that given $\varepsilon>0$, $K\subseteq G$ compact, $\xi\in E$, there are finitely many points $\eta_1,\cdots,\eta_n\in F$ such that,
$$\big|\langle \pi_t(\xi), \xi\rangle-\sum_{i=1}^{n}\langle \sigma_t(\eta_i), \eta_i\rangle\big|<\varepsilon, \ \ (t\in K).$$  
Let $1_G$ be the {\it trivial} representation on $G$. By an argument as in \cite[Corollary F.1.5]{bdv}, using inequalities,
$$\frac{1}{2}\|\pi_t(\xi)-\xi\|^2\leq\big|1-\langle \pi_t(\xi), \xi\rangle\big|\leq \|\pi_t(\xi)-\xi\|^2,$$
we have the following result.

\begin{lem} \label{wc}
If $\pi$ is a partial representation of $G$ in a Hilbert space $E$, then $1_G\preceq\pi$ if and only if for each $\varepsilon>0$ and $K\subseteq G$ compact, there is a unit vector $\xi\in E$ with 
$$\sup_{t\in K}\|\pi_t(\xi)-\xi\|<\varepsilon.$$
\end{lem}

Next, we relate a weak containment property to the Reiter condition $(\tilde P_2)$, defined in the paragraph after Definition \ref{rc}.
 
\begin{thm} \label{wcrei}
Let $\pi$ be a partial representation of a locally compact group $G$ in a Hilbert space $E$ with conjugate partial representation $\bar\pi$ on $\bar E$. The following are equivalent:

$(i)$ $1_G\preceq \pi\otimes\bar\pi$,

$(ii)$ $\pi$ satisfies Reiter condition $(\tilde P_2)$. 
\end{thm}
\begin{proof}
$(i)\Rightarrow (ii)$. Given $\varepsilon>0$ and $K\subseteq G$ compact, there exists a norm one operator $S\in L^2(E)$ with $|1-\langle\pi_tS\pi_{t^{-1}}, S\rangle|<\varepsilon^2/2$, for $t\in K$. By the first  inequality in the paragraph before Lemma  \ref{wc}, $\|S-\pi_tS\pi_{t^{-1}}\|_2<\varepsilon,$ for $t\in K$, thus $(\tilde P_2)$ holds.

$(ii)\Rightarrow (i)$. First let us observe that $\pi\otimes\bar\pi$ is unitarily equivalent to the partial representation $\sigma$ of $G$ on $L^2(E)$ given by, $$\sigma_t(T):=\pi_tT\pi_{t^{-1}},\ \ (t\in G, T\in L^2(E)).$$
Since,
$$\sigma_{t^{-1}}\sigma_{ts}(T):=\sigma_{t^{-1}}(\pi_{ts}T\pi_{s^{-1}t^{-1}})=\pi_{t^{-1}}\pi_{ts}T\pi_{s^{-1}t^{-1}}\pi_t=\sigma_{t^{-1}}\sigma_{t}\sigma_s(T),$$
for $s,t\in G$ and $T\in L^2(E)$ (and similarly for the other identity), $\sigma$ is partial representation. Now, the unitary isomorphism $L^2(E)\simeq E\otimes \bar E$ maps $L^2(E)_t:={\rm Im}(\sigma_t)$ onto $E_t\otimes \bar E_t:={\rm Im}(\pi_t\otimes \bar\pi_t)$, for each $t\in G$, thus, $\sigma\sim_u\pi\otimes\bar\pi$, as claimed.   If $\pi$ satisfies Reiter condition $(\tilde P_2)$, given $\varepsilon>0$ and $K\subseteq G$ compact, there exists a norm one operator $T\in L^2(E)$ with $\|\pi_tT\pi_{t^{-1}}-T\|_2<\varepsilon$, for $t\in K$. By the second  inequality in the paragraph before Lemma  \ref{wc}, 
$$\big|1-\langle \sigma_t(T),T\rangle\big|=\big|1-\langle \pi_tT\pi_{t^{-1}}, T\rangle\big|<\varepsilon,\ \ (t\in K),$$
that is, $1\preceq\sigma\sim_u\pi\otimes\bar\pi$. 
\end{proof}

\begin{rk}
$(i)$ In order to get equivalence with $(P_2)$ in the above result we need a modified (stronger) version of weak containment. For partial representations $\pi$ and $\sigma$ on Hilbert spaces $E$ and $F$, let us write $\pi\preceq^{\rm s}\sigma$ if given $\varepsilon>0$, $K\subseteq G$ compact, $\xi\in E$, there are finitely many points $\eta_1,\cdots,\eta_n\in F$ such that,
$$\big|\langle \pi_t(\xi), \xi\rangle-\sum_{i=1}^{n}\langle \sigma_t(\eta_i), \eta_i\rangle\big|<\varepsilon, \ \ \sum_{i=1}^{n}\|\eta_i-\sigma_{t^{-1}}\sigma_t(\eta_i)\|<\varepsilon, \ \ (t\in K).$$  
Now in the above result,  $1\preceq^{\rm s}\pi\otimes\bar\pi$ implies the stronger condition $(P_2)$ by the following argument: Since we may choose $S$ within $\varepsilon$ of $L^2(E)_t$, $S$ $S$ is within $2\varepsilon$ of a finite linear combination of the vectors of the form $\pi_t\xi\otimes\bar\pi_t\bar\eta$, with $\xi,\eta\in E$ and $t\in G$. Identifying $\pi_t\xi\otimes\bar\pi_t\bar\eta$ with the corresponding rank one operator on $E$, for $\zeta\in E$,
\begin{align*}
	(\pi_t\xi\otimes\bar\pi_t\bar\eta)\pi_{t^{-1}}\pi_t(\zeta)&=\langle\pi_{t^{-1}}\pi_t(\zeta), \pi_t\eta\rangle\pi_t\xi=\langle\pi_t\pi_{t^{-1}}\pi_t(\zeta), \eta\rangle\pi_t\xi\\&=\langle\pi_t(\zeta), \eta\rangle\pi_t\xi=\langle\zeta, \pi_t\eta\rangle\pi_t\xi=(\pi_t\xi\otimes\bar\pi_t\bar\eta)(\zeta),
\end{align*}
thus, by linearity and continuity, $S$ is within $4\varepsilon$ of $S\pi_{t^{-1}}\pi_t$. As above, we also have, $\|S-\pi_tS\pi_{t^{-1}}\|_2<\varepsilon^2/2,$ for $t\in K$. 
Thus, for the positive norm one operator $T:=|S|^2$, by the Cauchy-Schwartz and Powers-St{\o}rmer inequalities, 
\begin{align*}
	\|\pi_tT\pi_{t^{-1}}-T\|^2_2 &=\|\pi_tSS^*\pi_{t^{-1}}-SS^*\|^2_2\\&\leq\|\pi_tSS^*\pi_{t^{-1}}-SS^*\|_1\\&=4\varepsilon+\|\pi_tS\pi_{t^{-1}}\pi_tS^*\pi_{t^{-1}}-SS^*\|_1\\&\leq 4\varepsilon+2\|\pi_tS\pi_{t^{-1}}-S\|_2<4\varepsilon+\varepsilon^2,
\end{align*}
for each $t\in K$, thus $(P_2)$ holds.

$(ii)$ Similar idea could lead to a stronger version of $(P_p)$. We say that a partial representation $\pi$ on a Hilbert space $E$ satisfies {\it strong} Reiter condition $(P^{\rm s}_p)$, for $1\leq p<\infty$, if for each $\varepsilon > 0$ and $K\subseteq  G$ compact, there exists a positive norm one operator $T\in L^p(E)$ with $\|\pi_tT\pi_{t^{-1}}-T\|_p<\varepsilon$ and $\|T\pi_{t^{-1}}\pi_{t}-T\|_p<\varepsilon$, for $t\in K$. Unlike the case of $(P_p)$, it is not hard to see (using an argument similar to the one used in part $(i)$ above) that $(P^{\rm s}_1)$ and $(P^{\rm s}_2)$ are equivalent.  
\end{rk}

Next, we prove the {\it continuity of induction} for ${\rm ind}^G_H$. This result is not available for ${\rm Ind}^G_H$, as far as we know, and it was one of main motivations to introduce the induction on the smaller Hilbert space $\dot E^0$.

\begin{thm}\label{ci}
Let $H$ be a closed subgroup of a locally compact group $G$ and $\pi$ be a partial representation of $H$ in a Hilbert space $E_\pi$ satisfying the Urysohn condition. Then for each partial representation $\rho$ of $H$ in a Hilbert space $E_\sigma$, $\pi\preceq\rho$ implies ${\rm ind}^G_H\pi\preceq{\rm ind}^G_H\rho$. 
\end{thm} 
\begin{proof}
By lemma \ref{total}, we only need to show that, given $\varepsilon>0$,  coefficient functions of the form $\langle{\rm ind}^G_H\pi(\cdot)\xi_{f,v},\xi_{f,v}\rangle$, for  $v\in E_\pi$ and $f\in C^v_c(G)$, could be approximated within $\varepsilon$ by a finite sum of coefficient functions of  ${\rm ind}^G_H\rho$. By Remark \ref{total2}, we may assume that $\|f\|_\infty<\varepsilon$ and $f(xt)=0$, for each $x\in J$ and $t\notin Q$, for some compact subset $Q\subseteq H$. For a quasi-invariant measure $\mu$ on $G/H$ with cocycle $\sigma_\mu$, 
\begin{align*}
\langle{\rm ind}^G_H&\pi(s)\xi_{f,v},\xi_{f,v}\rangle=\int_{G/H} \sigma_\mu(s^{-1}, \dot x)^{\frac{1}{2}}\langle\xi_{f,v}(s^{-1}x),\xi_{f,v}(x)\rangle d\mu(\dot x)\\&=\int_{G/H}\int_H \sigma_\mu(s^{-1}, \dot x)^{\frac{1}{2}}f(s^{-1}xt)\langle \pi_t(v),\xi_{f,v}(x)\rangle d\mu(\dot x)\\&=\int_{G/H}\int_H \sigma_\mu(s^{-1}, \dot x)^{\frac{1}{2}}f(s^{-1}xt)\langle \pi_t(v),\pi_t\xi_{f,v}(xt)\rangle d\mu(\dot x)\\&=\int_{G/H}\int_H\int_H \sigma_\mu(s^{-1}, \dot x)^{\frac{1}{2}}f(s^{-1}xt)\bar f(xtu)\langle\pi_{t}(v),\pi_t\pi_{u}(v)\rangle dtdud\mu(\dot x)\\&=\int_{G/H}\int_H\int_H \sigma_\mu(s^{-1}, \dot x)^{\frac{1}{2}}f(s^{-1}xt)\bar f(xtu)\langle v,\pi_{t^{-1}}\pi_t\pi_{u}(v)\rangle dtdud\mu(\dot x)\\&=\int_{G/H}\int_H\int_H \sigma_\mu(s^{-1}, \dot x)^{\frac{1}{2}}f(s^{-1}xt)\bar f(xtu)\langle v,\pi_{u}(v)\rangle dtdud\mu(\dot x)\\&=\int_{G/H}\int_H\int_H \sigma_\mu(s^{-1}, \dot x)^{\frac{1}{2}}f(s^{-1}xt)\bar f(xu)\langle v,\pi_{t^{-1}u}v\rangle dtdud\mu(\dot x),
\end{align*} 
where, as before, the sixth equality follows from the facts that $f(xtu)=0$ whenever $\pi_u(v)\notin E_{t^{-1}}$, and that $p_{t^{-1}}$ acts as the identity on $E_{t^{-1}}$. 

Given $K\subseteq G$ compact, set $L:=(S^{-1}KS)\cap H$, where $S:={\rm supp}f$. By assumption, there are vectors $w_1,\cdots, w_n\in E_\rho$ such that,
$$\sup_{t\in L} \big|\langle \pi_t(v)-v\rangle-\sum_{i=1}^{n}\langle\rho_t(w_i), w_i\rangle\big|<\varepsilon.$$
We do not have $f\in C^{w_i}_c(G)$, but since this is characterized by vanishing over certain set, we may arrange that $hf\in C^{w_i}_c(G)$, for each $i$, for a function $h\in C_c(G)$ with $\|h\|_\infty\leq 1$. Put $g:=hf$, then $g\in C^v_c(G)$ and for vectors $\xi_{kf,v}\in \dot E_\pi$, we have the estimates,
\begin{align*}
\|\xi_{kf,v}\|&=\int_{G/H}\Big\|\int_H k(xt)f(xt)\pi_t(v)dt\Big\|d\mu(\dot x)\\&\leq \mu(q(S))m_H(Q)\|k\|_\infty\|f\|_\infty\|v\|\\&<\mu(q(S))m_H(Q)\|k\|_\infty\|v\|\varepsilon,
\end{align*}
for each $k\in C_c(G)$, and,
\begin{align*}
\big|\langle{\rm ind}^G_H\pi(s)\xi_{g,v},\xi_{f,v}\rangle-&\langle{\rm ind}^G_H\pi(s)\xi_{f,v},\xi_{f,v}\rangle\big|=\big|\langle{\rm ind}^G_H\pi(s)\xi_{g-f,v},\xi_{f,v}\rangle\big|
\\&\leq \|\xi_{g-f,v}\|\|\xi_{f,v}\|=\|\xi_{(h-1)f,v}\|\|\xi_{f,v}\|\\&\leq \mu(q(S))^2m_H(Q)^2\|h-1\|_\infty\|v\|^2\varepsilon^2\\&\leq 2\mu(q(S))^2m_H(Q)^2\|v\|^2\varepsilon^2,
\end{align*}
and similarly,
$$\big|\langle{\rm ind}^G_H\pi(s)\xi_{g,v},\xi_{g,v}\rangle-\langle{\rm ind}^G_H\pi(s)\xi_{g,v},\xi_{f,v}\rangle\big|\leq 2\mu(q(S))^2m_H(Q)^2\|v\|^2\varepsilon^2,$$
thus, 
$$\big|\langle{\rm ind}^G_H\pi(s)\xi_{g,v},\xi_{g,v}\rangle-\langle{\rm ind}^G_H\pi(s)\xi_{f,v},\xi_{f,v}\rangle\big|\leq 4\mu(q(S))^2m_H(Q)^2\|v\|^2\varepsilon^2.$$
Now since $g\in \bigcap_{i=1}^{n}C^{w_i}_c(G)$, we have the vectors $\xi_{g,w_i}$ in $\dot E_\rho$, for which we have the estimate,
\begin{align*}
\big|\langle{\rm ind}^G_H\pi&(s)\xi_{g,v},\xi_{g,v}\rangle-\sum_{i=1}^{n}\langle{\rm ind}^G_H\rho(s)\xi_{g,w_i},\xi_{g,w_i}\rangle\big|\\&=\int_{G/H}\sigma_\mu(s^{-1},\dot x)^{\frac{1}{2}}\int_H\int_H g(s^{-1}xt)\bar g(xu)\bar D(t^{-1}u)dtdud\mu(\dot x),
\end{align*}
for $s\in K$, where bar is complex conjugate and,
$$D(t):=\langle \pi_t(v)-v\rangle-\sum_{i=1}^{n}\langle\rho_t(w_i), w_i\rangle, \ \ (t\in H).$$
Since, $g(s^{-1}xt)\bar g(xu)=h(s^{-1}xt)\bar h(xu)f(s^{-1}xt)\bar f(xu)=0$, unless $s^{-1}xt, xu\in {\rm supp}f=:S$, which implies $t^{-1}u\in (S^{-1}KS)\cap H=:L$, we have $|D(t^{-1}u)|<\varepsilon,$ whenever the integrand is nonzero in the last estimate. Recall that,
$\dot k(\dot x):=\int_H k(xt)dt,$ defines an element of $C_c(G/H)$,
for $k\in C_c(G)$. For $g_0(s):=|g(s)|$ and $f_0(s):=|f(s)|$, $s\in G$, we have the estimate,
\begin{align*}
\big|\langle{\rm ind}^G_H\pi&(s)\xi_{g,v},\xi_{g,v}\rangle-\sum_{i=1}^{n}\langle{\rm ind}^G_H\rho(s)\xi_{g,w_i},\xi_{g,w_i}\rangle\big|\\&\leq\int_{G/H}\sigma_\mu(s^{-1},\dot x)^{\frac{1}{2}}\int_H\int_H g_0(s^{-1}xt) g_0(xu)|D(t^{-1}u)|dtdud\mu(\dot x)\\&\leq \varepsilon\int_{G/H}\sigma_\mu(s^{-1},\dot x)^{\frac{1}{2}}\int_H\int_H g_0(s^{-1}xt) g_0(xu)dtdud\mu(\dot x)\\&\leq \varepsilon\Big(\int_{G/H}\sigma_\mu(s^{-1},\dot x)\dot g_0(s^{-1}\dot x)d\mu(\dot x)\Big)^{\frac{1}{2}}\Big(\int_{G/H}\dot g_0(\dot x)d\mu(\dot x)\Big)^{\frac{1}{2}}\\&= \varepsilon\int_{G/H}\dot g_0(\dot x)d\mu(\dot x)=\varepsilon\|\dot g_0\|_2^2\leq \varepsilon\|\dot f_0\|_2^2.
\end{align*} 
This plus the above estimates now gives, 
\begin{align*}
\sup_{s\in K}\big|\langle{\rm ind}^G_H\pi(s)\xi_{f,v},\xi_{f,v}\rangle-&\sum_{i=1}^{n}\langle{\rm ind}^G_H\rho(s)\xi_{g,w_i},\xi_{g,w_i}\rangle\big|\\&<4\mu(q(S))^2m_H(Q)^2\|v\|^2\varepsilon^2+\varepsilon\|\dot{\overset{\huge\frown}{|f|}}\|_2^2,
\end{align*}
for each $s\in K$, which shows that  ${\rm ind}^G_H\pi\preceq{\rm ind}^G_H\rho$. 
\end{proof}

Now we are ready to find conditions for amenability of the induced representation. 
First, let us recall the following notion of amenability of homogeneous spaces due to Eymard \cite{ey}.

\begin{df}
Let $H$ be a closed subgroup of a locally compact group $G$, then the homogeneous space $G/H$ is {\it amenable} in the sense of Eymard if there a left translation invariant mean on $L^\infty(G/H, \mu)$ for some quasi-invariant measure $\mu$ on $G/H$. 
\end{df}
  
It is known that $G/H$ is  amenable in the sense of Eymard if and only if ${\rm ind}^G_H 1_H$ is amenable in the sense of Bekka \cite[Theorem 2.3]{b}. Since the trivial representation is global, so is the induced representation ${\rm ind}^G_H 1_H$, which is nothing but the quasi-regular representation of $G$ on $L^2(G/H,\mu)$, for a quasi-invariant measure $\mu$ on $G/H$.

\begin{prop}
If $\pi$ is a partial representation of $H$ in a Holbert space $E$ and ${\rm Ind}^G_H\pi$ is amenable in the sense of Bekka, then $G/H$ is amenable in the sense of Eymard. The same implication holds for ${\rm ind}^G_H\pi$. 
\end{prop}
\begin{proof}
Let us observe that $L^\infty(G/H,\mu)$ acts on the Hilbert space $\dot E$ of ${\rm Ind}^G_H\pi$ by multiplication and,
$$({\rm Ind}^G_H\pi)_sT_\varphi({\rm Ind}^G_H\pi)_{s^{-1}}=P_sT_{L_s\varphi},\ \ \big(s\in G, \varphi\in L^\infty(G/H,\mu)\big),$$ 
where $P_s$ is the orthogonal projection onto $\dot E_s$. Given $\xi\in \dot E_s$, 
\begin{align*}
({\rm Ind}^G_H\pi)_sT_\varphi({\rm Ind}^G_H\pi)_{s^{-1}}\xi(x)&=T_\varphi({\rm Ind}^G_H\pi)_{s^{-1}}\xi(s^{-1}x)\\&=\varphi(s^{-1}\dot x)({\rm Ind}^G_H\pi)_s\xi(s^{-1}x)\\&=\varphi(s^{-1}\dot x)\xi(x)\\&=P_sT_{L_s\varphi}\xi(x),
\end{align*} 
where as both sides are zero when $\xi$ is in the orthogonal complement of $\dot E_s$ in $\dot E$. Now if $\Phi$ is an ${\rm Ind}^G_H\pi$-invariant  mean on  $\mathbb B(\dot E)$, then $m(\varphi):=\Phi(T_\varphi)$ defines a left translation invariant mean on $L^\infty(G/H, \mu)$. A similar argument works for ${\rm ind}^G_H\pi$ acting on the Hilbert space $\dot E^0$.  
\end{proof}

\begin{lem} \label{resind}
Let $H$ be a closed subgroup of a locally compact group $G$ and $\sigma$ and $\pi$ be partial representations of $G$ and $H$, respectively on Hilbert spaces $E_\sigma$ and $E_\pi$, then,

$(i)$ ${\rm ind}^G_H({\rm Res}^G_H\sigma\otimes\pi)$ is equivalent to $\sigma\otimes{\rm ind}^G_H\pi$,

$(ii)$ if $\bar \pi$ is the conjugate representation of $\pi$ on $\bar E_\pi$, ${\rm ind}^G_H\bar\pi$ is equivalent to the conjugate representation of ${\rm ind}^G_H\pi$ on the conjugate Hilbert space of ${\dot E_\pi}$. 
\end{lem}   
\begin{proof}
$(i)$ We let $(E_\sigma\odot \dot E_\pi)_s$ consist of those finite sums $\zeta:=\sum_i \eta_i\otimes\xi_i$ of elementary tensors, for which $q({\rm supp}\xi_i)=q(A_i)$, where $A_{i}:=\{s\in G: \eta_i\in (E_\sigma)_s\}$, and let $(E_\sigma\otimes \dot E_\pi)_s$ be the closure of $(E_\sigma\odot \dot E_\pi)_s$ in $E_\sigma \otimes E_\pi$. Define a linear map $U: E_\sigma\odot \dot E_\pi\to (E_\sigma\otimes E_\pi)^{\dot{}}$ over finite sums of elementary tensors by,
$$U(\sum_i \eta_i\otimes\xi_i)(s):=\sum_i \mathds{1}_{A_{i}}(s)\sigma_{s^{-1}}(\eta_i)\otimes\xi_i(s),$$
where $A_{i}$ is as above. If $\zeta:=\sum_i \eta_i\otimes\xi_i\in (E_\sigma\odot \dot E_\pi)_s$, then for quasi-invariant measure $\mu$ on $G/H$,
\begin{align*}
\big\|U(\sum_i \eta_i\otimes\xi_i)\big\|^2&=\sum_i\int_{q(A_i)} \|\sigma_{s^{-1}}(\eta_i)\|^2\|\xi_i(s)\|^2d\mu(\dot s)\\&=\sum_i \|\eta_i\|^2\int_{q(A_i)} \|\xi_i(s)\|^2d\mu(\dot s)\\&=\sum_i \|\eta_i\|^2\int_{q({\rm supp}\xi_i)} \|\xi_i(s)\|^2d\mu(\dot s)\\&=\sum_i \|\eta_i\|^2\|\xi_i\|^2=\big\|\sum_i \eta_i\otimes\xi_i\big\|^2,        
\end{align*}
thus the restriction of $U$ to $(E_\sigma\odot \dot E_\pi)_s$ is an isometry into $(E_\sigma\otimes E_\pi)^{\dot{}}_s$ with dense range. Moreover, 
\begin{align*}
{\rm ind}^G_H({\rm Res}^G_H\sigma\otimes\pi)(s)U\big(\sum_i \eta_i\otimes\xi_i\big)(x)&=\sum_i\sigma_\mu(s^{-1},\dot x)\sigma_{x^{-1}s}(\eta_i)\otimes\xi_i(s^{-1}x)\\&=U\big(\sum_i\sigma_s\eta_i\otimes({\rm ind}^G_H\pi)_s\xi_i\big)(x),
\end{align*}
for $s,x\in G$.

$(ii)$ This is an immediate consequence of the definition of conjugate representation. 
\end{proof}

Bekka asked in \cite[page 387]{b} if amenability of  ${\rm ind}^G_H\pi$ implies that of $\pi$. Now we know that the answer is negative even for global representations \cite{pes}. The last result of this section gives a reverse implication for partial representations (in terms of the Reiter condition, which is known to be equivalent to amenability for global representations), extending \cite[Corollary 5.6]{b}. When $\sigma$ is a partial representation of $G$ in $E_\sigma$, ${\rm Res}^G_H\sigma$ simply denotes the {\it restriction} of $\pi$ to $H$ in $E_\sigma$.   

\begin{thm}
Let $H$ be a closed subgroup of a locally compact group $G$ and $\pi$ be a partial representation of $H$ in a Hilbert space $E$ satisfying the Urysohn condition, such that $\pi\preceq {\rm Res}^G_H{\rm ind^G_H}\pi$. Assume that  
$G/H$ is amenable in the sense of Eymard. If $\pi$ satisfies Reiter condition $(\tilde P_2)$ then so is ${\rm ind^G_H}\pi$.
\end{thm}
\begin{proof}
By Theorem \ref{wcrei}, $1_H\preceq\pi\otimes\bar\pi$. Also, $1_G\preceq {\rm ind^G_H}1_H$ \cite[page 22]{ey}. By Theorem \ref{ci} an Lemma \ref{resind}$(i)$, 
$$1_G\preceq {\rm ind^G_H}1_H\preceq {\rm ind^G_H}(\pi\otimes\bar\pi)\preceq {\rm ind^G_H}({\rm Res}^G_H{\rm ind^G_H}\pi\otimes\bar\pi)\preceq {\rm ind^G_H}\pi\otimes{\rm ind^G_H}\bar\pi,$$
thus, ${\rm ind^G_H}\pi$ satisfies Reiter condition $(\tilde P_2)$, by Theorem \ref{ci} an Lemma \ref{resind}$(ii)$.  
\end{proof} 
 
\begin{rk}
$(i)$ The condition  $\pi\preceq {\rm Res}^G_H{\rm ind^G_H}\pi$ is not automatic even for global representations, for instance it is known to fail for $G = SL(3, \mathbb C)$, $H = SL(2, \mathbb C)$ and any complementary series representation of $H$ \cite[Theorem 6.1]{f}, but to hold when $H$ is open, normal, or $G$ is  [SIN]$_H$-group \cite[5.3]{h}.

$(ii)$ Since we don't have the continuity of induction for ${\rm Ind}^G_H$, we don't know at this point if, under the assumptions of the above theorem, the  Reiter condition $(\tilde P_2)$ for $\pi$ implies this condition for ${\rm Ind^G_H}\pi$. Also, at this point we don't know if amenability of $\pi$ in the sense of Bekka implies amenability of  ${\rm Ind^G_H}\pi$ or ${\rm Ind^G_H}\pi$.
\end{rk}

\vspace*{1cm}
\end{document}